\documentclass{article}

\RequirePackage[OT1]{fontenc}
\RequirePackage[
amsthm,amsmath,natbib]{imsart}
\usepackage{xcolor}
\usepackage[normalem]{ulem}

\usepackage{graphicx}

\usepackage{latexsym,amsmath}
\usepackage{amsmath,amsthm,amscd}
\usepackage{amsfonts}
\usepackage[psamsfonts]{amssymb}
\usepackage{pstricks,pst-plot}
\usepackage{enumerate}

\usepackage{url}

\usepackage{tocvsec2}

\usepackage[active]{srcltx}

\usepackage{epstopdf}

\startlocaldefs
\numberwithin{equation}{section}

\theoremstyle{plain} 
\newtheorem{theorem}{Theorem}[section]

\newtheorem{lemma}[theorem]{Lemma}
\newtheorem{proposition}[theorem]{Proposition}

\theoremstyle{definition} 
\newtheorem{definition}[theorem]{Definition}
\theoremstyle{definition} 

\newtheorem*{ex*}{Example}
\theoremstyle{remark} 

\theoremstyle{remark} 
\newtheorem{remark}[theorem]{Remark}
\newtheorem*{remark*}{Remark}
\numberwithin{equation}{section}
 
\newcommand{\beqa}{\begin{eqnarray}}
\newcommand{\eeqa}{\end{eqnarray}}
 
\newcommand{\bseq}{\begin{subequations}}
 
\newcommand{\eseq}{\end{subequations}}

\newcommand{\B}{\mathbf{B}}
\renewcommand{\SS}{\mathbf{S}}
\newcommand{\X}{\mathbf{X}}
\newcommand{\Xd}{\mathbf{X}^{(d)}}
\newcommand{\Zd}{\mathbf{Z}}

\newcommand{\N}{\mathbb{N}}

\newcommand{\dd}{\partial}
\newcommand{\0}{\mathbf{0}}
\newcommand{\one}{\mathbf{1}}
\newcommand{\s}{\mathbf{s}}
\newcommand{\xx}{\mathbf{x}}
\newcommand{\uu}{\mathbf{u}}
\newcommand{\vv}{\mathbf{v}}
\newcommand{\ww}{\mathbf{w}}
\newcommand{\yy}{\mathbf{y}}
\newcommand{\zz}{\mathbf{z}}
\newcommand{\ee}{\mathbf{e}}
\newcommand{\ZZ}{\mathbf{Z}}

\renewcommand{\dd}{{\operatorname{d}}}

\newcommand{\sign}{\operatorname{sign}}

\newcommand{\are}{\operatorname{ARE}}

\newcommand{\fb}{\fbox{\rule{0pt}{5pt}\hspace*{-5pt}\rule{0pt}{5pt}}}
\newcommand{\no}[1]{\fb\,#1\,\fb\,}

\newcommand{\OO}{\mathrel{\text{\raisebox{2pt}{$\O$}}}}
\newcommand{\OOG}{\mathrel{\text{\raisebox{2pt}{$\OG$}}}}

\newcommand{\al}{\alpha}
\newcommand{\g}{\gamma}
\newcommand{\Ga}{\Gamma}
\newcommand{\si}{\sigma}
\newcommand{\Si}{\Sigma}
\newcommand{\ka}{\kappa}
\newcommand{\La}{\Lambda}
\newcommand{\la}{\lambda}

\newcommand{\tla}{\tilde\lambda}
\newcommand{\tmu}{\tilde\mu}
\newcommand{\de}{\delta}

\newcommand{\be}{\beta}
\newcommand{\De}{\Delta}
\newcommand{\vpi}{\varphi}
\renewcommand{\th}{{\boldsymbol{\theta}}}
\newcommand{\tth}{\theta}

\newcommand{\ffrown}{\text{\raisebox{3pt}[0pt][0pt]{$\frown$}}}

\renewcommand{\O}{\underset{\ffrown}{<}}
\newcommand{\OG}{\underset{\ffrown}{>}}

\newcommand{\ii}[1]{\,\mathbf{I}\{#1\}} 
\newcommand{\Bigii}[1]{\,\mathbf{I}\Big\{#1\Big\}}

\newcommand{\PP}{\operatorname{\mathsf{P}}} 
\newcommand{\E}{\operatorname{\mathsf{E}}}
\newcommand{\Var}{\operatorname{\mathsf{Var}}}
\newcommand{\Cov}{\operatorname{\mathsf{Cov}}}

\newcommand{\Z}{\mathbb{Z}}
\newcommand{\R}{\mathbb{R}}

\newcommand{\G}{\mathcal{G}}
\newcommand{\OOO}{\mathcal{O}}

\newcommand{\vp}{\varepsilon}

\newcommand{\ta}{{\tilde{a}}}

\newcommand{\tbe}{{\tilde{\beta}}}

\newcommand{\tu}{{\tilde{u}}}
\newcommand{\tr}{{\tilde{r}}}
\newcommand{\tA}{{\tilde{A}}}

\renewcommand{\le}{\leqslant}
\renewcommand{\ge}{\geqslant}

\newcommand{\No}{\operatorname{N}}

\renewcommand{\framebox}[1]{\noindent\textbf{#1}.\quad}

\endlocaldefs

\begin{document}

\begin{frontmatter}

\title{Asymptotic efficiency of $p$-mean tests for means in high dimensions
}
\runtitle{$p$-mean tests for means in high dimensions
}
\date{\today}

\begin{aug}
\author{\fnms{Iosif} \snm{Pinelis}\ead[label=e1]{ipinelis@mtu.edu}}
\runauthor{Iosif Pinelis}

\affiliation{Michigan Technological University}

\address{Department of Mathematical Sciences\\
Michigan Technological University\\
Houghton, Michigan 49931, USA\\
E-mail: \printead[ipinelis@mtu.edu]{e1}
}
\end{aug}

\begin{abstract}
The asymptotic efficiency, $\are_{p,2}$, of the tests for multivariate means $\th\in\R^d$ based on the 
$p$-means 
$
	\no\xx_p
	:=\big(\frac1d\,\sum_{j=1}^d|x_j|^p\big)^{1/p}  
$
($\xx\in\R^d$) relative to the standard $2$-mean, (approximate) likelihood ratio test (LRT), is considered for large dimensions $d$. 
It turns out that these $p$-mean tests for $p>2$ may greatly outperform the LRT while never being significantly worse than the LRT. For instance, $\are_{p,2}$ for $p=3$ varies from $\approx0.96$ to $\infty$, depending on the direction of the alternative mean vector $\th_1$, for the null hypothesis $H_0\colon\th=\0$. 
These results are based on a complete characterization, under certain general and natural conditions, of the varying pairs $(n,\th_1)$ for which the values of the power of the $p$-mean test for $\th=\0$ and $\th=\th_1$ tend, respectively, to prescribed values $\al$ and $\be$ such that $0<\al<\be<1$. 
The proofs use such classic results as the Berry-Esseen bound in the central limit theorem and the conditions of convergence to a given infinitely divisible distribution, as well as a recent result by the author on the Schur${}^2$-concavity properties of Gaussian measures. 
\end{abstract}


%
%
%
%

\end{frontmatter}

\settocdepth{chapter}

\tableofcontents 

\settocdepth{subsubsection}

\theoremstyle{plain} 
\numberwithin{equation}{section}

\section{Introduction}\label{intro} 
Testing for an unknown mean, say $\th$, of a multivariate distribution is a classic statistical problem. It is common to assume that the underlying population distribution is normal. In this introductory section, let us focus on such a ``normal'' case; at that, let us also assume that the population correlation matrix is the identity matrix $I_d$.   

Suppose that one is to test a simple null hypothesis, $H_0\colon\th=\th_0$, versus the complementary alternative $H_1\colon\th\ne\th_0$, where $\th_0$ is a given vector in $\R^d$. Without loss of generality, $\th_0=\0$, the zero vector. Even when the dimension $d$ is $1$, it is well known and easy to see that there is no uniformly most powerful test. 

A common way to deal with this problem is to consider the likelihood ratio test (LRT) as a surrogate of the most powerful Neyman-Pearson test for two simple hypotheses; the LRT will reject $H_0$ if the Euclidean norm $\|\overline\X\|$ of the sample mean $\overline\X$ is greater than a critical value $c$. It is easy to see that the LRT will be uniformly most powerful among all spherically invariant tests (that is, among all the tests invariant with respect to the group $\OOO_d$ of all orthogonal linear transformations of $\R^d$); indeed, using the well-known representation of the non-central $\chi^2$ distribution as a mixture of central ones, with greater numbers of degrees of freedom, one can immediately verify the monotonicity of the ratio of the density of $\|\overline\X\|$ for any given $\th\ne\0$ to that density for $\th=\0$. 

Clearly, spherical invariance is a very strong restriction on the test, as it takes all the directions in $\R^d$ as equally important or as equally unknown. In many situations, typically when the dimension $d$ is large, the unknown mean vector $\th$ may have only comparatively few large, dominant (in absolute value) coordinates $\tth_j$. For example, in a Fourier basis only a few coordinates may be of significance; in other words, only a few harmonics in a Fourier decomposition of the unknown ``signal'' $\th$ may have significantly large amplitudes. 
So, one may want to consider tests which perform especially well for ``un-equalized'' directions of the vector $\th$, with comparatively few dominant $|\tth|_j$'s -- as opposed to ``equalized'' directions, with all  $|\tth|_j$'s more or less of the same order of magnitude. 

When it is not known which few of the many coordinates $\tth_j$ are dominant, one may try to approximate the multivariate testing problem by one of a much smaller dimension. Otherwise, when it is known which ones of the $\tth_j$'s may be dominant, 
one may want to consider tests that are invariant with respect to groups much smaller than the orthogonal group $\OOO_d$; cf.\ e.g.\ \cite{eaton-perl} and references therein. One such natural group is the one generated by all permutations of the coordinates $\tth_j$ and the $d$ reflections $\tth_j\mapsto-\tth_j$, changing just the sign of one of the coordinates of $\th$ while leaving the other coordinates intact; let us denote this group of transformations by $\G_d$; clearly, it is a subgroup of $\OOO_d$. 

A natural family of $\G_d$-invariant tests, depending on the parameter $p$, consists of the tests that reject $H_0$ when the (absolute) $p$-mean $\no{\overline\X}_p$ of $\overline\X$ exceeds a critical value $c$, where the $p$-mean of a vector $\th$ is defined as $\big(\frac1d\,\sum_{j=1}^d|\tth_j|^p\big)^{1/p}$. For $p=2$, the $p$-mean test is equivalent to the mentioned LRT. This definition of the $p$-mean of a vector can be extended by continuity to all $p\in[-\infty,\infty]$, so that $\no\th_{-\infty}=\min_j|\tth_j|$, $\no\th_0=\prod_j|\tth_j|^{1/d}$ (the geometric mean of the $|\tth_j|$'s), and $\no\th_\infty=\max_j|\tth_j|$. 

In this paper, we shall consider asymptotic efficiency of the $p$-mean tests and, in particular, 
the Pitman asymptotic relative efficiency, $\are_{p,2}$, of the $p$-mean test relative to the LRT, that is, relative to the $2$-mean test; the value of $\are_{p,2}$ depends on the prescribed (possibly asymptotic) values, $\al$ and $\be$ ($0<\al<\be<1$), of the power function at the null value $\th=\0$ and at the alternative value of $\th$, respectively. 
Clearly, $\are_{p,2}$ may also depend on the direction of the alternative vector $\th$; so, on most occasions we shall write $\are_{p,2,\uu}$ in place of $\are_{p,2}$, where $\uu$ is the $\no\cdot_2$-unit vector in the direction of $\th$.  
It is also clear that $\are_{p,2,\uu}$ is $\G_p$-invariant, that is, $\are_{p,2,\uu}$ is invariant with respect to all permutations of the coordinates $u_j$ of $\uu$ as well as to all sign changes of the $u_j$'s. 

A further and much less trivial property of $\are_{p,2,\uu}$ is what we shall refer to as the Schur${}^2$-convexity/concavity; namely, $\are_{p,2,\uu}$ is Schur-convex in $(u_1^2,\dots,u_d^2)$ 
for each $p\in[2,\infty]$ and Schur-concave in $(u_1^2,\dots,u_d^2)$ 
for each $p\in[-\infty,2]$; see \cite
{schur2} and Theorem~\ref{th:are schur2} in the present paper. 
Informally, the Schur${}^2$-convexity/concavity means that, for each $p\in[2,\infty]$, the more ``unequalized'' vector $\uu$ one takes the greater $\are_{p,2,\uu}$ is; and for each $p\in[-\infty,2]$ this relation is reversed. 
Therefore, for each $p\in[-\infty,\infty]$, the value of $\are_{p,2,\uu}$ for any $\no\cdot_2$-unit vector $\uu$ lies between the values of $\are_{p,2,\uu}$ for the ``completely equalized'' $\no\cdot_2$-unit vector $\one:=(1,\dots,1)\in\R^d$ and that for the ``maximally unequalized'' $\no\cdot_2$-unit vector $\sqrt{d}\,\ee_1$, where 
$\ee_1:=(1,0,\dots,0)\in\R^d$. 

Of course, for $d=1$ the $p$-mean is the same for all values of $p$. So, the least nontrivial dimension is $d=2$, in which case the possible values of $\are_{p,2,\uu}$ for $\al=0.05$ and $\be=0.95$ are shown in the left half of Figure~\ref{fig:are}. 

\begin{figure}[htbp]
	\centering
	\includegraphics[scale=.65]{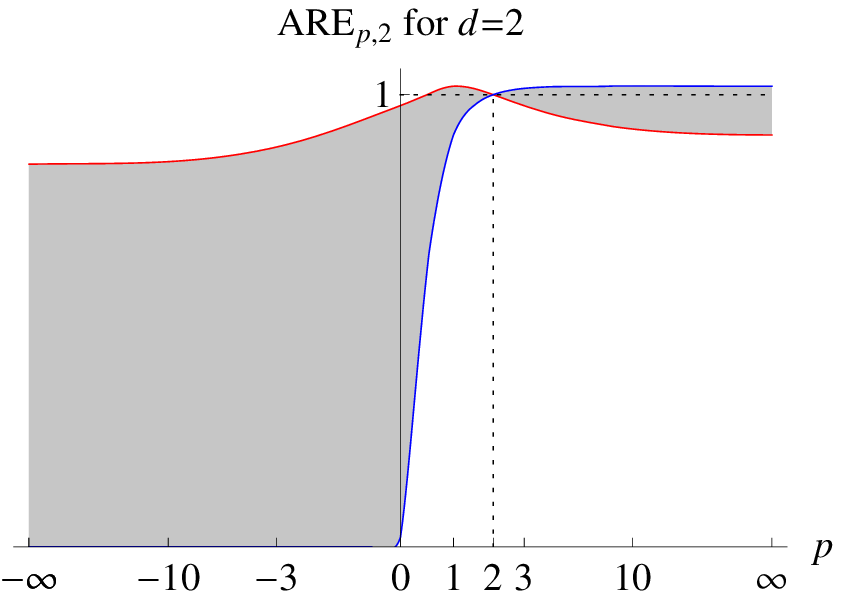}
	\hspace*{.4cm}
	\raisebox{-4pt}{
	\includegraphics[scale=.65]{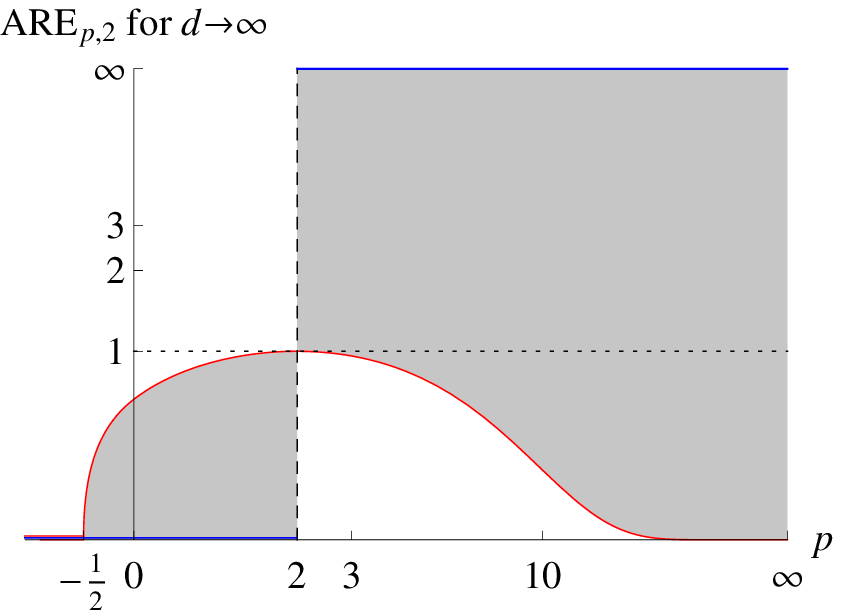}
	}
	\caption{On the left: the values of $\are_{p,2,\uu}$ for $d=2$, $\al=0.05$, $\be=0.95$; 
	on the right: the values of $\are_{p,2,\uu}$ for $d\to\infty$; 
	the horizontal and 
	vertical scales here are nonlinear, with points $(p,\are_{p,2,\uu})$ represented by points  $\big(\psi(p/4),\psi(\are_{p,2,\uu})\big)$ (shaded), where $\psi(x):=2x/(2|x|+3)$, so that $\psi(x)$ increases from $-1$ to $0$ to $1$ as $x$ increases from $-\infty$ to $0$ to $\infty$. 
The values of $\are_{p,2,\uu}$ for the ``completely equalized'' directions $\uu$ are represented by the red curve; for the ``maximally unequalized'' $\uu$, by the blue curve for $d=2$ and by the two horizontal blue lines for $d\to\infty$.}
	\label{fig:are}
\end{figure}

Note that for $d=2$ and any given nonzero vector $\uu\in\R^2$ one has $\are_{\infty,2,\uu}=\are_{1,2,R^{\pi/4}\uu}$, where $R^{\pi/4}$ is the operator of rotation through the angle $\pi/4$; this follows because  $B_1^2=\sqrt2\,R^{\pi/4}B_\infty^2$, where $B_p^2$ stands for the $\no\cdot_p$-unit ball in $\R^2$; however, this symmetry appears to be lost for any $d>2$. 
Going back to $d=2$, one has 
$\are_{1,2,(1,1)}=\are_{\infty,2,(\sqrt2,0)}=1.0317\dots$, again for $\al=0.05$ and $\be=0.95$. 
It thus appears that, for such $d$, $\al$, and $\be$, the $p$-mean test can at best outperform the LRT by about 3.2\%, which happens for $p=\infty$ and ``maximally unequalized'' directions as well as for $p=1$ and ``completely equalized'' directions. 

One may further ask in which directions $\uu$ the $p$-mean test outperforms the LRT test, in the sense that $\are_{p,2,\uu}>1$. It appears (at least for $d=2$) that for each $\uu$ there is some $p$ such that $\are_{p,2,\uu}>1$. 
Indeed, the left half of Figure~\ref{fig:d=2} suggests that (again for $d=2$, $\al=0.05$, and $\be=0.95$) in nearly all directions $\uu$ either $\are_{2.1,2,\uu}>1$ or $\are_{1.9,2,\uu}>1$. 

\begin{figure}[htbp]
	\centering
\includegraphics[scale=0.5]
{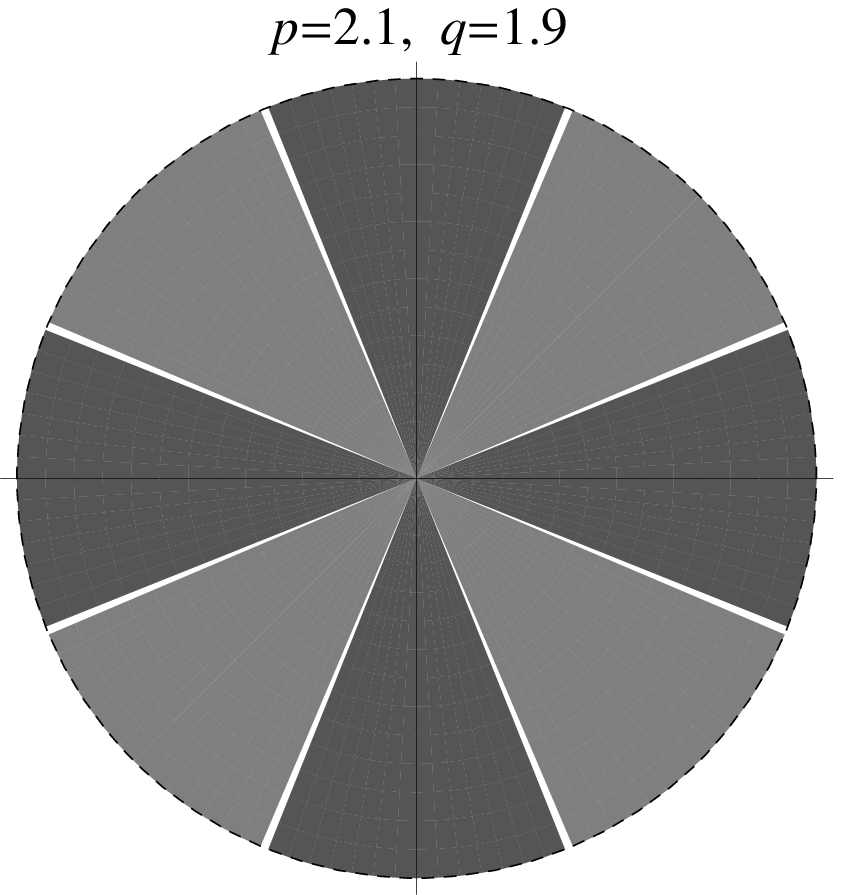}
\hspace*{1cm}
\includegraphics[scale=0.5]
{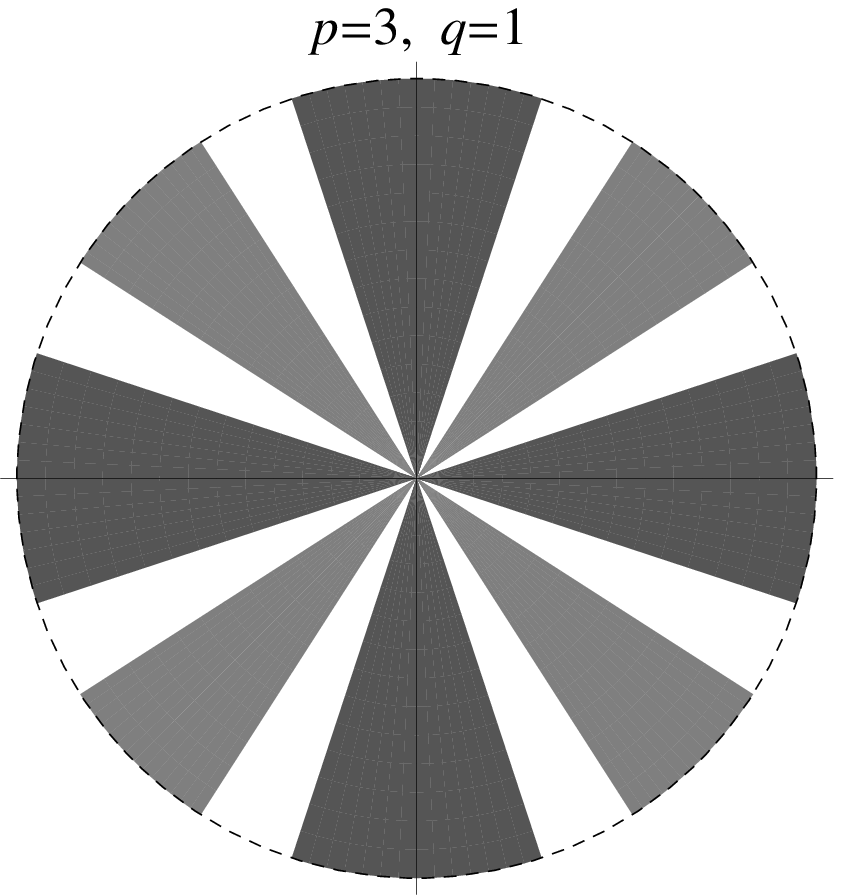}
	\caption{The sectors of directions $\uu$ for $d=2$ where $\are_{p,2,\uu}>1$ (dark\-er gray) 
	and $\are_{q,2,\uu}>1$ (lighter gray); in each of the 8 narrow white sectors, one has $\are_{p,2,\uu}\le1$ and $\are_{q,2,\uu}\le1$. }
	\label{fig:d=2}
\end{figure}

However, at that the improvement in performance is small: $\are_{2.1,2,(\sqrt2,0)}=1.00429\dots$ and  $\are_{1.9,2,(1,1)}=1.00459\dots$. 
Recall also that the maximum improvement \big(for $d=2$, $\al=0.05$, and $\be=0.95$, over all directions $\uu$ and all $p$\big) of the $p$-mean tests over the LRT test appears to be less than 3.2\%. 

The matter is quite different for large $d$, as shown in the right half of Figure~\ref{fig:are}. 
As one can see, for each $p\in[-\infty,2)$ all possible values of $\are_{p,2,\uu}$ \big(for $d\to\infty$, when the dependence of $\are_{p,2,\uu}$ on $\al$ and $\be$ disappears\big) are less than $1$, for all directions $\uu$ of the alternative vector $\th$; that is, for any $p\in[-\infty,2)$ the $2$-mean, LRT test is always asymptotically better than the $p$-mean test.  
However, for each $p\in(2,\infty]$ values of $\are_{p,2,\uu}$ can be arbitrarily large if the vector $\uu$ is sufficiently unequalized. For instance, for $p=3$ the possible values of $\are_{p,2,\uu}$ range from $\approx0.96$ to $\infty$.   
This suggests that (say) the $3$-mean test should generally be preferred to the $2$-mean test -- especially, when the direction of the alternative vector $\th$ is unknown or, even more so, when the direction of $\th$ is known to be far from equalized. One can say that the $3$-mean test is much more robust than the LRT with respect to a few large coordinates of an alternative mean vector $\th$, while the asymptotic efficiency of the $3$-mean test relative to the LRT never falls below $\approx96\%$. 
Theorem~\ref{th:are} shows that for each $\in(2,\infty)$ there is a ``phase transition'' phenomenon for $\are_{p,2,\uu}$ depending on the direction of $\uu$; namely, $\are_{p,2,\uu}$ equals either $\infty$ or a certain finite positive number $a_p$ depending on whether the $p$-mean $\no{\uu}_p$ is much greater or much less than $d^{(p-2)/(4p)}$; a similar ``phase transition'' takes placed for $p\in(-\frac12,2)$. 
As seen from the right panel of Figure~\ref{fig:are}, another ``phase transition'', in $p$, occurs at $p=2$; it would be interesting to study this latter phenomenon further, letting $p$ go to $2$ while $d$ tends to $\infty$. 

In contrast with the just discussed behavior of the $\are$ with fixed $p,\al,\be$ and $d\to\infty$, it was shown in \cite{schur2} that $\are_{p,2,\uu}$ does not exceed $1$ for any fixed $d,p,\uu$ when $\al\to0,\be\to1$ -- that is, when both types of error probabilities tend to $0$. Thus, it would also be of interest to study the $\are$ when $d\to\infty$ \emph{and} $\al\to0,\be\to1$.
 

\bigskip

In this paper, we shall assume that  
$d\to\infty$, which represents an interesting case and allows for an informative theory. 
A simple but key observation is that, say for $p\in(0,\infty)$, one has 
$
	\PP_\th(\no{\overline\X}_p>c)=\PP_\th(\sum_1^d|Z_j+\tth_j\sqrt n|^p>c^p),
$ 
where the $Z_j$'s are independent standard normal random variables (r.v.'s). 
so that one has to deal with the distribution of the sum $\sum_1^d|Z_j+\tth_j\sqrt n|^p$ of independent r.v.'s $|Z_j+\tth_j\sqrt n|^p$. If, after appropriate re-centering and/or re-scaling, these r.v.'s are uniformly asymptotically negligible and also 
satisfy certain other conditions, then, by a well-known theorem, the distribution of the sum $\sum_1^d|Z_j+\tth_j\sqrt n|^p$ can be approximated by an infinitely divisible distribution. 
Fortunately, it turns out that the just mentioned conditions for such an approximation to hold will be satisfied whenever $c$ and $\th$ vary (with $d\to\infty$) in such a way that the conditions 
$\PP_\0(\no{\overline\X}_p>c)\to\al$ and $\PP_\th(\no{\overline\X}_p>c)\to\be$ on the power function of the test are fulfilled; the proof of this relies in part on the fact that the probability $\PP_{t\uu}(\no{\overline\X}_p>c)$ is increasing in $t\ge0$ for each $\uu$.  
Moreover, for $p\in(0,\infty)$ the limit infinitely divisible distribution is simply normal. The cases $p=0$ and $p\in(-\infty,0)$ can be treated similarly. However, for each $p\in(-\infty,-\frac12)$ the limit distribution is, not normal, but a stable distribution with exponent $-\frac1p\in(0,2)$. 
As for the cases $p=\pm\infty$, they are in a sense more elementary, because then the distribution of $\no{\overline\X}_p$ can be easily expressed in terms of the standard normal distribution. 

The assumption of normality of the population distribution was made in this introductory section only to simply the discussion. In fact, in the paper we shall consider a more general case of independent identically distributed observations $\X_i$; moreover, the results will hold for non-identically distributed observations, as long as one has a convergence to normality, in a certain sense. 

The paper is organized as follows. In Section~\ref{results}, we state the necessary definitions and the main theorems, as well as a number of propositions, which complement the definitions and theorems presented there. 
In Section~\ref{proofs}, we state a number of key propositions, from which the theorems stated in Section~\ref{results} follow; these propositions may be of independent interest. 
The proofs of the propositions stated in Section~\ref{proofs} are deferred to Sections~\ref{proof of shifts} and \ref{proofs of props p}, which are preceded by Section~\ref{lemmas}, where a number of lemmas are stated, to be used in Sections~\ref{proof of shifts} and \ref{proofs of props p}. Finally, in Section~\ref{proofs of lemmas}, we prove the mentioned lemmas. 
This natural tree-like structure allows us to present the ideas, big and small, each in its own place. 
Note also that the propositions in Section~\ref{results} complement the definitions and theorems presented there, while the propositions in Section~\ref{proofs} are what these theorems immediately follow from.

\section{Statements  and discussion of results}\label{results}

\subsection{Observation data and hypotheses}\label{obsrvs,hyp} 
For each natural $d$, let $\Xd,\Xd_1,\Xd_2,\dots$ be independent identically distributed random vectors in $\R^d$, with 
a distribution indexed by 
the unknown mean vector $\th$ of $\Xd$;  
let $\E_\th$ and $\PP_\th$ denote the corresponding expectation and probability functionals, so that $\E_\th\Xd=\th$ for all $\th\in\R^d$. 
Suppose that for any given $d$ and $\th$ the covariance matrix 
\begin{equation*}
\Si_d(\th):=
\Cov_\th\Xd	
\end{equation*}
is finite, nonsingular, and continuous in $\th$ in a neighborhood $\mathcal{V}_d$ of the zero vector $\0=\0_d$ in $\R^d$. Suppose also that   
\begin{equation}\label{eq:rho}
\rho_3(d):=\sup_{\th\in\mathcal{V}_d}\E_\th\|\Si_d(\th)^{-1/2}(\Xd-\th)\|^3<\infty; 	
\end{equation}
here and in what follows, $\|\cdot\|$ denotes the usual Euclidean norm in $\R^d$, so that $\|\vv\|=\sqrt{\sum v_j^2}$ for any $\vv\in\R^d$; unless specified otherwise, the summation sign $\sum$ will stand for $\sum_{j=1}^d$; for vectors $\s,\uu,\vv,\dots$ in $\R^d$, we let $s_j,u_j,v_j,\dots$ ($j=1,\dots,d$) denote their respective coordinates. 

Consider testing the hypothesis $H_0\colon\th=\0$ versus the alternative $H_1\colon\th\ne\0$, 
based on the statistic $\Si_d(\0)^{-1/2}\overline{\Xd_n}$, where $\overline{\Xd_n}:=\frac1n\sum_{i=1}^n\Xd_i$. 
At that, $\Si_d(\0)$ is supposed to be known. 
Moreover, let us assume that 
\begin{equation}\label{eq:I}
	\Si_d(\0)=I_d,
\end{equation}
the $d\times d$ identity matrix; this assumption does not diminish generality, since one may replace $\Xd_i$ by $\Si_d(\0)^{-1/2}\Xd_i$. 

In fact, the results will hold for non-identically distributed observations, as long as one has a convergence to normality; cf.\ Subsection~\ref{NA} below. For very large $n$, one may also remove the requirement that the covariance matrix $\Si_d(\th)$ be known; cf.\ \cite{bhat,nonlinear}.  

\subsection{$p$-means tests}\label{means tests}

Consider then tests of the form   
\begin{equation}\label{eq:test}
	\de_{n,p}:=\de_{n,p,c}:=\Bigii{\sqrt{n}\,\no{\overline{\Xd_n}}_p>c}, 
\end{equation}
where $\ii\cdot$ denotes the indicator function, 
$c\in\R$, $p\in[-\infty,\infty]$, and $\no{\s}_p$ is the $p$-mean of a vector $\s\in\R^d$ defined as follows.  
For any $p\in(-\infty,\infty)\setminus\{0\}$, let 
\begin{equation}\label{eq:np_p}
	\no\s_p
	:=\Big(\frac1d\,\sum\limits_{j=1}^d|s_j|^p\Big)^{1/p}; 
\end{equation} 
at that, if $p\in(-\infty,0)$, use the continuity conventions $0^p:=\infty$ and $\infty^{1/p}:=0$, so that $\no\s_p=0$ if $p\in(-\infty,0)$ and at least one of the $s_j$'s is $0$. 
As usual, extend definition \eqref{eq:np_p} by continuity to all $p\in[-\infty,\infty]$, so that 
\begin{equation}\label{eq:means excep}
	\no\s_{-\infty}=\min_{j=1}^d|s_j|,\quad 
	\no\s_0=\prod_{j=1}^d|s_j|^{1/d},\quad 
	\no\s_\infty=\max_{j=1}^d|s_j|.  
\end{equation}
For $p\in[1,\infty]$, the function $\no\cdot_p$ is a norm, which differs from the more usual $p$-norm $\|\cdot\|_p$ by the factor $\frac1{d^{1/p}}$; this factor is needed in order that $\no\s_p\to\no\s_0$ as $p\to0$. 

For any $p\in[-\infty,\infty]$, let us say that a vector $\uu\in\R^d$ is \emph{$\no\cdot_2$-unit} if $\no\uu_p=1$. Note that the vector $\one:=\one_d:=(1,\dots,1)\in\R^d$ is $\no\cdot_2$-unit for every $p\in[-\infty,\infty]$, and so are all vectors $\uu\in\{-1,1\}^d$. 
Let us refer to any vector in the direction of a vector $\uu\in\{-1,1\}^d$ as \emph{(perfectly) equalized}. 
Informally, let us say that a vector $\vv\in\R^d$ is \emph{(well enough) equalized} if the $|v_j|$'s are mainly of the same order of magnitude; otherwise -- if there are comparatively few (as compared with $d$) dominating $|v_j|$'s, let us say that the vector $\vv$ is \emph{(too) unequalized}.   

\subsection{Normal admissibility (NA)}\label{NA}
Let $\th_1$ denote a nonzero vector in $\R^d$, so that $\th_1$ represents a possible mean vector compatible with the alternative hypothesis $H_1$. 
Note that $\th_1$ must be varying with $d$, even if for no reason other than that $\th_1$ is in $\R^d$. The dimension $d$ may be considered an attribute (that is, a function) of $\th_1$, and thus one may write
$
	d=\dim\th_1. 
$ 
The sample size $n$ and the critical value $c$ will also be allowed to vary, all together with $d$ and $\th_1$, whereby one has a triple $(n,\th_1,c)$, varying with $d$.

\begin{definition}\label{def:NA}
For any given $p\in[-\infty,\infty]$, let us say that a varying \emph{pair} $(n,\th_1)$ is \emph{$p$-normally admissible} ($p$-NA) if 
\begin{equation*}
	d=\dim\th_1\to\infty
\end{equation*}
and $\overline{\Xd_n}$ is asymptotically normal in the sense that 
\begin{equation}\label{eq:NA}
	\sup_{c\in\R}\Big|\PP_\th\Big(\no{\sqrt n\,\overline{\Xd_n}}_p\le c\Big) 
	-\PP\big(\,\no{\Zd+\sqrt n\,\th}_p\le c\big)\Big|\longrightarrow0
\end{equation}
for $\th=\0$ and for $\th=\th_1$, where  
\begin{equation}\label{eq:Zd}
	\Zd=\mathbf{Z}^{(d)}=(Z_1,\dots,Z_d)\sim\No(\0,I_d),
\end{equation}
a standard normal random vector in $\R^d$. 

Further, let us say that a varying \emph{triple} $(n_p,n_2,\th_1)$ is $(p,2)$-NA if the pair $(n_p,\th_1)$ is $p$-NA and the pair $(n_2,\th_1)$ is $2$-NA.  
\end{definition}

The following proposition implies that, for a varying pair $(n,\th_1)$ to be $p$-NA, 
usually it is sufficient that $n\to\infty$ fast enough as $d$ grows to $\infty$ while $\th_1$ stays close enough to $\0$. 

\begin{proposition}\label{prop:NA}
For each $p\in[-\infty,\infty]$, there exist some functions $\N\ni d\mapsto n_p(d)$ and $\N\ni d\mapsto\tth_p(d)$ \big(which depend on the functions $d\mapsto\mathcal{V}_d$, $(d,\th)\mapsto\Si_d(\th)$, and $d\mapsto\rho_3(d)$\big), such that 
(i) $\tth_p(d)>0$ for all natural $d$ and (ii) all varying pairs $(n,\th_1)$ with 
$d=\dim\th_1\to\infty$, $n\ge n_p(d)$, and $\|\th_1\|\le\tth_p(d)$ are $p$-NA. 
\end{proposition}

The necessary proofs are deferred to Sections~\ref{proofs}--\ref{proofs of lemmas}. 

The condition that the random vectors $\Xd_i$ be i.i.d.\ is assumed in this paper only to simplify the presentation. Indeed, this condition is only used in the proof of Proposition~\ref{prop:NA}, which is based on a Berry-Esseen type bound \cite[Corollary~15.3]{bhat}, which allows the distributions of the random vectors $\Xd_i$ to differ to a certain extent.

\subsection{Asymptotically sufficient (AS) pairs and triples}\label{AS}

We shall consider the asymptotic efficiency of the tests $\de_{n,p}$ for all values of $p\in[-\infty,\infty]$ relative to $\de_{n,2}$, which is the likelihood ratio text (LRT) in the case when the observations $\Xd_i$ are $d$-variate normal. 
We shall be working under the assumption that 
\begin{equation}\label{eq:d to infty}
d\to\infty. 	
\end{equation}
Take any $\al$ and $\be$ such that
\begin{equation}\label{eq:al,be}
	0<\al<\be<1. 
\end{equation}
Unless specified otherwise, in what follows we shall always assume conditions \eqref{eq:d to infty} and \eqref{eq:al,be} to hold and consider $\al$ and $\be$ to be fixed. 
These values, $\al$ and $\be$, will be the (approximate) target values of the power function of the tests $\de_{n,p}$ at the values of $\th$ equal $\th_0:=\0$ and $\th_1\ne\0$, respectively. 
Given $p$, $\al$, $\be$, and $\th_1$, the efficiency of the test can be measured, as usual, by the necessary sample size $n$. Thus, one comes to the notion of a $p$-asymptotically sufficient pair $(n,\th_1)$. Moreover, when comparing the tests $\de_{n,p}$ to the tests $\de_{n,2}$ for the same alternative mean vector $\th_1$, it is natural to combine a $p$-asymptotically sufficient pair $(n_p,\th_1)$ and a  $2$-asymptotically sufficient pair $(n_2,\th_1)$ into a $p$-asymptotically sufficient triple $(n_p,n_2,\th_1)$, as was done in Definition~\ref{def:NA} for $p$-NA pairs and triples.

\begin{definition}\label{def:AS}
Take any $p\in[-\infty,\infty]$. 
Say that a varying pair $(n,\th_1)$ is \emph{$p$-weakly asymptotically sufficient} ($p$-weakly-AS) if, for \emph{some} varying (with $d$, $n$, and $\th_1$) critical value $c$, one has 
\begin{equation}\label{eq:AS}
	\E_{\0}\de_{n,p,c}\to\al\quad\text{and}\quad \E_{\th_1}\de_{n,p,c}\to\be.
\end{equation} 
Next, say that $(n,\th_1)$ is $p$-strongly-AS if, for \emph{any} varying $c$ such that $\E_{\0}\de_{n,p,c}\to\al$ one has $\E_{\th_1}\de_{n,p,c}\to\be$. 
Further, say that a varying triple $(n_p,n_2,\th_1)$ is $(p,2)$-weakly-AS if the pair $(n_2,\th_1)$ is $2$-weakly-AS and the pair $(n_p,\th_1)$ is $p$-weakly-AS; 
similarly defined is a $(p,2)$-strongly-AS triple $(n_p,n_2,\th_1)$. 
\end{definition}

An important fact is given by 

\begin{proposition}\label{prop:AS}
Any $p$-NA pair $(n,\th_1)$ is $p$-strongly-AS if and only if it is $p$-weakly-AS. 
Therefore, any $(p,2)$-NA triple $(n_p,n_2,\th_1)$ is $(p,2)$-strongly-AS if and only if it is $(p,2)$-weakly-AS. 
\end{proposition}

By virtue of this proposition, which appears nontrivial, we shall be simply referring to $p$-NAAS pairs \big(that is, pairs $(n,\th_1)$ that are $p$-NA and $p$-AS\big), as well as to $(p,2)$-NAAS triples, at that omitting the adjectives ``weakly'' and ``strongly''. 

\subsection{Description of $p$-NAAS pairs and triples}\label{NAAS}

We shall provide an explicit and complete characterization of all the $p$-NA pairs $(n,\th_1)$ that are $p$-AS. Toward this end, we need to introduce the terms in which such pairs will be described. 

First here, recall that a probability distribution on $\R$ is infinitely divisible if and only if 
its characteristic function is of the form 
\begin{equation}\label{eq:levy}
u\longmapsto\exp\Big(ibu-\tfrac12 au^2+\int_{\R}\big(e^{iux}-1-iux\ii{|x|\le1}\big)\,\nu(\dd x)\Big),
\end{equation}
where $b\in\R$, $a\in[0,\infty)$, and $\nu$ is a nonnegative Borel measure on $\R$ such that $\nu(\{0\})=0$ and $\int_{\R}(1\wedge x^2)\,\nu(\dd x)<\infty$. The measure $\nu$ is called the L\'evy measure and the objects $a,b,\nu$ are called the characteristics of the infinitely divisible distribution or, equivalently, of any r.v.\ with this distribution. See e.g.\ \cite[Corollary~15.8]{kallenberg}. 

\begin{definition}\label{def:zeta}
For each $p\in(-\infty,-\frac12)$, let $\zeta_{p,b}$ denote any r.v.\ with the infinitely divisible distribution with characteristics $a=0$, $b\in\R$, and $\nu=
\sqrt{\frac2\pi}\,\nu_p$, 
where 
\begin{equation*}
	\nu_p(\dd x):=-\tfrac1p\,x^{-1+1/p}\,\ii{x>0}\,\dd x. 
\end{equation*}
\end{definition}

\begin{remark*}
For each $p\in(-\infty,-\frac12)$, such a r.v.\ $\zeta_{p,b}$ is stable with index $-\frac1p\in(0,2)$; see e.g.\ \cite[Theorem~2.2.1]{ibr-lin}. In particular, it follows that the support of the distribution of $\zeta_{p,b}$ is connected; see e.g.\ \cite[Theorem~2.3.1, and Remark~1 on page 49]{ibr-lin}. 
So, the distribution function (d.f.) -- which we shall denote by $\Phi_{p,b}$ -- of the r.v.\ $\zeta_{p,b}$, is continuous and strictly increasing from $0$ to $1$ on the connected support of the distribution of $\zeta_{p,b}$. Thus, the quantiles $\Phi_{p,b}^{-1}(\al)$ and $\Phi_{p,b}^{-1}(\be)$ are well, and uniquely, defined. 
\qed\end{remark*}

Let 
\begin{equation*}
	Z\sim\No(0,1);
\end{equation*}
as usual, let $\Phi$ and $\vpi$ denote, respectively, the d.f.\ and the density function of $\No(0,1)$. 
Further, consider the moments  
\begin{equation}\label{eq:la_p}
\begin{gathered}
\la_p(s):=\E|Z+s|^p,\quad 
\la_{p,m}(s):=\E\big||Z+s|^p-\la_p(s)\big|^m; 
\end{gathered}
\end{equation}
note that $\la_p(s)<\infty$ for all $p\in(-1,\infty)$ and $s\in\R$, and $\la_{p,m}(s)<\infty$ for all $m>0$, $p\in(-1/m,\infty)$, and $s\in\R$. 
Also let 
\begin{equation*}
\begin{gathered}
\tla(s):=\E\ln|Z+s|, \quad
\tla_m(s):=\E\big|\ln|Z+s|-\tla(s)\big|^m, 
\end{gathered}
\end{equation*}
\begin{equation}\label{eq:tmu}
	\tmu_d(s):=\E\big(|Z+s|^{-1}\wedge d\big), 
\end{equation}
\begin{equation}\label{eq:la infty}
	\la_\infty(s):=\la_{\infty;d}(s):=\la_{\infty;d,\al}(s):=-\ln\PP(|Z+s|\le c_{d,\al}), 
\end{equation}
\begin{equation}\label{eq:c(d,al)}
	c_{d,\al}:=\sqrt{2\ln\Big(-\frac d{\sqrt{\pi\ln d}}\,\frac1{\ln(1-\al)}\Big)}\ . 
\end{equation}

For arbitrary expressions $a$ and $b$ which may depend on $p$, $d$, the distribution of $\X_d$, and other variables, notation $a=O(b)$ will mean that $b>0$ and $\limsup\frac{|a|}b<\infty$; alternatively, in other contexts $a=O(b)$ may also mean $|a|\le Cb$ for some positive real $C$ that is constant over the specified tuple of variables. 
Oftentimes, it will be more convenient for us to write the relation $a=O(b)$ in a parentheses-free form, as $a\OO b$ or, equivalently, $b\OOG a$.  
We shall also use the following notations: 
$a\approx b$ if $b=a+o(1)$; $a\sim b$ if $b=a(1+o(1))$; $a<<b$ or, equivalently, $b>>a$ if $|a|=o(b)$. The term ``eventually'' will, informally, mean ``when the limit transition process is close enough to the limit''; formally, the limit transition process can for instance be represented by a filter \cite{bourbaki} of subsets of the set of all the tuples; the term ``eventually'' can be then understood as ``for all the tuples in some set belonging to the filter''.  
In a somewhat restricted sense, one may think of all the variables under considerations as indexed by the dimension $d$, with $d\to\infty$; then ``eventually'' can be understood simply as ``for all large enough $d$''. 
\big(It is hoped that there will not be any confusion with the other kind of use of the symbol $\sim$, as e.g.\ in \eqref{eq:Zd}, where it means ``has the following distribution''.\big)   

Now we are ready to state the central result of this subsection and, perhaps, the entire paper. 

\begin{theorem}\label{th:AS iff}
For each $p\in[-\infty,\infty]$, a varying $p$-NA pair $(n,\th_1)$ is $p$-AS if and only if the vector $\s:=\sqrt{n}\th_1$ satisfies the relation 
\begin{equation}\label{eq:as iff}
	\sum_{j=1}^d f_p(s_j)\sim K_{\al,\be;p}\,\ka_p(d),
\end{equation}
where $f_p(s)$, $\ka_p(d)$, and $K_{\al,\be;p}$ are given, depending on $p$, by the table 
\begin{center}
  \begin{tabular*}{.999\textwidth}{@{\hspace*{4pt}} l @{\hspace*{3pt}}
  || @{\hspace*{4pt}}l @{\hspace*{3pt}}
  |@{\hspace*{4pt}}l@{\hspace*{3pt}}
  | @{\hspace*{4pt}}l @{\hspace*{3pt}} }
     $p$ & $f_p(s)$ & $\ka_p(d)$ & $K_{\al,\be;p}$ \\ \hline
        \hline
$\rule{0pt}{12pt}=-\infty$ & $e^{-s^2/2}$ & $d$ & $\ln\be\,/\,\ln\al$ \\ \hline

$\rule[-8pt]{0pt}{20pt}\in(-\infty,-1)$ & $e^{-s^2/2}$ & $d$ & $[\Phi_{p,b_p}^{-1}(\be)/\Phi_{p,b_p}^{-1}(\al)]^{1/p}$ \\ \hline

$\rule[-10pt]{0pt}{22pt}=-1$ & $\tmu_d(0)-\tmu_d(s)$ & $d$ & $\Phi_{-1,-\sqrt{2/\pi}}^{-1}(\be)-\Phi_{-1,-\sqrt{2/\pi}}^{-1}(\al)$ \\ \hline

$\rule[-6pt]{0pt}{18pt}\in(-1,-\tfrac12)$ & $\la_p(0)-\la_p(s)$ & $d^{|p|}$ & $\Phi_{p,b_p}^{-1}(\be)-\Phi_{p,b_p}^{-1}(\al)$ \\ \hline

$\rule[-6pt]{0pt}{18pt}=-\tfrac12$ & $\la_{-\frac12}(0)-\la_{-\frac12}(s)$ & $(d\,\ln d)^{1/2}$ & $[\Phi^{-1}(\be)-\Phi^{-1}(\al)]\,(2/\pi)^{1/4}$ \\ \hline

$\rule{0pt}{12pt}\in(-\tfrac12,0)$ & $\la_p(0)-\la_p(s)$ & $d^{1/2}$ & $[\Phi^{-1}(\be)-\Phi^{-1}(\al)]\,\sqrt{\la_{p,2}(0)}$ \\ \hline

$\rule[-6pt]{0pt}{22pt}=0$ & $\tla(s)-\tla(0)$ & $d^{1/2}$ & $[\Phi^{-1}(\be)-\Phi^{-1}(\al)]\,\sqrt{\tla_2(0)}$ \\ \hline

$\rule{0pt}{12pt}\in(0,\infty)$ & $\la_p(s)-\la_p(0)$ & $d^{1/2}$ & $[\Phi^{-1}(\be)-\Phi^{-1}(\al)]\,\sqrt{\la_{p,2}(0)}$ \\ \hline

$\rule{0pt}{12pt}=\infty$ & $\la_\infty(s)-\la_\infty(0)$ & $1$ & $\ln(1-\al)-\ln(1-\be)$ \\ \hline

  \end{tabular*}
\end{center}
with 
\begin{equation}\label{eq:b}
b_p:=\left\{
\begin{alignedat}{2}
&-\frac{\sqrt{2/\pi}}{p+1} &&\text{ if $p\in(-\infty,-\tfrac12)\setminus\{-1\}$}, \\
&-\sqrt{2/\pi} &&\text{ if $p=-1$}. 
\end{alignedat}
\right.
\end{equation}
\end{theorem}
Note that $f_p(s)$ depends on $d$ only if $p=-1$ or $p=\infty$ (in which latter case, it depends also on $\al$). 

The next proposition addresses the question of existence of $p$-NAAS triples. 

\begin{proposition}\label{prop:AS exist}
\ 
\begin{enumerate}[(I)] 
	\item 
	For any $p\in[-\infty,\infty]$, there exists a varying pair $(n,c)$ such that the relation $\E_{\0}\de_{n,p,c}\to\al$ in \eqref{eq:AS} holds and at that \eqref{eq:NA}
holds for $\th=\0$. 
	\item 
For any $p\in[0,\infty]$ and any varying $\no\cdot_2$-unit vector $\uu$, there exists a $(p,2)$-NAAS varying triple  $(n_p,n_2,\th_1)$ with $\th_1$ in the direction of $\uu$. 
	\item
Take any $p\in[-\infty,0)$. Take any varying $\no\cdot_2$-unit vector $\uu$ and let 
\begin{equation}\label{eq:d_0}
	d_0(\uu):=\sum\ii{u_j=0}. 
\end{equation}
Then the following five statements are equivalent to each other: 
\begin{enumerate}
	\item	there is a $(p,2)$-NAAS varying triple $(n_p,n_2,\th_1)$ with $\th_1$ in the direction of $\uu$; 
	\item  for \emph{some} varying $c$ and some $p$-NA varying pair $(n,\th_1)$ with $\th_1$ in the direction of $\uu$ such that $\E_{\0}\de_{n,p,c}\to\al$, one has $\lim\E_{\th_1}\de_{n,p,c}\ge\be$; 
	\item  for \emph{any} varying $c$ and some $p$-NA varying pair $(n,\th_1)$ with $\th_1$ in the direction of $\uu$ such that $\E_{\0}\de_{n,p,c}\to\al$, one has $\lim\E_{\th_1}\de_{n,p,c}\ge\be$; 
	\item 
\begin{equation*}
\begin{alignedat}{2}
	\inf_{t\in[0,\infty)}\sum f_p(tu_j)&\le(K_p+o(1))\,\ka_p(d) &&\text{ if }p\in[-\infty,1), \\
	\sup_{t\in[0,\infty)}\sum f_p(tu_j)&\ge(K_p-o(1))\,\ka_p(d) &&\text{ if }p\in[-1,0);
\end{alignedat}
\end{equation*}
	\item 
\begin{equation*}
	d_0(\uu)\le
	\left\{
	\begin{alignedat}{2}
&(K_p+o(1))d &&\text{ if }p\in[-\infty,-1), \\
&\textstyle{d-\big(\sqrt{\tfrac\pi2}\,K_{-1}-o(1)\big)\tfrac d{\ln d} } &&\text{ if }p=-1, \\
&d-\big(\tfrac{K_p}{\la_p(0)}-o(1)\big)d^{|p|\vee\frac12} &&\text{ if }p\in(-1,0)\setminus\{-\tfrac12\}, \\
&d-\big(\tfrac{K_{-1/2}}{\la_{-1/2}(0)}-o(1)\big)(d\ln d)^{\frac12} &&\text{ if }p=-\tfrac12.
	\end{alignedat}
	\right.
\end{equation*}
\end{enumerate}
\end{enumerate}
\end{proposition}

\begin{remark}\label{rem:power}
In view of Proposition~\ref{prop:NA}, statement (III) of Proposition~\ref{prop:AS exist} implies that the tests $\de_{n,p,c}$ -- for $p\in[-\infty,0)$ and directions $\uu$ with too many zero coordinates $u_j$ -- lack the necessary power even for arbitrarily large sample sizes $n$. 
\qed\end{remark}

\begin{remark}\label{rem:test,power}
From the proof of Theorem~\ref{th:AS iff} -- see e.g.\ \eqref{eq:converg}, one can can easily obtain critical values $c$ for the size of $\E_{\0}\de_{n,p,c}$ of the test $\de_{n,p,c}$ defined in \eqref{eq:test} to approximately equal the prescribed value $\al$. 
For instance, if $p\in(0,\infty)$, then an approximate-size-$\al$ test is the indicator of the inequality 
\begin{equation*}
	\no{\overline{\Xd_n}}_p^p>\E|Z|^p+\tfrac1{\sqrt d}\,\Phi^{-1}(1-\al)\sqrt{\Var(|Z|^p)}. 
\end{equation*}
Also, from the proof or from the statement of Theorem~\ref{th:AS iff}, it is easy to obtain an approximate expression for the power of the test; for instance, solving \eqref{eq:as iff} for $\be$ when $p\in(0,\infty)$, one has   
\begin{equation*}
	\be\approx\Phi\Big(\Phi^{-1}(\al)+\frac1{\sqrt{d\,\Var(|Z|^p)}}\,\sum_{j=1}^d\big(\E|Z+\tth_{1,j}\sqrt n\,|^p-\E|Z|^p\big)\Big),
\end{equation*}
where the $\tth_{1,j}$'s are the coordinates of the alternative vector $\th_1$. 
\qed\end{remark}

\subsection{Asymptotic relative efficiency ($\are$) of the $p$-mean tests}\label{are}
In this subsection we shall describe the asymptotic relative efficiency ($\are$) of the $p$-mean tests, relative to the corresponding $2$-mean tests. 

\begin{definition}\label{def:are}
For any given $p\in[-\infty,\infty]$ and any given varying $\no\cdot_2$-unit vector $\uu$,    
let  
\begin{equation*}
\are_{p,2}:=\are_{p,2,\uu}
:=\are_{p,2,\uu}(\al,\be):=\lim\frac{n_2}{n_p} 	
\end{equation*}
provided that 
	this limit 
	exists, and is the same, for all $(p,2)$-NAAS varying triples $(n_p,n_2,\th_1)$ with $\th_1$ in the direction of $\uu$; at that, let us allow the value $\infty$ for this limit, and hence for $\are_{p,2}$. 
As Proposition~\ref{prop:AS exist} shows, for $p\in[-\infty,0)$ and varying $\no\cdot_2$-unit vectors $\uu$ with too many zero coordinates $u_j$, there are no $(p,2)$-NAAS varying triples $(n_p,n_2,\th_1)$ with $\th_1$ in the direction of $\uu$; for such $\uu$, in view of Remark~\ref{rem:power}, we may and shall set $\are_{p,2}:=0$. 
 \qed
\end{definition}

To simplify writing, we allow the same symbol $\uu$ to denote sometimes a particular, ``fixed'' vector in $\R^d$ and, in other contexts, a ``varying vector'', that is, a function (of $d$ and/or other, possibly ``hidden'' parameters), yet hoping to avoid confusion. Let us emphasize that, in the notation $\are_{p,2,\uu}$, the symbol $\uu$ denotes such a function, so that the value of $\are_{p,2,\uu}$ depends on this entire vector function $\uu$ rather than on a particular vector value of the function; thus, $\are_{p,2,\uu}$ can be thought of as a constant, while thinking of the given function $\uu$ as a ``varying vector''. 

Note that Definitions~\ref{def:AS} and \ref{def:are} are generally in accordance with 
\cite{aredef}.

To describe $\are_{p,2,\uu}$, we need more definitions. 

Let 
\begin{equation}\label{eq:a(p)}
a_p:=\left\{
\begin{aligned}
\frac{|p| \Ga (\frac{p+1}{2})}{\sqrt{2} \sqrt{\Ga(\frac12)\,
   \Ga(p+\frac12)-\Ga(\frac{p+1}{2})^2}}\, &\text{ if } p\in(-\tfrac12,0)\cup(0,\infty),\\
\tfrac2\pi&\text{ if } p=0. 
\end{aligned}
\right.
\end{equation}
Note that $a_p$ is well defined -- because the Gamma function is strictly log-convex and hence $\Ga(\frac12)\,
   \Ga(p+\frac12)-\Ga(\frac{p+1}{2})^2>0$ for all $p\in(-\tfrac12,\infty)\setminus\{0\}$. 
The following proposition establishes the main properties of $a_p$ and, in particular, allows one to extend the above definition of $a_p$ to all $p\in[-\infty,\infty]$.    

\begin{proposition}\label{prop:a_p}
One has 
\begin{equation}\label{eq:0<a<1}
	a_p\in(0,1)
\end{equation}
for all $p\in(-\tfrac12,\infty)\setminus\{2\}$. 
Moreover, $a_2=1$ and $a_p$ is continuous at $p=0$ \big(and hence real-analytic in $p\in(-\tfrac12,\infty)$\big). 
Also, $a_{-\frac12+}=a_{\infty-}=0$; accordingly, let us \emph{define} the values $a_{-\frac12}$ and $a_\infty$ as $0$. Moreover, let $a_p:=0$ for $p\in[-\infty,-\frac12)$. 
\end{proposition}

Next, recall that for any $E\subseteq\R^d$ a function $f\colon E\rightarrow[0,\infty]$ is referred to as Schur-concave if it reverses the
Schur majorization $\succeq$: for any $\uu$ and $\vv$ in $E$ such that $\uu\succeq\vv$, one has $f(\uu)\le f(\vv)$; replacing here the inequality $f(\uu)\le f(\vv)$ by $f(\uu)\ge f(\vv)$, one obtains the definition of a Schur-convex function. 

Recall also the definition of the
Schur majorizarion: $\uu\succeq\vv$ (or, equivalently, 
$\vv\preceq\uu$)
means that $u_1+\cdots+u_d=v_1+\dots+v_d$ and $u_{[1]}+\cdots+u_{[j]}\ge v_{[1]}+\dots+v_{[j]}$ for all $j\in\left\{ 1,\ldots,d\right\} $,
where $u_{[1]},\dots,u_{[d]}$ are the
ordered numbers $u_1,\dots,u_d$, from the largest to the smallest. 

Let us say that a Lebesgue-measurable function $f\colon\R^d\rightarrow[0,\infty]$ is \emph{Schur${}^2$-concave} if $f(\vv)=f(v_1,\dots,v_d)$ is Schur-concave in $(v_1^2,\dots,v_d^2)$, that is, if there exists a Schur-concave function $g\colon[0,\infty)^d\to[0,\infty]$ such that $f(v_1,\dots,v_d)=g(v_1^2,\dots,v_d^2)$ for all $(v_1,\dots,v_d)\in\R^d$; replacing here all entries of ``concave'' by ``convex'', one obtains the definition of a \emph{Schur${}^2$-convex function}.  

For brevity, let $\vv^2:=(v_1^2,\dots,v_d^2)$ for any $\vv=(v_1,\dots,v_d)\in\R^d$. 

Let us then say that a set $A\subseteq\R^d$ is \emph{Schur${}^2$-convex} if its indicator function is  Schur${}^2$-concave; that is, if the conditions $\vv\in A$, $\ww\in\R^d$, and $\ww^2\preceq\vv^2$ imply $\ww\in A$. 
Similarly, one defines Schur${}^2$-concave sets; thus, a set $A\subseteq\R^d$ is \emph{Schur${}^2$-concave} if and only if its complement to $\R^d$ is Schur${}^2$-convex. 

Note that 
\begin{equation*}
	\one^2\preceq\uu^2\preceq\big(\sqrt{d}\,\ee_1\big)^2
\end{equation*}
for any $\no\cdot_2$-unit vector $\uu\in\R^d$, where 
\begin{equation*}
	\text{$\one:=\one_d:=(1,\dots,1)\in\R^d$\quad and\quad $\ee_1:=\ee_{d,1}:=(1,0,\dots,0)\in\R^d$.}
\end{equation*} 


\begin{theorem}\label{th:are schur2} \ 
In a certain sense, 
$\are_{p,2,\uu}$ is Schur${}^2$-concave or Schur${}^2$-convex in $\uu$ depending on whether $p\in[-\infty,2]$ or $p\in[2,\infty]$. Namely, take any $p\in[-\infty,\infty]$ and 
let $\uu$ and $\vv$ be any two varying $\no\cdot_2$-unit vectors such that $\uu^2\preceq\vv^2$ and at that $\are_{p,2,\uu}$ and $\are_{p,2,\vv}$ exist; then $\are_{p,2,\uu}\ge\are_{p,2,\vv}$ if $p\in[-\infty,2]$ and $\are_{p,2,\uu}\le\are_{p,2,\vv}$ if $p\in[2,\infty]$. 
\end{theorem}

\begin{definition}\label{def:orl}
For any vector $\vv\in\R^d$ and $p\in(-\frac12,2)$, let 
\begin{equation}\label{eq:orl}
	\no\vv_{p,2}:=\inf\big\{v\in(0,\infty)\colon\sum\nolimits_j g_p(\tfrac{v_j}v)\le K_p\,d^{1/2}\big\},
\end{equation}
where 
\begin{equation}\label{eq:g_p}
g_p(s):=\left\{
	\begin{aligned}
		f_p(s)&\text{ if }p\in(-\tfrac12,0), \\
		\tfrac{s^2}{e^2}\ii{|s|\le e}+\ln|s|\ii{|s|>e}
		&\text{ if }p=0, \\
		s^2\wedge|s|^p&\text{ if }p\in(0,2),  
	\end{aligned}
	\right.
\end{equation}
with $f_p(s)$ and 
$$K_p:=K_{\al,\be;p}$$
as in the table in Theorem~\ref{th:AS iff}. 
\end{definition}

Note that the set $\big\{v\in(0,\infty)\colon\sum g_p(\tfrac{v_j}v)\le K_p\,d^{1/2}\big\}$ is non-empty for any choices of $\vv\in\R^d$, $p\in(-\frac12,2)$, and $d\in\N$, because $g_p$ is continuous and $g_p(0)=0$. 
If the function $g_p$ were convex \big(which it is not, for any $p\in(-\frac12,2)$\big), then $\no\cdot_{p,2}$ would be a so-called Orlicz norm. 
However, the function $\no\cdot_{p,2}$ is positive-homogeneous: $\no{t\vv}_{p,2}=|t|\,\no\vv_{p,2}$ for all $\vv\in\R^d$, $t\in\R$, and $p\in(-\frac12,2)$.

\begin{theorem}\label{th:are}
For any $p\in[-\infty,\infty]$ and any varying $\no{\cdot}_2$-unit vector $\uu$,  
\begin{equation}\label{eq:are,d->infty}
\are_{p,2,\uu}\left\{
\begin{alignedat}{2}
&=0 &&\text{ if $p\in[-\infty,-\tfrac12]$}; \\
&=0 &&\text{ if  $p\in(-\tfrac12,2)$ and $\no{\uu}_{p,2}<<d^{1/4}$}; \\
&\in(0,a_p] &&\text{ if $p\in(-\tfrac12,2)$ and $\no{\uu}_{p,2}\OOG d^{1/4}$}; \\
&=a_p &&\text{ if $p\in(-\tfrac12,2)$ and $\uu=\one$};  
\\
&=1 &&\text{ if } p=2; \\
&=a_p &&\text{ if $p\in(2,\infty)$ and $\no{\uu}_p<<d^{(p-2)/(4p)}$}; \\
&=\infty &&\text{ if $p\in(2,\infty)$ and $\no{\uu}_p>>d^{(p-2)/(4p)}$}; \\
&=0 &&\text{ if $p=\infty$ and $\no{\uu}_\infty<<d^{1/4}\,\sqrt{\ln d}$}; \\
&=\infty &&\text{ if $p=\infty$ and $\no{\uu}_\infty>>d^{1/4}\,\sqrt{\ln d}$}.
\end{alignedat}
\right.
\end{equation}
At that, line 3 of \eqref{eq:are,d->infty} should be understood as follows: if for a given varying $\no\cdot_2$-unit vector $\uu$ the value $\are_{p,2,\uu}$ exists in $[0,\infty]$, it is in fact in $(0,a_p]$; the other lines of \eqref{eq:are,d->infty} should be understood in the sense that $\are_{p,2,\uu}$ exists and equals the specified value -- for any  $\no\cdot_2$-unit vector $\uu$ specified in the given line \big(or, as in lines 1 and 5 of \eqref{eq:are,d->infty}, for any varying  $\no\cdot_2$-unit vector $\uu$ whatsoever\big). 
\end{theorem}

Note that line 5 of \eqref{eq:are,d->infty} is not quite trivial -- since for $p=2$ there exist  $(p,2)$-AS varying triples $(n_p,n_2,\th_1)$ with $n_p\ne n_2$. 

As for line 3 of \eqref{eq:are,d->infty}, it is rather unlike any of the other lines there. In particular, one may wonder: would it be true that $\are_{p,2,\uu}=a_p$ if $p\in(-\tfrac12,2)$ and $\no{\uu}_{p,2}>>d^{1/4}$? However, for any given $p\in(-\frac12,2)$, there is no $\no\cdot_2$-unit vector $\uu$ such that $\no{\uu}_{p,2}>>d^{1/4}$. 
In fact, one has 

\begin{proposition}\label{prop:u range}\ 
\begin{enumerate}[(I)]
\item For each $p\in(-\tfrac12,2)$ and each $d\in\N$, the function $\R^d\ni\uu\mapsto\no{\uu}_{p,2}$ is Schur${}^2$-concave. 
	\item 
Therefore, for each $p\in(-\tfrac12,2)$ and all large enough $d$, the values of $\no{\uu}_{p,2}$ for $\no\cdot_2$-unit vectors $\uu$ 
range as follows: 
\begin{align}
	\max\big\{\no{\uu}_{p,2}\colon\no{\uu}_2=1\big\}&=\no{\one}_{p,2}\asymp d^{1/4}; 
	\label{eq:max}\\ 
	\min\big\{\no{\uu}_{p,2}\colon\no{\uu}_2=1\big\}&=\no{\sqrt d\ee_1}_{p,2} 
	\,\left\{
	\begin{alignedat}{2}
&=0 & & \text{\quad if }p\in(-\tfrac12,0), \\
&=e^{-C\sqrt d} & & \text{\quad if }p=0, \\
&\asymp d^{(p-1)/(2p)} & & \text{\quad if }p\in(0,2),
	\end{alignedat}	
	\right.
	\label{eq:min}
\end{align}
for some varying $C=C(d)\asymp1$. 
\item For each $p\in[2,\infty]$ and each $d\in\N$, the function $\R^d\ni\uu\mapsto\no{\uu}_p$ is Schur${}^2$-convex. 
\item 
Therefore, for each $p\in[2,\infty]$ and each $d\in\N$, the values of $\no{\uu}_p$ for $\no\cdot_2$-unit vectors $\uu$ 
range from $\no{\one}_p=1$ to $\no{\sqrt d\ee_1}_p=d^{(p-2)/(2p)}$, where $(p-2)/(2p)$ is interpreted as $1/2$ for $p=\infty$. 
\end{enumerate}
\end{proposition} 

It follows that for each $p\in(2,\infty)$, the threshold value $d^{(p-2)/(4p)}$ is, on the logarithmic scale, exactly in the middle of the range of values of $\no{\uu}_p$ for $\no\cdot_2$-unit vectors $\uu$, while for $p=\infty$ the threshold value $d^{1/4}\sqrt{\ln d}$ is, logarithmically, slightly above the middle of the range. As for $p\in(-\tfrac12,2)$, for each such $p$ the ``phase transition'' occurs at the threshold $d^{1/4}$, which is near the upper end of the range of values of $\no{\uu}_{p,2}$ for $\no\cdot_2$-unit vectors $\uu$. 


\begin{theorem}\label{th:are range}\ 

\begin{enumerate}[(I)]
\item For each of the two (varying with $d$) $\no\cdot_2$-unit vectors, $\uu=\one$ and $\uu=\sqrt{d}\,\ee_1$ in $\R^d$, the values $\are_{p,2,\uu}$ exist and are extremal: if $\uu$ is any other varying $\no\cdot_2$-unit vector in $\R^d$ for which $\are_{p,2,\uu}$ exists, then 
\begin{equation*}
\begin{alignedat}{2}
0=&\are_{p,2,\sqrt{d}\,\ee_1}\le\are_{p,2,\uu}\le\are_{p,2,\one}=a_p	
\quad &&\text{for each }p\in[-\infty,2); \\
&\are_{p,2,\sqrt{d}\,\ee_1}=\are_{p,2,\uu}=\are_{p,2,\one}=1	
\quad &&\text{for }p=2; \\
\infty=&\are_{p,2,\sqrt{d}\,\ee_1}\ge\are_{p,2,\uu}\ge\are_{p,2,\one}=a_p	
\quad &&\text{for each }p\in(2,\infty]. 
\end{alignedat}
\end{equation*}
\item
For each $p\in[-\infty,\infty]$, the range of values of $\are_{p,2,\uu}$ over all varying \break 
$\no\cdot_2$-unit vectors $\uu$ is the entire interval between $\are_{p,2,\sqrt{d}\,\ee_1}$ and $\are_{p,2,\one}$; that is, for each value $a$ in this interval, there is a varying $\no\cdot_2$-unit vector $\uu$ such that $\are_{p,2,\uu}$ exists and equals $a$.  
\end{enumerate}
\end{theorem}

Theorems~\ref{th:are} and \ref{th:are range} are illustrated in Figure~\ref{fig:are}, given in the Introduction. 
The red boundary curve is the graph of the function $p\mapsto a_p$; recall Proposition~\ref{prop:a_p}.  

The following proposition together with Theorem~\ref{th:are} and Proposition~\ref{prop:a_p} imply, for any given $p\in(2,\infty)$, that if the $\no{\cdot}_2$-unit vector $\uu$ is completely random for each $d$, then with high probability $\are_{p,2,\uu}=a_p<1$.   

\begin{proposition}\label{prop:random u}
If a varying (with $d\to\infty$) vector $\uu$ is uniformly distributed on the sphere in $\R^d$ of $\|\cdot\|_2$-radius $\sqrt d$ centered at $\0$, then for each $p\in(2,\infty)$ one has  $\no{\uu}_p<<d^{(p-2)/(4p)}$ with probability close to $1$ (as $d\to\infty$). 
\end{proposition}

However, no reason is seen in general for $\uu$ to be completely random. Rather, as discussion in the Introduction suggests, it appears rather typical for large dimensions $d$ that only comparatively few coordinates of $\uu$ are large in absolute value.   

\begin{remark}\label{rem:n,th}
The condition in Proposition~\ref{prop:NA} that the sample size $n$ be large enough is clearly needed in general for normal approximation, so that the varying pair $(n,\th_1)$ be $p$-NA. As for the condition in Proposition~\ref{prop:NA} that $\|\th_1\|$ be small, it is needed there only to ensure that the covariance matrix $\Si_d(\th_1)$ be close enough to $\Si_d(\0)=I_d$ and hence well-conditioned, in sense that the matrix norm $\|\Si_d(\th_1)^{-1}\|$ be not too large -- see the proof of Proposition~\ref{prop:NA}. 
Obviously, as far as the $p$-NA condition is concerned, the mentioned restrictions on $n$ and $\|\th_1\|$ in Proposition~\ref{prop:NA} will be not needed if the observations $\X_i$ can themselves be assumed to be normally distributed. 

The further concern, about the existence of $p$-NAAS pairs, will addressed in the next section in Proposition~\ref{prop:s->n,th}, where it is required as well that $\|\th_1\|$ be small; the reason for this requirement is the discrete nature of the sample size $n$, which must be an integer.  

However, none of these conditions imposed in Propositions~\ref{prop:NA} and \ref{prop:s->n,th} will be needed if the observations are normally distributed \emph{and} the ``sample size'' $n$ is ``continuously valued'' (that is, $n$ can take any positive real value) -- as in the following setting. Suppose one observes an unknown vector $\th$ in $\R^d$ in the standard Gaussian white noise $\dd\B_\cdot$ over the time period $[0,n]$, for some real $n>0$, which can be regarded as the ``sample size''. That is, the observations are given by the simplest stochastic differential equation 
$\dd\SS_t=\th\,\dd t+\dd\B_t$ for all real $t\in[0,n]$ with $\SS_0=\0$, where $(\B_t)$ is the standard Brownian motion in $\R^d$. The obvious solution to this equation (with the initial condition $\SS_0=\0$) is $\SS_t=t\th+\B_t$ for $t\in[0,n]$. The test here still will be $\de_{n,p,c}$ as given in \eqref{eq:test}, with $\overline{\Xd_n}:=\frac1n\,\SS_n$. 
Then all the results stated in this section will remain valid, without any prior restrictions such as $n$ being large or $\|\th_1\|$ being small. 
\qed\end{remark}

\section{Statements of key propositions, and proofs of the results stated in Section~\ref{results}
}\label{proofs}
In this section, we shall state a number of propositions (possibly of independent interest), whence the theorems of Section~\ref{results} will immediately follow; the proofs of these propositions will be deferred to Sections~\ref{proof of shifts} and \ref{proofs of props p}. Here we shall also prove the propositions stated in Section~\ref{results}, which may or may not depend on the propositions stated in the current section, Section~\ref{proofs}. 

\subsection{Statements of key propositions}\label{key props}

The key notion is provided by 

\begin{definition}\label{def:AS shift} [Cf.\ Definition~\ref{def:AS}.] 
Take any $p\in[-\infty,\infty]$. 
Say that a varying vector $\s$ in $\R^d$ is a $p$-weakly-AS \emph{shift} if for \emph{some} choice of a varying real number $c$ \big(necessarily with values in $(0,\infty)$\big) one has 
\begin{equation*}
	\PP(\no{\Zd}_p>c)\to\al\quad\text{and}\quad\PP(\no{\Zd+\s}_p>c)\to\be;    
\end{equation*}
Say that $\s$ is a $p$-strongly-AS \emph{shift} if for \emph{any} choice of a varying $c$ such that $\PP(\no{\Zd}_p>c)\to\al$ one has 
$\PP(\no{\Zd+\s}_p>c)\to\be$.    
\end{definition}

\begin{proposition}\label{prop:n,th->s}
A $p$-NA varying pair $(n,\th_1)$ is 
$p$-weakly-AS if and only if the vector $\s:=\sqrt{n}\,\th_1$ is a $p$-weakly-AS shift. 
The same holds with ``strongly'' instead of ``weakly''. 
\end{proposition}

Proposition~\ref{prop:n,th->s} follows immediately from  Definitions~\ref{def:AS}, \ref{def:NA}, \ref{def:AS shift}, and \eqref{eq:test}.

\begin{proposition}\label{prop:AS shift} [Cf.\ Proposition~\ref{prop:AS}.] 
Any $p$-weakly-AS shift $\s$ is a $p$-strongly-AS shift; obviously, vice versa is true as well.  
\end{proposition}

By virtue of this proposition, on which Proposition~\ref{prop:AS} is based, we shall be simply referring to $p$-AS shifts $\s$, at that omitting the adjectives ``weakly'' and ``strongly''. 

Next is the centerpiece of this section: 
\begin{proposition}\label{prop:AS shift iff}
(Cf.\ Theorem~\ref{th:AS iff}.)\quad 
For each $p\in[-\infty,\infty]$, a varying vector $\s$ is a $p$-AS shift if and only if condition \eqref{eq:as iff} holds. 
\end{proposition}

\begin{proposition}\label{prop:below}
For each $p\in[-\infty,\infty]$, there is a function $\N\ni d\mapsto\si_p(d)\in(0,\infty)$ such that for each $p$-AS shift $\s$ one has $\|\s\|\ge\si_p(d)$ eventually. In fact, one may choose $\si_p(d)$ so that 
\begin{equation}\label{eq:si_p(d)}
	\si_p(d)\asymp
	\left\{
	\begin{alignedat}{2}
	& d^{1/2} &&\text{ if }p\in[-\infty,1),\\
	& (\tfrac d{\ln d})^{1/2} &&\text{ if }p=-1,\\
	& d^{|p|/2} &&\text{ if }p\in(-1,-\tfrac12),\\
	& (d\ln d)^{1/4} &&\text{ if }p=-\tfrac12,\\
	&	d^{1/4} &&\text{ if }p\in(-\tfrac12,2],\\
	&	d^{1/(2p)} &&\text{ if }p\in[2,\infty),\\
	&	(\ln d)^{1/2} &&\text{ if }p=\infty.
	\end{alignedat}
	\right.
\end{equation}
\end{proposition}

\begin{proposition}\label{prop:t->1}
For each $p\in[-\infty,\infty]$, 
if $\s$ is a $p$-AS shift and a varying scalar $t$ goes to $1$ fast enough, then $t\s$ is a $p$-AS shift as well. 
More specifically, let 
\begin{equation*}
	\tau_p(d):=
	\left\{
	\begin{alignedat}{2}
	&1 &\text{ if } &p\in[-\infty,\infty),\\
	&\frac1{\ln d} &\text{ if } &p=\infty; 
	\end{alignedat}
	\right.
\end{equation*}
then 
for 
each $p\in[-\infty,\infty]$, 
any varying $t\in(0,\infty)$ such that $|t-1|<<\tau_p(d)$, and for each $p$-AS shift $\s$, the varying vector  $t\s$ is also a $p$-AS shift.   
\end{proposition} 

\begin{proposition}\label{prop:s->n,th}
For each $p\in[-\infty,\infty]$, there exists a function $\N\ni d\mapsto\tth_p(d)\in(0,\infty)$ such that the following holds: if $\|\th_1\|\le\tth_p(d)$ and 
$\s_p$ is a $p$-AS shift in the direction of $\th_1$, then 
the pair $(n_p,\th_1)$ with 
\begin{equation}\label{eq:n}
	n_p:=\bigg\lceil\frac{\|\s_p\|^2}{\|\th_1\|^2}\bigg\rceil
\end{equation}
will be $p$-NAAS, and at that one will have $\|\s_p\|\sim\sqrt{n_p}\,\|\th_1\|$.
\end{proposition}

\begin{proposition}\label{prop:AS shift exist}
(Cf.\ Proposition~\ref{prop:AS exist}.)  
\begin{enumerate}[(I)] 
	\item 
For each $p\in[-\infty,\infty]$ and each $d\in\N$, there exists a unique $c\in(0,\infty)$ such that 
$\PP(\no{\Zd}_p>c)=\al$. 
		\item 
For each $p\in[0,\infty]$, each $d\in\N$, and each $\no{\cdot}_2$-unit vector $\uu\in\R^d$, 
there exists a unique vector $\s$ in the direction of $\uu$ such that 
$\PP(\no{\Zd+\s}_p>c)=\be$, 
where $c$ is as in part (i); 
clearly, this vector $\s$ will be a $p$-AS shift. 
	\item
Take any $p\in[-\infty,0)$. Take any varying $\no\cdot_2$-unit vector $\uu$ and let $d_0(\uu):=\sum\ii{u_j=0}$, as in \eqref{eq:d_0}. 
Then the following five statements are equivalent to each other: 
\begin{enumerate}[(a)]
	\item	there is a $p$-AS shift $\s$ in the direction of $\uu$; 
	\item for \emph{some} varying $c$ such that $\PP(\no{\Zd}_p>c)\to\al$ 
				and some varying vector $\s$ in the direction of $\uu$, one has $\lim\PP(\no{\Zd+\s}_p>c)\ge\be$.  
	\item for \emph{any} varying $c$ such that $\PP(\no{\Zd}_p>c)\to\al$ 
				and some varying vector $\s$ in the direction of $\uu$, one has $\lim\PP(\no{\Zd+\s}_p>c)\ge\be$. 
	\item statement (III)(d) of Proposition~\ref{prop:AS exist} holds. 
	\item statement (III)(e) of Proposition~\ref{prop:AS exist} holds. 
\end{enumerate}
\end{enumerate}
\end{proposition}

\begin{definition}\label{def:a_p2} (Cf.\ Definition~\ref{def:are}.) 
For any given $p\in[-\infty,\infty]$ and any given varying $\no\cdot_2$-unit vector $\uu$,    
let  
\begin{equation*}
a_{p,2}:=a_{p,2,\uu}
:=a_{p,2,\uu}(\al,\be):=\lim\frac{\|\s_2\|^2}{\|\s_p\|^2},	
\end{equation*}
provided that 
	this limit 
	exists, and is the same, for all $p$-AS shifts $\s_p$ in the direction of $\uu$ and all $2$-AS shifts $\s_2$ in the same direction; at that, let us allow the value $\infty$ for this limit, and hence for $a_{p,2}$. 
As Proposition~\ref{prop:AS shift exist} shows, for $p\in[-\infty,0)$ and some varying nonzero vectors $\uu$, there are no $p$-AS shifts $\s_p$ in the direction of $\uu$; in such a case, set $a_{p,2,\uu}:=0$.  \qed
\end{definition}

\begin{proposition}\label{prop:are=a_p2}
Take any $p\in[-\infty,\infty]$ and any varying $\no{\cdot}_2$-unit vector $\uu$. Then $\are_{p,2,\uu}$ exists if and only if $a_{p,2,\uu}$ exists, in which case one has $\are_{p,2,\uu}=a_{p,2,\uu}$. 
\end{proposition}

\begin{proposition}\label{prop:a_p2 schur2}
The statement of Theorem~\ref{th:are schur2} will hold with $\are_{p,2,\cdot}$ replaced everywhere by $a_{p,2,\cdot}$.  
\end{proposition}

\begin{proposition}\label{prop:sARE}
The statement of Theorem~\ref{th:are} will hold with $\are_{p,2,\cdot}$ replaced everywhere by $a_{p,2,\cdot}$.  
\end{proposition}

\begin{proposition}\label{prop:a_p2 range}
The statement of Theorem~\ref{th:are range} will hold with $\are_{p,2,\cdot}$ replaced everywhere by $a_{p,2,\cdot}$.  
\end{proposition}


\subsection{Proofs of the results stated in Section~\ref{results}}\label{th proofs}

\begin{proof}[Proof of Proposition~\ref{prop:NA}]
By \cite[Corollary~15.3]{bhat},
\begin{align}
	\De_\th(A)&:=\Big|\PP_\th\Big(\big(\tfrac1n\,\Si_d(\th)\big)^{-1/2}\big(\overline{\Xd_n}-\th\big)\in A\Big)
	-\PP\big(\Zd\in A\big)\Big| \notag\\
	&\le C_1(d)\rho_3(d)\,n^{-1/2}+2\sup_{\yy\in\R^d}\PP\big(\Zd\in(\partial A)^\eta+\yy\big) \label{eq:BE}
\end{align}
where $\th$ is in the neighborhood $\mathcal{V}_d$ as in \eqref{eq:rho}, 
$A$ is any Borel subset of $\R^d$, 
$(\partial A)^\eta$ is the $\eta$-neighborhood of boundary $\partial A$ of the set $A$, $\eta:=C_2(d)\rho_3(d)\,n^{-1/2}$; here and in what follows in this proof, $C_1(d),C_2(d),\dots$ stand for finite positive real constants depending only on $d$.  

For any $c\in\R$, consider the set 
\begin{equation*}
	A_c:=\{\zz\in\R^d\colon\no{\zz}_p\le c\}.
\end{equation*}
Observe that $A_c$ is convex for $p\in[1,\infty]$. As for the case $p\in[-\infty,1]$, the complement of $A_c$ to $\R^d$ is the union of $2^d$ convex sets, namely, the union of the intersections of the complement of $A_c$ with each of the $2^d$ coordinate orthants in $\R^d$. So, by \cite[Corollary~3.2]{bhat}, 
\begin{equation*}
	\sup_{\yy\in\R^d}\PP\big(\Zd\in(\partial(A_c))^\eta+\yy\big)\le C_3(d)\,\rho_3(d)\,n^{-1/2}, 
\end{equation*}
which goes to $0$ 
for any fixed $d$ as $n\to\infty$, whence, by \eqref{eq:BE}, one has 
\begin{equation*}
	\sup_{\th\in\mathcal{V}_d}\sup_{c\in\R}\De_\th(A_c)\longrightarrow0;
\end{equation*}
in particular, recalling now \eqref{eq:I}, one has \eqref{eq:NA} for $\th=\0$. 
Observe next that the events $\big\{\no{\sqrt n\,\overline{\Xd_n}}_p\le c\big\}$ 
and $\big\{\no{\Zd+\sqrt{n}\,\th}_p\le c\big\}$ 
can be rewritten as 
$\big\{\big(\tfrac1n\,\Si_d(\th)\big)^{-1/2}\big(\overline{\Xd_n}-\th\big)\in \tA_{\th,c}\big\}$ and $\big\{\Zd\in A_{\th,c}\big\}$, where 
\begin{equation*}
	\tA_{\th,c}:=\Si_d(\th)^{-1/2}\,A_{\th,c}\quad\text{and}\quad A_{\th,c}:=A_c-\sqrt{n}\,\th.
\end{equation*} 

So, to complete the proof of Proposition~\ref{prop:NA}, it suffices to observe that for each fixed $d$ 
\begin{equation}\label{eq:A{c,th}}
\sup_{c\in\R}\big|\PP\big(\Zd\in \tA_{\th,c}\big)-\PP\big(\Zd\in A_{\th,c}\big)\big|
\longrightarrow0	
\end{equation}
as $\th\to\0$; indeed, then \eqref{eq:A{c,th}} will hold for $d\to\infty$, some positive real function $\tth_p(\cdot)$, and 
for all $\th$ with $\|\th\|\le\tth_p(d)$. 
To verify \eqref{eq:A{c,th}} for a fixed $d$, take any $R\in(0,\infty)$ and consider the ball 
\begin{equation*}
	D_R:=\{\zz\in\R^d\colon\|\zz\|\le R\}. 
\end{equation*} 
Let $\|\th\|$ be so small as $\|T^{-1}\|\le2$, where $T:=\Si_d(\th)^{-1/2}$ \big(recall that $\Si_d(\th)\to I_d$ as $\|\th\|\to0$\big). 

Then $D_R\cap(\tA_{\th,c}\oplus A_{\th,c})\subseteq(\partial A_{\th,c})^{\eta_1}$, where $\oplus$ denotes the Boolean symmetric difference and $\eta_1:=3\,\|\Si_d(\th)^{-1/2}-I_d\|\,R$. 
Indeed, w.l.o.g.\ $\eta_1\ne0$. 
Let for brevity $A:=A_{\th,c}$, so that $\tA_{\th,c}=TA$. 
If $D_R\cap(\tA_{\th,c}\oplus A_{\th,c})\not\subseteq(\partial A_{\th,c})^{\eta_1}$, then, by \cite[Corollary~2.6]{bhat}, there exists some $x\in\R^d$ such that 
\begin{equation*}
	\|x\|\le R,\quad x\in(A\setminus TA)\cup(TA\setminus A),\quad 
	\text{and } x\notin A^{\eta_1}\text{ or }x\in A^{-\eta_1}.  
\end{equation*}
Let $y:=T^{-1}x$, so that $\|y\|\le\|T^{-1}\|\,\|x\|\le2\|x\|\le2R$ and 
$\|x-y\|=\|Ty-y\|\le\|T-I_d\|\,\|y\|\le\|T-I_d\|2R<\eta_1$. 
Now, in the case when $x\notin A^{\eta_1}$, one has $x\notin A$, $x\in TA\setminus A$, $x\in TA$, $y=T^{-1}x\in A$, and $\|x-y\|<\eta_1$, which implies $x\in A^{\eta_1}$, a contradiction. 
In the remaining case, when $x\in A^{-\eta_1}$, one has $\|x-y\|<\eta_1$, $y\in A$, $x\in A$, $x\in A\setminus TA$, $x\notin TA$, $y=T^{-1}x\notin A$, again a contradiction.   

Therefore, 
\begin{multline*}
\sup_{c\in\R}\big|\PP\big(\Zd\in \tA_{\th,c}\cap D_R\big)-\PP\big(\Zd\in A_{\th,c}\cap D_R\big)\big| 
\le\sup_{c\in\R}\PP\big(\Zd\in(\partial A_{\th,c})^{\eta_1}\big)\\
\le C_4(d)\,\eta_1
\le 3C_4(d)\,\|\Si_d(\th)^{-1/2}-I_d\|\,R
\longrightarrow0	
\end{multline*}
as $\th\to\0$, for each $R\in(0,\infty)$.
It remains to note that $\PP\big(\Zd\notin D_R\big)\to0$ as $R\to\infty$. 
\end{proof}

\begin{proof}[Proof of Proposition~\ref{prop:AS}]
This follows by Propositions~\ref{prop:n,th->s} and \ref{prop:AS shift}. 
\end{proof}

\begin{proof}[Proof of Theorem~\ref{th:AS iff}]
This follows by Propositions~\ref{prop:n,th->s} and \ref{prop:AS shift iff}. 
\end{proof}

\begin{proof}[Proof of Proposition~\ref{prop:AS exist}]\ 

\framebox{(I)} Part (I) of Proposition~\ref{prop:AS exist} follows by Proposition~\ref{prop:AS shift exist}(I) and Proposition~\ref{prop:NA}. 

\framebox{(II)} Part (II) of Proposition~\ref{prop:AS exist} follows by Proposition~\ref{prop:AS shift exist}(II) and Proposition~\ref{prop:s->n,th}. 

\framebox{(III)} To prove part (III) of Proposition~\ref{prop:AS exist}, note first that, in fact for each  $p\in[-\infty,\infty]$, statement (III)(a) of Proposition~\ref{prop:AS exist} is equivalent to (III)(a) of Proposition~\ref{prop:AS shift exist}. In other words, there exists a $(p,2)$-NAAS varying triple $(n_p,n_2,\th_1)$ with $\th_1$ in the direction of $\uu$ if and only if there exists a $p$-AS shift $\s$ in the same direction; the ``only if'' part of the latter equivalence follows by Proposition~\ref{prop:n,th->s}, while the ``if'' part follows by part (II) of Proposition~\ref{prop:AS shift exist} (for $p=2$) and Proposition~\ref{prop:s->n,th}. 


Next, letting \ref{prop:AS exist} and \ref{prop:AS shift exist} stand for Proposition~\ref{prop:AS exist} and Proposition~\ref{prop:AS shift exist}, respectively, 
one has implications 
\ref{prop:AS exist}(III)(b)
$\implies$\ref{prop:AS shift exist}(III)(b)  
$\implies$\ref{prop:AS shift exist}(III)(c)
$\implies$\ref{prop:AS exist}(III)(c)
$\implies$\ref{prop:AS exist}(III)(b);
the first of them follows by Definition~\ref{def:NA}; the second one follows by Proposition~\ref{prop:AS shift exist}(III); the third one, by Proposition~\ref{prop:s->n,th} \big(used with $\tbe:=\lim\PP(\no{\Z+\s}_p>c)$ in place of $\be$\big); and the last implication follows by part (I) of Proposition~\ref{prop:AS exist}. 
Thus, part (III) and hence the entire Proposition~\ref{prop:AS exist} follows by Proposition~\ref{prop:AS shift exist}(III).  
\end{proof}

\begin{proof}[Proof of Proposition~\ref{prop:a_p}] 
Let us first prove the continuity of $a_p$ at $p=0$. By \eqref{eq:a(p)},
\begin{equation*}
	\frac{\Ga\left(\frac{p+1}{2}\right)^2}{2 a_p^2}=\frac{F(p)}{G(p)},
\end{equation*}
where 
$
	F(p):=\Ga\left(\frac{1}{2}\right)
   \Ga\left(p+\frac{1}{2}\right)-\Ga\left(\frac{p+1}{2}\right)^2$ and  
   $G(p):=p^2.$  
So, using l'Hospital's rule for limits and the identity $\Ga(\frac12)=\sqrt\pi$, one has 
\begin{equation*}
	\frac\pi{2 a_p^2}\underset{p\to0}\longrightarrow\frac{F''(0)}{G''(0)}
	=\tfrac{1}{4} \big[\Ga\left(\tfrac{1}{2}\right)
   \Ga''\left(\tfrac{1}{2}\right)-\Ga'\left(\tfrac{1}{2}\right)^2\big]
   =\tfrac{\pi^3}8,
\end{equation*}
in view of \eqref{eq:Ga,pi}. Now the continuity of $a_p$ at $p=0$ follows. 
Moreover, $F$ is real-analytic on $(-\frac12,\infty)$, with $F(0)=F'(0)=0<F''(0)$; hence, there is a function $F_1$ such that $F_1$ is positive and real-analytic on $(-\frac12,\infty)$ and  
$F(p)=p^2F_1(p)$ for all $p\in(-\frac12,\infty)$. It follows that $a_p$ is real-analytic in $p\in(-\frac12,\infty)$. 

The equality $a_2=1$ follows because $\Ga(\frac12)=\sqrt\pi$ and $\Ga(x)=(x-1)\Ga(x-1)$ for $x>1$. The relations  $a_{-\frac12+}=0$ and $a_{\infty-}=0$ follow because $\Ga(0+)=\infty$ and, by Stirling's formula, $\Ga(p+\frac12)>>\Ga(\frac{p+1}{2})^2$ as $p\to\infty$. 

It remains to check the claim that \eqref{eq:0<a<1} holds for all $p\in(-\frac12,\infty)\setminus\{2\}$. Since $a_0=\frac2\pi\in(0,1)$, this claim can be rewritten as 
\begin{equation}\label{eq:r>}
	r(p):=\frac{\Ga(\tfrac12)\Ga(p+\tfrac12)}{\Ga(\tfrac{p+1}2)^2}
	>1+\tfrac{p^2}2
\end{equation}
for all $p\in(-\frac12,\infty)\setminus\{0,2\}$. 
Observe that, by the strict log-convexity of the Gamma function, $r(p)>1$ for all $p\in(-\frac12,\infty)\setminus\{0\}$. Of course, this is not enough to prove \eqref{eq:r>}. One idea to improve on the just mentioned log-convexity argument is to use the ``plus 1'' identity 
$\Ga(x)=\tfrac1x\,\Ga(x+1)$ for $x>0$,
whereby a simple strictly log-convex factor, $\frac1x$, is extracted from $\Ga(x)$, while the remaining factor, $\Ga(x+1)$, is still strictly log-convex \big(but necessarily ``less strictly'' log-convex than $\Ga(x)$; one may note here \cite[(1.2.14)]{andrews} that $\tfrac{\dd^2}{\dd x^2}\ln\Ga(x)=\sum_{k=0}^\infty\frac1{(x+k)^2}$ decreases to $0$ as $x$ increases from $0$ to $\infty$ -- cf.\ \eqref{eq:Ga,pi}\big). One may also note that the ``plus 1'' identity and the log-convexity property are the two characteristic properties of the Gamma function; indeed, by the Bohr-Mollerup theorem, these two properties (together with the normalization $\Ga(1)=1$) completely characterize the Gamma function -- see e.g.\ \cite[Theorem~2.1]{artin} or \cite[Theorem~1.9.3]{andrews}. 
Another fundamental property of the Gamma function is  
the Legendre duplication formula, $\Ga(\frac12)\Ga(x)=2^{x-1}\Ga(\frac x2)\Ga(\frac{x+1}2)$; in fact \cite[Theorem~6.1]{artin}, the Gamma function is the only positive twice continuously differentiable function on $(0,\infty)$ that satisfies both the ``plus 1'' and Legendre duplication formulas. 
Using these properties of the Gamma function, one has 
\begin{equation}\label{eq:r=}
	r(p)=r_i(p)\,\tr_i(p) 
\end{equation}
for $i=1,2,3$, 
where
\begin{align*}
	r_1(p)&:=
	\prod _{j=0}^0
   \frac{\left(\frac{p+1}{2}+j\right)^2}{\left(\frac{1}{2}+j\right)
   \left(p+\frac{1}{2}+j\right)}; \\
\tr_1(p)&:=\frac{\Gamma \left(\frac{1}{2}+1\right) \Gamma
   \left(p+\frac{1}{2}+1\right)}{\Gamma
   \left(\frac{p+1}{2}+1\right)^2}
>1\quad\text{for }p\in(-\tfrac12,\infty)\setminus\{0\}; \\
	r_2(p)&:=
	\tfrac{(p+1)^2}{3} \prod _{j=1}^3
   \frac{\left(\frac{p+1}{2}+j\right)^2}{\left(\frac{3}{2}+j\right)
   \left(p-\frac{1}{2}+j\right)}; \\
\tr_2(p)&:=\frac{\Gamma \left(\frac{3}{2}+4\right) \Gamma
   \left(p-\frac{1}{2}+4\right)}{\Gamma
   \left(\frac{p+1}{2}+4\right)^2}
>1\quad\text{for }p\in(-\tfrac12,\infty)\setminus\{2\}; \\
	r_3(p)&:=
	2^{p-\frac{1}{2}} \prod _{j=0}^1
   \frac{\left(\frac{p+1}{2}+j\right)^2}{\left(\frac{p}{2}+\frac{1}{4}
   +j\right) \left(\frac{p}{2}+\frac{3}{4}+j\right)}; \\
\tr_3(p)&:=\frac{\Gamma \left(\frac{p}{2}+\frac{1}{4}+2\right) \Gamma
   \left(\frac{p}{2}+\frac{3}{4}+2\right)}{\Gamma
   \left(\frac{p+1}{2}+2\right)^2}
>1\quad\text{for }p\in(-\tfrac12,\infty);
\end{align*}
the Legendre duplication formula, whereby 
$\Ga(\tfrac12)\Ga(p+\tfrac12)$ in \eqref{eq:r>} can be rewritten as 
$2^{p-\frac{1}{2}} \Gamma \left(\tfrac{p}{2}+\tfrac{1}{4}\right) 
\Gamma\left(\tfrac{p}{2}+\tfrac{3}{4}\right)$, 
was used here only to obtain \eqref{eq:r=} for $i=3$. 
Thus, on $(-\tfrac12,\infty)\setminus\{0,2\}$, the functions $\tr_1,\tr_2,\tr_3$ are $>1$ and hence $r>\max[r_1,r_2,r_3]$  
\big(at that, the lower bound $r_1$ for $r$ is good enough in a neighborhood of $0$, with equality $r=r_1$ at $0$; similarly, $r_2$ works in a neighborhood of $2$; and $r_3$ works elsewhere\big). 
So, by \eqref{eq:r=}, to prove the inequality in \eqref{eq:r>} it suffices to show that 
$\max[r_1(p),r_2(p),r_3(p)]>1+\tfrac{p^2}2$ for all $p\in(-\tfrac12,\infty)\setminus\{0,2\}$. 
But this can be done completely algorithmically, since 
the functions $r_1$ and $r_2$ are rational, while $r_3$ is the product of an exponential function and a rational one. 
Such an algorithm is implemented in Mathematica 7 via the command \verb9Reduce9. The Mathematica input  
\begin{verbatim}
Simplify[
Reduce[(r1[p]>1+p^2/2||r2[p]>1+p^2/2||r3[p]>1+p^2/2) && p>-1/2]
]
\end{verbatim}

\vspace*{-1pt}

\noindent\big(with \verb9r1[p]9, \verb9r2[p]9, \verb9r3[p]9 standing for $r_1(p),r_2(p),r_3(p)$ and \verb9||9 representing ``or''\big)
results (in under 1 sec on a standard Core 2 Duo laptop) in \\ 
\verb9-(1/2) < p < 0 || 0 < p < 2 || p > 29.
\end{proof}

\begin{proof}[Proof of Theorem~\ref{th:are schur2}]
This follows by Propositions~\ref{prop:a_p2 schur2} and \ref{prop:are=a_p2}. 
\end{proof}

\begin{proof}[Proof of Theorem~\ref{th:are}]
This follows by Propositions~\ref{prop:sARE} and \ref{prop:are=a_p2}. 
\end{proof}

\begin{proof}[Proof of Proposition~\ref{prop:u range}]\ 

\framebox{(I)} By \eqref{eq:g_p} and Lemma~\ref{lem:lla}(iii), 
$g_p(\sqrt u)$ is concave in $u\in[0,\infty)$ for each $p\in(-\tfrac12,0)$. The same conclusion holds for $p\in[0,2)$, in which case it can be verified directly, using the simpler form of $g_p$. So, by Lemma~\ref{lem:hardy}, $\sum g_p(s_j)$ is Schur${}^2$-concave in $\s$. 
Now it follows from the definition \eqref{eq:orl} of $\no{\cdot}_{p,2}$ that $\no{\cdot}_{p,2}$ is Schur${}^2$-concave. Indeed, if $\vv\in\R^d$, $\ww\in\R^d$, and $\ww^2\preceq\vv^2$, then 
$\sum g_p(w_j/v)\ge\sum g_p(v_j/v)$ for all $v\in(0,\infty)$; so, letting $S_\vv$ denote the set on the right-hand side of \eqref{eq:orl}, one has $S_\ww\subseteq S_\vv$ and hence $\no{\ww}_{p,2}=\inf S_\ww\ge\inf S_\vv=\no{\vv}_{p,2}$. 

\framebox{(II)} 
By part (I) of Proposition~\ref{prop:u range}, the maximum and minimum of $\no{\uu}_{p,2}$ over all 
$\no\cdot_2$-unit vectors $\uu\in\R^d$ are attained, respectively, when $\uu=\one$ and $\uu=\sqrt d\ee_1$. 

For the vector $\uu=\one=(1,\dots,1)$, its ``norm'' $\no{\uu}_{p,2}$ in the case $p\in(-\tfrac12,0)$ is the unique solution $u\in(0,\infty)$ of the equation $d\,f_p(\frac1u)=K_pd^{1/2}$, so that, in view of Lemma~\ref{lem:lla}(i,ii), $\frac1u\to0$ and $d\,(\frac1u)^2\asymp d^{1/2}$, whence $\no{\uu}_{p,2}=u\asymp d^{1/4}$. The same conclusion holds for $p=0$ and $p\in(0,2)$, in which cases it is only easier to verify. 
Thus, \eqref{eq:max} is proved. 

Consider now the subcases $p\in(-\tfrac12,0)$, $p=0$, and $p\in(0,2)$ for the vector $\uu=\sqrt d\ee_1$. 
Let first $p\in(-\tfrac12,0)$. Then for all large enough $d$ and all  $u\in(0,\infty)$ one has 
$\sum g_p(\frac{u_j}u)=g_p(\frac{\sqrt d}u)\le K_pd^{1/2}$, since the function $g_p=f_p$ is bounded for $p\in(-\tfrac12,0)$; hence, $\no{\uu}_{p,2}=0$. 
Next, if $p\in[0,2)$, then the ``norm'' $\no{\uu}_{p,2}$ of the vector $\uu=\sqrt{d}\ee_1$ is the unique solution $u\in(0,\infty)$ of the equation $g_p(\frac{\sqrt d}u)=K_pd^{1/2}$, so that $\ln\frac{\sqrt d}u\asymp d^{1/2}$ if $p=0$ and $(\frac{d^{1/2}}u)^p\asymp d^{1/2}$ if $p\in(0,2)$. 
Now \eqref{eq:min} follows as well. 

\framebox{(III)} For each $p\in[2,\infty)$, part (III) of Proposition~\ref{prop:u range} follows because $|u|^p=(u^2)^{p/2}$, and $v^{p/2}$ is convex in $v\in[0,\infty)$. For $p=\infty$, the Schur${}^2$-convexity of $\no{\uu}_p$ in $\uu$ follows immediately by the definition.  

\framebox{(IV)} Part (IV) of Proposition~\ref{prop:u range} follows immediately from part (III). 
\end{proof}

\begin{proof}[Proof of Theorem~\ref{th:are range}]
This follows by Propositions~\ref{prop:a_p2 range} and \ref{prop:are=a_p2}. 
\end{proof}

\begin{proof}[Proof of Proposition~\ref{prop:random u}]
Suppose that indeed $p\in(2,\infty)$ and $\uu$ is a completely random $\no{\cdot}_2$-unit vector. Then $\uu$ equals in distribution the vector $\sqrt d\,\Zd/\|\Zd\|$. So, for any $\vp>0$ the probability of the event $\big\{\no{\uu}_p<\vp d^{(p-2)/(4p)}\big\}$ is the same as that of the event $\big\{\sum|Z_j|^p/(\sum Z_j^2)^{p/2}<\vp^p d^{(2-p)/4}\big\}$, which tends to $1$ as $d\to\infty$ -- since, by the law of large numbers,  $\frac1d\sum|Z_j|^p\to\E|Z|^p$ and $\frac1d\sum Z_j^2\to1$ in probability as $d\to\infty$. 
\end{proof}

\section{Statements of lemmas}\label{lemmas}
In this section, we shall state a few lemmas, which will be used in the proofs of propositions stated in Section~\ref{proofs}. The proofs of the lemmas will be deferred to Section~\ref{proofs of lemmas}. 

\begin{lemma}\label{lem:vpi}
For any two varying vectors $\s$ and $\vv$ in $\R^d$ and any varying $v>0$ one has the implication
\begin{equation*}
\big(v>\max_{j=1}^d|v_j|\ \&\ v\to0\big) 
\implies\sum\vpi(s_j+v_j)=\sum\vpi(s_j)(1+o(1))+O(d\,e^{-1/v}). 	
\end{equation*}
\end{lemma}

\begin{lemma}\label{lem:lla}
The following statements take place. 
\begin{enumerate}[(i)]
\item The expression $\la_p(s)$ is strictly and continuously increasing \big(respectively, decreasing\big) in $|s|$ for each $p\in(0,\infty)$ \big(respectively, for each $p\in(-1,0)$\big). 
\item 
For each $p\in(-1,\infty)$  
\begin{gather}
\label{eq:la_p(s)-la_p, s to0}
	\la_p(s)-\la_p(0)\sim\tfrac p2\,\la_p(0)\,s^2 \text{ as $s\to0$}; \\
\label{eq:la_p(s)-la_p, s to infty}
	\la_p(s)=|s|^p\big(1+O(|s|^{-2})\big)\text{ over $|s|>1$}.		  
\end{gather}
\item 
$\frac{\la'_p(s)}s$ increases in $s\in(0,\infty)$ for each $p\in(-1,0)\cup(2,\infty]$ and decreases in $s\in(0,\infty)$ for each $p\in(0,2)$; 
the same monotonicity patterns hold for $\frac{\la_p(s)-\la_p(0)}{s^2}$; 
moreover, $\la_p(\sqrt{t})$ is convex in $t\in[0,\infty)$ for each $p\in(-1,0)\cup(2,\infty]$ and concave in $t\in[0,\infty)$ for each $p\in(0,2)$. 
\item 
For each $p\in[2,\infty)$,  
\begin{equation}\label{eq:la_p(s)-la_p>}
	\la_p(s)-\la_p(0)\ge\tfrac p2\,\la_p(0)\,s^2 \text{ for all $s\in\R$},
\end{equation}
and this inequality is strict if $p>2$ and $s\ne0$.  
\item 
For each $p\in(-\frac12,\infty)$ and over all $s\in\R\setminus\{0\}$,  
\begin{equation}\label{eq:la_2p(s)-la_2p}
	\big|\la_{p,2}(s)-\la_{p,2}(0)\big|\OO s^2\ii{|s|\le1}+(1+|s|^{2p-2})\ii{|s|>1}. 
\end{equation}
\item 
For each $p\in[0,\infty)$ and over all $s\in\R\setminus\{0\}$, 
\begin{equation}\label{eq:la_3p(s)}
	\la_{p,3}(s)\OO1+|s|^{3p-1}\ii{|s|>1}. 
\end{equation}
\end{enumerate}
\end{lemma}

\begin{lemma}\label{lem:ttmu}
The following statements take place. 
\begin{enumerate}[(i)]
\item 
For each $d\in\N$, the expression $\tmu_d(s)$ strictly and continuously decreases in $|s|$ from $\tmu_d(0)$ to $0$. 
\item Moreover, $\tmu_d(0)\sim2\vpi(0)\ln d$ and $\tmu_d(s)\OO\ln d$ over all $s\in\R$. 
\item Over all $s\in\R$ and $d\in\{2,3,\dots\}$ 
\begin{gather}
\label{eq:tmu(0)-tmu(s)}
	\tmu_d(0)-\tmu_d(s)\asymp(s^2\wedge1)\ln d\asymp\big(\vpi(0)-\vpi(s)\big)\ln d. 
\end{gather}
\item
Over all $s\in\R$ and $d\in\N$, 
\begin{equation*}
	\Big|\frac{sf'(s)}{f(s)}\Big|\OO1,
\end{equation*}
for $f(s):=f_{-1}(s)=\tmu_d(0)-\tmu_d(s)$. 
\end{enumerate}
\end{lemma}

\begin{lemma}\label{lem:ttla}
The following statements take place. 
\begin{enumerate}[(i)]
\item The expression $\tla(s)$ is strictly and continuously increasing in $|s|$. 
\item 
One has   
\begin{gather}
\label{eq:tla(s)-tla, s to0}
	\tla(s)-\tla(0)\sim\tfrac{s^2}2\quad\text{as }s\to0; \\
\intertext{moreover, for each $m\in\N$ and over all $s\in[e,\infty)$}	
\label{eq:tla(s), s to infty}
	\E\ln^m|Z+s|=\ln^m s\,[1+O(\ln^{-m}s)]\sim\ln^m s,	  
\end{gather}
the latter relation taking place as $s\to\infty$. 
\item Expressions
$\frac{\tla'(s)}s$ and hence $\frac{\tla(s)-\tla(0)}{s^2}$ decrease in $s\in(0,\infty)$; 
moreover, $\tla(\sqrt t)$ is concave in $t\in[0,\infty)$.  
\item 
Over all $s\in\R$, 
\begin{equation*}
	\big|\tla_2(s)-\tla_2(0)\big|\OO g_0(s),  
\end{equation*}
where $g_0$ is as in Definition~\ref{def:orl} (for $p=0$). 
\item Moreover, $\tla_2(0)=\tfrac{\pi^2}8$.  
\item 
Over all $s\in\R\setminus\{0\}$, 
\begin{equation*}
	\tla_3(s)\OO1+\ln^2s\ii{|s|>e}.  
\end{equation*}
\item
Over all $s\in\R$, 
\begin{equation*}
	|s\tla'(s)|\OO1.  
\end{equation*}
\end{enumerate}
\end{lemma}

\begin{lemma}\label{lem:h}
For each $p\in\R$, 
\begin{equation*}
	f_p(s)\asymp h_p(s):=h_{p,d}(s):=
\left\{
\begin{alignedat}{2}
& \big(\vpi(0)-\vpi(s)\big)\ln d && \text{\quad if }p=-1, \\
& \vpi(0)-\vpi(s) && \text{\quad if }p\in(-1,0), \\
& g_p(s) && \text{\quad if }p\in[0,2), \\
& s^2+|s|^p && \text{\quad if }p\in[2,\infty), 
\end{alignedat}
\right.	
\end{equation*}
over all $s\in\R$ (and, for $p=-1$, over all $d\in\N$), with 
$f_p$ and $g_p$ defined in Theorem~\ref{th:AS iff} and \eqref{eq:g_p}, respectively. 
\end{lemma}

\begin{lemma}\label{lem:s,v}
For any $p\in(-\frac12,0)$ and any varying vectors $\s$ and $\vv$ of the same direction such that 
$\sum g_p(s_j)\sim S\sim\sum g_p(v_j)$ for some varying $S\in(0,\infty)$, one has 
\begin{equation}\label{eq:s,v,S}
	\sum s_j^2>>S\iff\sum v_j^2>>S; 
\end{equation}
here, $g_p=f_p$, as in Definition~\ref{def:orl} \big(for $p\in(-\frac12,0)$\big). 
\end{lemma}

\begin{lemma}\label{lem:s,v,ge0}
\emph{[Cf.\ Lemma~\ref{lem:s,v}.]}\quad 
For any $p\in[0,2)$ and any varying vectors $\s$ and $\vv$ of the same direction such that 
$\sum g_p(s_j)\asymp S\asymp\sum g_p(v_j)$ for some varying $S\in(0,\infty)$, one has 
\begin{equation}\label{eq:s,v,S,ge0}
	\sum s_j^2>>S\iff\sum v_j^2>>S; 
\end{equation}
here, $g_p$ is as in Definition~\ref{def:orl} \big(for $p\in[0,2)$\big). 
\end{lemma}

\begin{lemma}\label{lem:nu}
\emph{\big(\cite[Theorem~4, Chapter~4]{gned-kolm}, \cite[Theorem~15.28]{kallenberg}\big)}\quad  
Let $(\xi_{dj}):=(\xi_{dj}\colon d\in\N,\ j=1,\dots,d)$ be a family of r.v.'s such that $\xi_{d1},\dots,\xi_{dd}$ are independent, for each $d\in\N$. Assume also that this double-indexed family is a null array; that is, for every $\vp>0$, 
\begin{equation}\label{eq:null}
	\max_j\PP(|\xi_{dj}|>\vp)\to0
\end{equation}
as $d\to\infty$. 
Let $\zeta$ be a r.v.\ with the infinitely divisible distribution with characteristics $a,b,\nu$ such that $\nu(\{-1,1\})=0$ (recall \eqref{eq:levy}). 
Then
the sum $\sum\xi_{dj}$ converges in distribution to $\zeta$ if and only if the following three conditions hold:
\begin{enumerate}[(I)]
	\item $\sum\PP(\xi_{dj}>x)\to\nu\big((x,\infty)\big)$ for each $x>0$ and \\ $\sum\PP(\xi_{dj}<x)\to\nu\big((-\infty,x)\big)$ for each $x<0$;
	\item $\sum\Var\xi_{dj}\ii{|\xi_{dj}|\le1}\to\ta:=a+\int_{\R}x^2\ii{|x|\le1}\,\nu(\dd x)$;
	\item $\sum\E\xi_{dj}\ii{|\xi_{dj}|\le1}\to b$. 
\end{enumerate} 
\end{lemma}

\begin{lemma}\label{lem:hardy}
\emph{\big(the Hardy-Littlewood-Polya theorem -- see e.g.\ \cite[4.B.1]{marsh-ol}\big)}\quad 
If a real-valued function $f$ is convex then $\sum f(s_j)$ is Schur-convex in $\s$. 
\end{lemma}

\begin{lemma}\label{lem:schur2}
\emph{\cite{schur2}}\quad 
For any given $p\in[-\infty,\infty]$, $c\in\R$, and $d\in\N$, the function 
$
\R^d\ni\vv\longmapsto\PP(\no{\ZZ+\vv}_p>c) 	
$ 
is Schur${}^2$-concave for $p\in[-\infty,2]$ and Schur${}^2$-convex for $p\in[2,\infty]$. 
\end{lemma}

\begin{lemma}\label{lem:mono}
\emph{\cite[Lemma~3.1]{schur2}}
For any given $p\in[-\infty,\infty]$, $\vv\in\R^d\setminus\{\0\}$, $c\in(0,\infty)$, and $d\in\N$, the function 
$
[0,\infty)\ni t\longmapsto\PP(\no{\ZZ+t\vv}_p>c) 	
$ 
is continuously and strictly increasing. 
\end{lemma}


\section{Proof of Propositions~\ref{prop:AS shift iff} and \ref{prop:AS shift}}\label{proof of shifts}
The main content of this section is the proof of Proposition~\ref{prop:AS shift iff}, in the course of which  Proposition~\ref{prop:AS shift} will also be proved. 

\subsection{Case: $p=-\infty$}
By Definitions~\ref{def:AS shift} and \eqref{eq:means excep}, 
a varying vector $\s$ will be a $(-\infty)$-weakly-AS shift if and only if 
\begin{align}
	\al&\approx\PP\Big(\min_1^d|Z_j|>c\Big)\quad\text{and} \label{eq:al,-infty}\\
	\be&\approx\PP\Big(\min_1^d|Z_j+s_j|>c\Big) \label{eq:be,-infty}
\end{align}
for some varying $c$. 

In turn, \eqref{eq:al,-infty} can be rewritten as $\al\approx\PP(|Z|>c)^d$, whence $\PP(|Z|>c)=(\al+o(1))^{1/d}
\to1$, $c\to0+$, $\PP(|Z|>c)=1-\PP(|Z|\le c)
=1-2\vpi(0)c(1+o(1))$, $-2\vpi(0)c\sim\ln\PP(|Z|>c)=\frac1d\ln(\al+o(1))\sim\frac1d\ln\al$, and  $c\sim-\frac1d\,\frac{\ln\al}{2\vpi(0)}$; vice versa, the latter relation implies \eqref{eq:al,-infty}. 
Thus, 
\begin{equation}\label{eq:-infty,c}
	\text{\eqref{eq:al,-infty}}\Longleftrightarrow c\sim-\frac1d\,\frac{\ln\al}{2\vpi(0)}.  
\end{equation}
With such a varying $c$, the condition 
\eqref{eq:be,-infty} can be rewritten as 
\begin{align}
	-\ln\be\approx-\sum\ln\PP(|Z_j+s_j|>c) 
	&=-\sum\ln\big(1-\PP(|Z_j+s_j|\le c)\big) \notag\\
	&\sim\sum\PP(|Z_j+s_j|\le c) \label{eq:-infty,sim1}\\
	&=2c\sum\vpi(s_j+v_j) \notag\\
	&\approx-\frac1d\,\frac{\ln\al}{\vpi(0)}\sum\vpi(s_j), \label{eq:-infty,sim2}
\end{align}
with some $v_j$'s such that $\max|v_j|<c$; relation \eqref{eq:-infty,sim2} follows in view of  \eqref{eq:-infty,c} and Lemma~\ref{lem:vpi}, whereas \eqref{eq:-infty,sim1} follows because $\max_j\PP(|Z_j+s_j|\le c)\OO c\to0$.   
 
Thus, a varying vector $\s=(s_1,\dots,s_d)$ will be $(-\infty)$-weakly-AS iff it is $(-\infty)$-strongly-AS iff 
$-\ln\be\sim-\frac1d\,\frac{\ln\al}{\vpi(0)}\sum\vpi(s_j)=-\frac{\ln\al}d\,\sum e^{-s_j^2/2}$. 
This completes the proof of Propositions~\ref{prop:AS shift iff} and \ref{prop:AS shift} -- in the case when $p=-\infty$. 

\medskip
\hrule
\medskip

The proof of Propositions~\ref{prop:AS shift iff} and \ref{prop:AS shift} for $p\in\R$ will depend on the case according to which subset of $\R$ the value of $p$ belongs to. However, the general scheme in all these cases will be the same. Namely, in each of these cases, first we shall take any varying vector $\s$ in $\R^d$ (with $d\to\infty$) satisfying a certain special condition \big(referred to in this paragraph as condition (*)\big), which latter depends on the set of values of $p$ under consideration. This condition (*) will in each case be implied by the necessary and sufficient condition \eqref{eq:as iff}. On the other hand, (*) will imply that the vector $\s$ is in a sense small enough, so that the distribution of $\no{\ZZ+\s}_p$ can be approximated \big(for $p\in(-\infty,-\frac12)$\big) by means of the stable distribution introduced in Definition~\ref{def:zeta} or \big(for $p\in[-\frac12,\infty)$\big) by means of the standard normal distribution. 
This will allow us to show that any $p$-weakly-AS shift $\s$ satisfying the special condition (*) will also satisfy the necessary and sufficient condition \eqref{eq:as iff}. 
On the other hand, since \eqref{eq:as iff} implies (*), it will also imply the approximation by the appropriate limit distribution, which in turn will be used to show that -- under condition \eqref{eq:as iff} -- the varying vector $\s$ must be a $p$-strongly-AS shift. In each case, it will then remain to prove that any $p$-weakly-AS shift $\s$ must be small enough so as to satisfy the condition (*). This is done by assuming the contrary, that $\s$ is too large so as to violate (*) and hence to violate \eqref{eq:as iff}; then we observe that, for any given $\tbe\in(\be,1)$, one can shrink $\s$ to $t\s$ for some varying $t\in(0,1)$ so that \eqref{eq:as iff} hold for $t\s$ and $\tbe$ in place of $\s$ and $\be$, respectively; by what has been proved, this would imply that $t\s$ is a $p$-strongly-AS shift, with $\tbe$ in place of $\be$, which in turn would contradict the fact that  $\PP(\no{\ZZ+t\s}_p>c)$ is increasing in $t>0$. 

\subsection{Case: $p\in(-\infty,-1)$}\label{p in(-infty,-1)}
Take any varying vector $\s$ (in $R^d$, with $d\to\infty$). 
To begin, assume the mentioned special condition (*) on the ``smallness'' of $\s$, which we let here take the form 
\begin{equation}\label{eq:assum,p in(-infty,-1)}
	\sum\vpi(s_j)\OOG d, 
\end{equation}
and use Lemma~\ref{lem:nu} with 
\begin{equation*}
	\xi_{dj}:=\frac{|Z_j+s_j|^p}{\si(\s)^p}, 
\end{equation*}
where
\begin{equation}\label{eq:si,p in(-infty,-1)}
	\si(\s):=\frac{\vpi(0)}{\sum\vpi(s_j)},
\end{equation}
so that, in view of \eqref{eq:assum,p in(-infty,-1)}, 
\begin{equation}\label{eq:si<1/d,p in(-infty,-1)}
	\si(\s)\OO\tfrac1d<<1. 
\end{equation}
Note that for any $\vp\in(0,\infty)$
\begin{equation*}
	\max_j\PP(|\xi_{dj}|>\vp)=\max_j\PP\big(|Z_j+s_j|<\vp^{1/p}\si(\s)\big)
	\OO\vp^{1/p}\si(\s)\to0,
\end{equation*}
so that the null-array condition \eqref{eq:null} is satisfied. 

Next, for any fixed $x\in(0,\infty)$ there are some varying $v_j$'s such that $\max_j|v_j|<x^{1/p}\si(\s)\to0$ and 
\begin{align*}
	\sum\PP(\xi_{dj}>x)&=\sum\PP\big(|Z_j+s_j|<x^{1/p}\si(\s)\big) \\
&=2x^{1/p}\si(\s)\sum\vpi(s_j+v_j)\to2x^{1/p}\vpi(0)
=\sqrt{\tfrac2\pi}\,x^{1/p},
\end{align*}
in view of Lemma~\ref{lem:vpi} and conditions~\eqref{eq:si,p in(-infty,-1)} and \eqref{eq:assum,p in(-infty,-1)}. 
It is also clear that $\sum\PP(\xi_{dj}<x)=0$ for any $x\in(0,\infty)$, since the r.v.'s $\xi_{dj}$ are nonnegative. 
Thus, condition (I) in Lemma~\ref{lem:nu} is verified for $\nu=\sqrt{\frac2\pi}\,\nu_p$, as in Definition~\ref{def:zeta}. 

Let us now verify condition (III) in Lemma~\ref{lem:nu}. One has 
\begin{align*}
	\sum\E\xi_{dj}\ii{|\xi_{dj}|\le1}&=E_1(p)\\
	&:=\si(\s)^{-p}\sum\E|Z_j+s_j|^p \ii{|Z_j+s_j|\ge\si(\s)} \\
	&=\si(\s)^{-p}\sum\int_{\si(\s)}^\infty u^p[\vpi(s_j-u)+\vpi(s_j+u)]\,\dd u \\
	&= \si(\s)^{-p}(E_{11}+E_{12}),
\end{align*}
where
\begin{align*}
		E_{12}&:=\sum\int_{d^{-1/2}}^\infty u^p[\vpi(s_j-u)+\vpi(s_j+u)]\,\dd u \\
		&\OO d\,\int_{d^{-1/2}}^\infty u^p\,\dd u
		\OO d^{(1-p)/2}<<\si(\s)^p
\end{align*}
by \eqref{eq:si<1/d,p in(-infty,-1)}, 
and 
\begin{align}
	E_{11}:=\sum\int_{\si(\s)}^{d^{-1/2}} \dots 
	&=\sum 2\vpi(s_j+v_j)\int_{\si(\s)}^{d^{-1/2}}u^p\,\dd u \notag\\
	&\sim\sum 2\vpi(s_j+v_j)\frac{\si(\s)^{p+1}}{-(p+1)} \label{eq:E11-1,p in(-infty,-1)}\\
	&\sim\frac2{-(p+1)} \si(\s)^p	\vpi(0) \label{eq:E11-p,2 in(-infty,-1)} 
\end{align}
for some $v_j$'s such that $\max_j|v_j|<d^{-1/2}\to0$; 
relation \eqref{eq:E11-1,p in(-infty,-1)} takes place by \eqref{eq:si<1/d,p in(-infty,-1)}, while relation \eqref{eq:E11-p,2 in(-infty,-1)} holds in view of Lemma~\ref{lem:vpi} and conditions~\eqref{eq:si,p in(-infty,-1)} and \eqref{eq:assum,p in(-infty,-1)}.  
Collecting these estimates for $E_{11}$ and $E_{12}$ into $E_1(p)$, one completes the verification of condition (III) in Lemma~\ref{lem:nu}, with $b=b_p$ as in \eqref{eq:b}. 

Next, let us verify condition (II) in Lemma~\ref{lem:nu}. Since $p\in(-\infty,-1)$ implies $2p\in(-\infty,-1)$, one can use the above estimate for $E_1(p)$ to immediately obtain the following:
\begin{equation*}
	E_2:=E_2(p):=\sum\E\xi_{dj}^2\ii{|\xi_{dj}|\le1}=E_1(2p)\to\frac2{-(2p+1)}\,\vpi(0). 
\end{equation*}
Also, 
\begin{equation*}
	E_3:=\sum\E^2\xi_{dj}\ii{|\xi_{dj}|\le1}
	\OO \si(\s)^{-2p}\cdot d\cdot\Big(\int_{\si(\s)}^\infty u^p\,\dd u\Big)^2
	\OO d\,\si(\s)^2\to0,
\end{equation*}
by \eqref{eq:si<1/d,p in(-infty,-1)}. 
So, 
\begin{equation*}
\sum\Var\xi_{dj}\ii{|\xi_{dj}|\le1}=E_2-E_3\to\frac{2\vpi(0)}{-(2p+1)}
=\int_{\R}x^2\ii{|x|\le1}\,\nu(\dd x)	
\end{equation*}
for $\nu=\sqrt{\frac2\pi}\,\nu_p$, as in Definition~\ref{def:zeta}. 
So, condition (II) in Lemma~\ref{lem:nu} holds with $a=0$. 
We conclude that -- under condition \eqref{eq:assum,p in(-infty,-1)} -- 
\begin{equation}\label{eq:D,p in(-infty,-1)}
		\si(\s)^{-p}\sum|Z_j+s_j|^p=\sum\xi_{dj}\overset{\mathrm{D}}{\longrightarrow}\zeta_{p,b_p},
\end{equation}
where $\overset{\mathrm{D}}{\to}$ denotes the convergence in distribution; since the d.f.\ of $\zeta_{p,b_p}$ is continuous (by the remark right after Definiton~\ref{def:zeta}), it follows that the d.f.\ of $\si(\s)^{-p}\sum|Z_j+s_j|^p$ converges to the d.f.\ of $\zeta_{p,b_p}$ uniformly on $\R$.  

 
Take now any $p$-weakly-AS shift $\s$ satisfying condition \eqref{eq:assum,p in(-infty,-1)}. Then, for some varying $c$, 
\begin{align}
	\al&\approx\PP\Big(\si(\0)^{-p}\sum|Z_j|^p<c\Big) \label{eq:al,p in(-infty,-1)}\\
	&\approx\PP(\zeta_{p,b_p}<c), \label{eq:al,p in(-infty,-1) a}\\ 
	\be&\approx\PP\Big(\si(\s)^{-p}\sum|Z_j+s_j|^p<c\frac{\si(\s)^{-p}}{\si(\0)^{-p}}\Big) 
	\label{eq:be,p in(-infty,-1)}\\
	&\approx\PP\Big(\zeta_{p,b_p}<c\frac{\si(\s)^{-p}}{\si(\0)^{-p}}\Big), \label{eq:be,p in(-infty,-1) a}
\end{align}
whence $c\to\Phi_{p,b_p}^{-1}(\al)$ and 
$\si(\s)^{-p}/\si(\0)^{-p}\to\Phi_{p,b_p}^{-1}(\be)/\Phi_{p,b_p}^{-1}(\al)$, which latter is equivalent to \eqref{eq:as iff}
\big(for $p\in(-\infty,-1)$\big). 

On the other hand, one can see that \eqref{eq:as iff} \big(for $p\in(-\infty,-1)$\big) implies that $\s$ is a $p$-strongly-AS shift. 
Indeed, 
\eqref{eq:as iff} implies \eqref{eq:assum,p in(-infty,-1)}, so that \eqref{eq:D,p in(-infty,-1)} holds; therefore, relations 
\eqref{eq:al,p in(-infty,-1) a} and \eqref{eq:be,p in(-infty,-1) a} hold; 
hence and because the d.f.\ $\Phi_{p,b_p}$ of $\zeta_{p,b_p}$ is strictly increasing in a neighborhood of $\Phi_{p,b_p}^{-1}(\al)$, for any varying $c$ satisfying relation \eqref{eq:al,p in(-infty,-1)} one will have $c\to\Phi_{p,b_p}^{-1}(\al)$; since \eqref{eq:as iff} is equivalent to $\si(\s)^{-p}/\si(\0)^{-p}\to\Phi_{p,b_p}^{-1}(\be)/\Phi_{p,b_p}^{-1}(\al)$, relation \eqref{eq:be,p in(-infty,-1)} will then follow. 
%

%

It remains to consider the case when $\s$ is a $p$-weakly-AS shift for which condition \eqref{eq:assum,p in(-infty,-1)} fails. Then without loss of generality (w.l.o.g.) $\sum\vpi(s_j)<<d$ or, equivalently, $\sum f_p(s_j)=\sum e^{-s_j^2/2}<<d$, where $f_p$ is as in \eqref{eq:as iff} \big(for $p\in(-\infty,-1)$\big). Take any $\tbe\in(\be,1)$. 
Note that $\sum f_p(ts_j)$ continuously decreases in $t\in[0,1]$ from $d$ to $o(d)$. 
Note also that $K_{\al,\tbe;p}\in(0,1)$. Hence, for some varying $t\in(0,1)$ relation \eqref{eq:as iff} holds with $t\s$ and $\tbe$ in place of $\s$ and $\be$, respectively. 
So,  
by 
the proved above implication ``\eqref{eq:as iff} implies that $\s$ is a $p$-strongly-AS shift'', one concludes that $t\s$ is a $p$-strongly-AS shift, but for $\tbe$ in place of $\be$. On the other hand, \eqref{eq:al,p in(-infty,-1)} and \eqref{eq:be,p in(-infty,-1)} take place for some varying $c$, since $\s$ is a $p$-weakly-AS shift. 
Thus, for such $c$ 
%
one has $\PP(\no{\Zd+\s}_p>c)\to\be$ and $\PP(\no{\Zd+t\s}_p>c)\to\tbe$, which is a contradiction, since $\be<\tbe$ while, by Lemma~\ref{lem:mono}, $\PP(\no{\Zd+\s}_p>c)\ge\PP(\no{\Zd+t\s}_p>c)$.   
This concludes the proof of Proposition~\ref{prop:AS shift iff} \big(in the case when $p\in(-\infty,-1)$\big). 

\subsection{Case: $p\in(-1,-\frac12)$}\label{p in(-1,-frac12)}
Take any varying vector $\s$. 
To begin, assume that 
\begin{equation}\label{eq:assum,p in(-1,-frac12)}
	\sum[\vpi(0)-\vpi(s_j)]\OO \si^p,
\end{equation}
where
\begin{equation*}
	\si:=\tfrac1d, 
\end{equation*}
and use Lemma~\ref{lem:nu} with 
\begin{equation*}
	\xi_{dj}:=\frac{|Z_j+s_j|^p-\la_p(s_j)}{\si^p}.  
\end{equation*}
Note that for each fixed $x\in(0,\infty)$ 
\begin{equation*}
	\xi_{dj}>x
	\iff|Z_j+s_j|^p>\la_p(s_j)+\si^p x
	\iff|Z_j+s_j|<\si[x+\si^{-p}\la_p(s_j)]^{1/p}; 
\end{equation*}
also, 
\begin{equation*}
	\si[x+\si^{-p}\la_p(s_j)]^{1/p}\sim\si x^{1/p}
\end{equation*}
uniformly in $j$, since $\si^{-p}=1/d^{|p|}\to0$ and $\la_p(s)\le\la_p(0)\OO1$ for all $s\in\R$, by part (i) of Lemma~\ref{lem:lla}.  
Similarly, 
for each fixed $x\in(-\infty,0)$ 
\begin{equation*}
	\xi_{dj}<x
	\iff|Z_j+s_j|^p<\la_p(s_j)+\si^p x, 
\end{equation*}
which is impossible for any $j$ if $d$ is large enough, since $\si^p=d^{|p|}\to\infty$ while $\la_p(s_j)\le\la_p(0)$.  

In particular, it follows that for any fixed $\vp\in(0,\infty)$ 
\begin{equation*}
	\max_j\PP(|\xi_{dj}|>\vp)=\max_j\PP\big(|Z_j+s_j|<\si(\vp+o(1))^{1/p}\big)
	\OO\si\to0,
\end{equation*}
so that the null-array condition \eqref{eq:null} is satisfied. 

Next, for any fixed $x\in(0,\infty)$ there are some varying $v_j$'s such that $\max_j|v_j|\OO\si\to0$ and 
\begin{equation}
\begin{aligned}\label{eq:->nu}
	\sum\PP(\xi_{dj}>x)&=\sum\PP\big(|Z_j+s_j|<\si(x+o(1))^{1/p}\big) \\
&=2\si\,(x+o(1))^{1/p}\sum\vpi(s_j+v_j) \\
&\sim2\si\, x^{1/p}\Big(\sum\vpi(s_j)+o(d)\Big) \\
&\to2x^{1/p}\vpi(0)
\end{aligned}	
\end{equation}
by Lemma~\ref{lem:vpi} and condition \eqref{eq:assum,p in(-1,-frac12)}, which latter implies 
$\sum\vpi(s_j)=d\,\vpi(0)+O(d^{|p|})\sim d\,\vpi(0)$. 
Also, $\sum\PP(\xi_{dj}<x)=0$ for any $x\in(-\infty,0)$ and all large enough $d$. 
Thus, condition (I) in Lemma~\ref{lem:nu} is verified for $\nu=\sqrt{\frac2\pi}\,\nu_p$, as in Definition~\ref{def:zeta}. 

Let us now verify condition (III) in Lemma~\ref{lem:nu}. From considerations above, it follows that 
\begin{equation}\label{eq:xi,iff}
		|\xi_{dj}|>1
	\iff|Z_j+s_j|<\si_j:=\si[1+\si^{-p}\mu_j]^{1/p}\sim\si=\tfrac1d, 
\end{equation}
where
\begin{equation*}
	\mu_j:=\la_p(s_j)\le\la_p(0). 
\end{equation*}
Hence and because $\E\xi_{dj}=0$,   
\begin{align*}
	E_1&:=\sum\E\xi_{dj}\ii{|\xi_{dj}|\le1}=-\sum\E\xi_{dj}\ii{|\xi_{dj}|>1}=\si^{-p}(E_{12}-E_{11}),
\end{align*}
where
\begin{align*}
		E_{12}&:=\sum\mu_j\PP(\xi_{dj}|>1)
		\OO\sum\PP(|Z_j+s_j|<\si_j) 
		\OO\si d=1<<\si^p, \\
		E_{11}&:=\sum\E|Z_j+s_j|^p \ii{|Z_j+s_j|<\si_j}=E_{111}-E_{112}, \\ 
		E_{111}&:=\sum\int_\R|z+s_j|^p \ii{|z+s_j|<\si_j}\vpi(0)\,\dd z
		=d\cdot2\vpi(0)\cdot\frac{\si_j^{p+1}}{p+1}
		\sim\frac{2\vpi(0)}{p+1}\,\si^p, \\ 
E_{112}&:=\sum\int_\R|x|^p \ii{|x|<\si_j}[\vpi(0)-\vpi(x+s_j)]\,\dd x \\
&\OO\sum\int_\R|x|^p \ii{|x|<\si_j}\,\dd x\,[\si_j^2+\vpi(0)-\vpi(s_j)] \\
&\OO \si^{p+1}\sum [\si^2+\vpi(0)-\vpi(s_j)] \\
&\OO \si^{p+1}[d\,\si^2+\si^p]<<\si^p;
\end{align*}
in the latter display, for $E_{112}$, the last of three $\OO$'s is by condition \eqref{eq:assum,p in(-1,-frac12)}, while the first one is obtained as follows: if $|x|<\si_j$ and $|s_j|\le1$ then $\vpi(0)-\vpi(x+s_j)\OO(x+s_j)^2\OO x^2+s_j^2
<\si_j^2+s_j^2\OO\si_j^2+\vpi(0)-\vpi(s_j)$; and if $|s_j|>1$ then $\vpi(0)-\vpi(x+s_j)\OO1\OO\vpi(0)-\vpi(s_j)\le\si_j^2+\vpi(0)-\vpi(s_j)$; thus, $|x|<\si_j$ always implies $\vpi(0)-\vpi(x+s_j)]\OO\si_j^2+\vpi(0)-\vpi(s_j)$. 

We conclude that
\begin{equation*}
	E_1\to-\frac{2\vpi(0)}{p+1}=-\frac{\sqrt{2/\pi}}{p+1}. 
\end{equation*}

Next, 
\begin{equation*}
	E_2:=\sum\E\xi_{dj}^2\ii{|\xi_{dj}|\le1}=\si^{-2p}(E_{21}-2E_{22}+E_{23}),
\end{equation*}
where 
\begin{align*}
		E_{21}&:=\sum\E|Z_j+s_j|^{2p} \ii{|Z_j+s_j|\ge\si_j} \\
		&=\sum\int_{\si_j}^\infty u^{2p}\,[\vpi(u-s_j)+\vpi(u+s_j)]\,\dd u, \\
		E_{22}&:=\sum\mu_j\E|Z_j+s_j|^p \ii{|Z_j+s_j|\ge\si_j}
		\OO d<<\si^{2p} \\
		E_{23}&:=\sum\mu_j^2
		\OO d<<\si^{2p}.
\end{align*}
To estimate $E_{21}$, note first that 
\begin{equation}\label{eq:2nd diff}
\vpi(u-s)+\vpi(u+s)-2\vpi(u)\OO\vpi(0)-\vpi(s)	
\end{equation}
over all real $u$ and $s$. Indeed, if $|s|\le1$ then $\vpi(u-s)+\vpi(u+s)-2\vpi(u)\OO s^2\OO\vpi(0)-\vpi(s)$, and if $|s|>1$ then $\vpi(u-s)+\vpi(u+s)-2\vpi(u)\OO1\OO\vpi(0)-\vpi(s)$. 
Hence, recalling also, once again, condition \eqref{eq:assum,p in(-1,-frac12)}, one has 
\begin{align*}
		E_{21}&=E_{211}+E_{212}, \\
		E_{212}&:=\sum\int_{\si_j}^\infty u^{2p}\,[\vpi(u-s_j)+\vpi(u+s_j)-2\vpi(u)]\,\dd u \\
		&\OO\sum[\vpi(0)-\vpi(s_j)]\int_{\si_j}^\infty u^{2p}\,\dd u
		\OO\si^p\si^{2p+1}<<\si^{2p}, \\
E_{211}&:=\sum\int_{\si_j}^\infty u^{2p}\,2\vpi(u)\,\dd u \sim-\frac{2\vpi(0)\si^{2p}}{2p+1},
\end{align*}
because $0<\vp<<\vp_1<<1$ implies 
\begin{align*}
		\int_\vp^\infty u^{2p}\,\vpi(u)\,\dd u&=\int_\vp^{\vp_1} + \int_{\vp_1}^1 + \int_1^\infty, \\
		\int_\vp^{\vp_1}&\sim\int_\vp^{\vp_1}u^{2p}\,\vpi(0)\,\dd u \sim-\frac{\vpi(0)\vp^{2p+1}}{2p+1},\\
		\int_{\vp_1}^1 + \int_1^\infty&\OO\vp_1^{2p+1}+1<<\vp^{2p+1}. 
\end{align*}
Also, 
\begin{align*}
	E_3&:=\sum\E^2\xi_{dj}\ii{|\xi_{dj}|\le1}=\sum\E^2\xi_{dj}\ii{|\xi_{dj}|>1}\OO\si^{-2p}(E_{31}+E_{33}), \\
		E_{31}&:=\sum\E^2|Z_j+s_j|^p \ii{|Z_j+s_j|<\si_j} 
		\OO d\,\Big(\int_0^{\si_j} u^p\,\dd u\Big)^2
		\OO d\,(\si^{p+1})^2<<\si^{2p}, \\
		E_{32}&:=\sum\mu_j^2\PP(|Z_j+s_j|<\si_j)\OO d\,\si<<\si^{2p}.
\end{align*}

The rest of the proof of Proposition~\ref{prop:AS shift iff} in the case $p\in(-1,-\frac12)$ is similar to that in  
the case $p\in(-\infty,-1)$. Instead of \eqref{eq:assum,p in(-infty,-1)}, here one has to deal with condition \eqref{eq:assum,p in(-1,-frac12)} -- which is implied, in view of Lemma~\ref{lem:h}, by \eqref{eq:as iff} \big(now for $p\in(-1,-\frac12)$\big). 
In this case, instead of \eqref{eq:al,p in(-infty,-1)} and \eqref{eq:be,p in(-infty,-1)} one should write 
\begin{align*}
	\al&\approx\PP\Big(\si^{-p}\sum\big(|Z_j|^p-\la_p(0)\big)<c\Big)\approx\PP(\zeta_{p,b_p}<c), 
	\\ 
	\be&\approx\PP\Big(\si^{-p}\sum\big(|Z_j+s_j|^p-\la_p(0)\big)<c\Big)
	\\
	&=\PP\Big(\si^{-p}\sum\big(|Z_j+s_j|^p-\la_p(s_j)\big)<c+\si^{-p}\sum\big(\la_p(0)-\la_p(s_j)\big)\Big) \\
	&\approx\PP\Big(\zeta_{p,b_p}<c+\si^{-p}\sum f_p(s_j)\Big), 
\end{align*}
whence $c\to\Phi_{p,b_p}^{-1}(\al)$ and 
$\si^{-p}\sum f_p(s_j)\to\Phi_{p,b_p}^{-1}(\be)-\Phi_{p,b_p}^{-1}(\al)$.   
In this case, when $p\in(-1,-\frac12)$, one should also note that, by Lemma~\ref{lem:lla}(i),  
the sum $\sum f_p(ts_j)$ continuously increases (rather than decreases) in $t\in[0,1]$, from $0$ to $\sum f_p(s_j)$, which, by Lemma~\ref{lem:h}, is w.l.o.g.\ $>>d^{|p|}$ if condition \eqref{eq:assum,p in(-1,-frac12)} fails to hold.  
Also, in this case one should use $K_{\al,\tbe;p}\in(0,\infty)$ instead of $K_{\al,\tbe;p}\in(0,1)$. 

\subsection{Case: $p=-1$}\label{p=-1}
Take any varying vector $\s$. 
As in previous cases, we begin by proving a limit theorem for $\sum|Z_j+s_j|^p$ (here with $p=-1$) under an additional assumption, which in this case will be the following:  
\begin{equation}\label{eq:assum,p=-1}
	\sum[\vpi(0)-\vpi(s_j)]\OO d/\ln d.
\end{equation}
Let again
\begin{equation*}
	\si:=\tfrac1d, 
\end{equation*}
and use Lemma~\ref{lem:nu}, now with 
\begin{equation*}
	\xi_{dj}:=\frac{|Z_j+s_j|^{-1}-\tmu_d(s_j)}{\si^{-1}}.  
\end{equation*}
As in the case with $p\in(-1,-\frac12)$, here one shows that 
conditions \eqref{eq:null} and (I) in Lemma~\ref{lem:nu} hold for $\nu=\sqrt{\frac2\pi}\,\nu_p$, as in Definition~\ref{def:zeta}; the only difference is that here, in view of Lemma~\ref{lem:ttmu}(ii), one writes $\si\tmu_d(s_j)\OO d^{-1}\ln d\to0$ instead of relations $\si^{-p}=1/d^{|p|}\to0$ and $\la_p(s)\le\la_p(0)\OO1$ used for $p\in(-1,-\frac12)$. 

Let us now verify condition (III) in Lemma~\ref{lem:nu}. As in the case $p\in(-1,-\frac12)$, here 
\eqref{eq:xi,iff} holds, but now with  
\begin{equation*}
	\mu_j:=\tmu_d(s_j)\OO\ln d; 
\end{equation*}
note also that $\si_j\le\si$ for all $j$. 

Another difference is that here the expectation $\E\xi_{dj}$ is not $0$; in fact, it does not exist (or, one may prefer to say, is infinite). Recall \eqref{eq:tmu} and write   
\begin{align*}
	E_1&:=\sum\E\xi_{dj}\ii{|\xi_{dj}|\le1} \\
	&=\tfrac1d\,\sum\big[\E|Z_j+s_j|^{-1} \ii{|Z_j+s_j|\ge\si_j}-\tmu_d(s_j)\PP(|Z_j+s_j|\ge\si_j)\big] \\
	&=\tfrac1d\,(E_{11}+E_{12}-E_{13}),
\end{align*}
where
\begin{align*}
		E_{11}&:=\sum\E|Z_j+s_j|^{-1} \ii{\si>|Z_j+s_j|\ge\si_j}, \\
		&\OO d\,\int_{\si_j}^\si|u|^{-1}\,\dd u=2d\ln\tfrac\si{\si_j}<<d, \\
		E_{12}&:=\sum\E|Z_j+s_j|^{-1} \ii{|Z_j+s_j|\ge\si} \PP(|Z_j+s_j|<\si_j) \\
		&\le\sum\tmu_d(s_j) \PP(|Z_j+s_j|<\si_j)\OO d\,(\ln d)\si<<d, \\
		E_{13}&:=d\,\sum\PP(|Z_j+s_j|<\si)\PP(|Z_j+s_j|\ge\si_j) \\
		&\sim d\,\sum\PP(|Z_j+s_j|<\si)=d\,2\si\,\sum\vpi(s_j+v_j)\sim2d\,\vpi(0), 
\end{align*}
for some $v_j$'s with $|v_j|<\si$, in view of Lemma~\ref{lem:vpi} and condition \eqref{eq:assum,p=-1}.  
On collecting these estimates, one has 
\begin{equation*}
	E_1\to-2\,\vpi(0). 
\end{equation*}

As for the term $E_2:=\sum\E\xi_{dj}^2\ii{|\xi_{dj}|\le1}$, it is treated (now for $p=-1$) in an almost literally the same way as in the case $p\in(-1,-\frac12)$, with the only difference that the estimate $\tmu_d(s_j)\OO\ln d$ is used (instead of $\mu_j\OO1$), which is still enough to adequately bound $E_{22}$ and $E_{23}$. 
The result is that $E_2\to2\vpi(0)$. 

The treatment of the term $E_3:=\sum\E^2\xi_{dj}\ii{|\xi_{dj}|\le1}$ here is similar to that of $E_1$ (for this case $p=-1$), but a bit simpler: 
\begin{align*}
	E_3&\OO\tfrac1{d^2}(E_{31}+E_{32}+E_{33}), \\
		E_{31}&:=\sum\E^2|Z_j+s_j|^{-1} \ii{\si>|Z_j+s_j|\ge\si_j} 
		\OO d\,\ln^2\tfrac\si{\si_j}<<d^2, \\
		E_{32}&:=\sum\E^2|Z_j+s_j|^{-1} \ii{|Z_j+s_j|\ge\si} \PP^2(|Z_j+s_j|<\si_j) \\
		&\OO d\,(\ln d)^2\si^2<<d^2, \\
		E_{33}&:=d^2\,\sum\PP^2(|Z_j+s_j|<\si)\PP^2(|Z_j+s_j|\ge\si_j) 
		\OO d^2\,d\,\si^2<<d^2.
\end{align*}
Thus, $E_3\to0$. 

The rest of the proof of Proposition~\ref{prop:AS shift iff} in this case $p=-1$ is similar to that for 
the previously considered cases $p\in(-\infty,-1)$ and $p\in(-1,-\frac12)$. Instead of \eqref{eq:assum,p in(-infty,-1)} or \eqref{eq:assum,p in(-1,-frac12)}, here one has to deal with condition \eqref{eq:assum,p=-1} -- which is implied, in view of Lemma~\ref{lem:h}, by \eqref{eq:as iff} (now for $p=-1$). 
In this case, by Lemma~\ref{lem:ttmu}(i), 
the sum $\sum f_p(ts_j)$ continuously increases in $t\in[0,1]$ from $0$ to $\sum f_p(s_j)$, which, 
by Lemma~\ref{lem:h}, is w.l.o.g.\ $>>d$ if condition \eqref{eq:assum,p=-1} fails to hold. 

\subsection{Case: $p=-\frac12$}\label{p=-frac12}
This case is somewhat similar to the case $p\in(-1,-\frac12)$. However, here the limit distribution is normal, rather than stable with index less than $2$. 

Take any varying vector $\s\in\R^d$, with $d\to\infty$. 
To begin, assume that 
\begin{equation}\label{eq:assum,p=-frac12}
	\sum[\vpi(0)-\vpi(s_j)]\OO \si^{-1/2}=\sqrt{d\ln d},
\end{equation}
where
\begin{equation*}
	\si:=\tfrac1{d\ln d}, 
\end{equation*}
and use Lemma~\ref{lem:nu} with 
\begin{equation*}
	\xi_{dj}:=\frac{|Z_j+s_j|^{-1/2}-\mu_j}{\si^{-1/2}},\quad\text{where}\quad\mu_j:=\la_{-1/2}(s_j).  
\end{equation*}
Then for each fixed $x\in(0,\infty)$ the first two lines of display \eqref{eq:->nu} will hold here as well; however, now instead of $\si d=1$ one has $\si d=\tfrac1{\ln d}\to0$, with the effect that $\sum\PP(\xi_{dj}>x)\to0$. 
Therefore, condition (I) in Lemma~\ref{lem:nu} holds with $\nu=0$. 

To verify condition (III) in Lemma~\ref{lem:nu}, let us employ the same $E_1$, $E_{11}$, and $E_{12}$ as in the case $p\in(-1,-\frac12)$ -- with the only difference that now $p=-\frac12$, and these terms are easier to estimate here: 
$E_{12}\OO\si d<<1<<\si^{-1/2}$ and 
\begin{align*}
		E_{11}\OO\sum\int_\R|x|^{-1/2} \ii{|x|<\si_j}\,\dd x 
\OO d\,\si^{1/2}<<\si^{-1/2}. 
\end{align*}
We conclude that
\begin{equation*}
	E_1\to0. 
\end{equation*}

Estimation of $E_2$ is more involved. Here, let us employ the same $E_2$, $E_{21}$, $E_{22}$, $E_{23}$, $E_{211}$, and $E_{212}$ as in the case $p\in(-1,-\frac12)$ -- with the only difference that now $p=-\frac12$. 
Then, just as easily as in the case $p\in(-1,-\frac12)$, here one has 
\begin{equation*}
	E_{22}+E_{23}<<\si^{-1}. 
\end{equation*}
To estimate $E_{21}$, begin with 
\begin{align*}
		E_{211}&=E_{2111}+E_{2112}+E_{2113}, \\
		E_{2111}&:=\sum\int_{\si_j}^{1/\ln d} u^{-1}\,2\vpi(u)\,\dd u
		\sim d\,2\vpi(0)\,\ln\tfrac1{\si\ln d}=2\vpi(0)\si^{-1}, \\
		E_{2112}&:=\sum\int_{1/\ln d}^1 u^{-1}\,2\vpi(u)\,\dd u
		\OO d\,\ln\ln d<<\si^{-1}, \\
		E_{2113}&:=\sum\int_1^\infty u^{-1}\,2\vpi(u)\,\dd u
		\OO d<<\si^{-1}.  
\end{align*}
Next, on recalling \eqref{eq:2nd diff} and \eqref{eq:assum,p=-frac12}, 
\begin{align*}
		E_{212}&=E_{2121}+E_{2122}, \\
		E_{2121}&:=\sum\int_{\si_j}^1 u^{-1}\,[\vpi(u-s_j)+\vpi(u+s_j)-2\vpi(u)]\,\dd u \\
		&\OO\sum[\vpi(0)-\vpi(s_j)]\,\ln\frac1\si 
		\sim\sum[\vpi(0)-\vpi(s_j)]\,\ln d
		\OO\sqrt{d\ln^3 d}<<\si^{-1},\\
		E_{2122}&:=\sum\int_1^\infty u^{-1}\,[\vpi(u-s_j)+\vpi(u+s_j)-2\vpi(u)]\,\dd u \\
		&\le \sum\int_1^\infty [\vpi(u-s_j)+\vpi(u+s_j)-2\vpi(u)]\,\dd u
		\OO d<<\si^{-1}.  
\end{align*}
It follows that 
\begin{equation*}
	E_2\to2\vpi(0)=(\tfrac2\pi)^{1/2}. 
\end{equation*}
Just as for $p\in(-1,-\frac12)$, one sees that $E_3\to0$ for $=-\frac12$ as well. Thus, by Lemma~\ref{lem:nu},
\begin{equation*}
	\sum\xi_{dj}\overset{\mathrm{D}}{\longrightarrow}(\tfrac2\pi)^{1/4}\,Z\sim\No\big(0,(\tfrac2\pi)^{1/2}\big). 
\end{equation*}
The rest of the proof of Proposition~\ref{prop:AS shift iff} for this case, $p=-\frac12$, is similar to that for 
the cases when $p\in(-\infty,-1)$ or $p\in(-1,-\frac12)$. 

\subsection{Case: $p\in(-\frac12,0)$}\label{p in(-frac12,0)}
This case is somewhat similar to the case $p=-\frac12$. 
Take any varying vector $\s\in\R^d$, $d\to\infty$. 
To begin, assume that 
\begin{equation}\label{eq:assum,p in(-frac12,0)}
	\sum[\vpi(0)-\vpi(s_j)]\OO \si^p,
\end{equation}
where
\begin{equation*}
	\si:=\big(d\,\la_{p,2}(0)\big)^{1/(2p)}, 
\end{equation*}
and use Lemma~\ref{lem:nu} with 
\begin{equation*}
	\xi_{dj}:=\frac{|Z_j+s_j|^p-\mu_j}{\si^p},\quad\text{where}\quad\mu_j:=\la_p(s_j).  
\end{equation*}
Then for each fixed $x\in(0,\infty)$ the first two lines of display \eqref{eq:->nu} will hold here as well; however, now instead of $\si d=1$ one has $\si d\asymp d^{(2p+1)/(2p)}\to0$, with the effect that $\sum\PP(\xi_{dj}>x)\to0$. 
Therefore, condition (I) in Lemma~\ref{lem:nu} holds with $\nu=0$. 

To verify condition (III) in Lemma~\ref{lem:nu}, let us employ the same $E_1$, $E_{11}$, and $E_{12}$ as in the case $p\in(-1,-\frac12)$, and these terms are easier to estimate here: 
$E_{12}\OO\si d<<\si^p$ and 
\begin{align*}
		E_{11}\OO\sum\int_\R|x|^p \ii{|x|<\si_j}\,\dd x 
\OO d\,\si^{p+1}<<\si^p, 
\end{align*}
so that
\begin{equation*}
	E_1\to0. 
\end{equation*}

Next, write 
\begin{equation*}
		E_2=\sum\E\xi_{dj}^2\ii{|\xi_{dj}|\le1}=E_{21}-E_{22}
\end{equation*}
in this case with
\begin{align*}
E_{22}&:=\sum\E\xi_{dj}^2\ii{|\xi_{dj}|>1}\OO E_{221}+E_{222},\\
E_{221}&:=\sum\tfrac{\mu_j^2}d\PP(|\xi_{dj}|>1)\to0 \\
\intertext{\big(since $\mu_j\le\la_p(0)<\infty$ and, by what was proved above, $\sum\PP(|\xi_{dj}|>1)\to0$\big), }
E_{222}&:=\sum\tfrac1d\E|Z_j+s_j|^{2p}\ii{|Z_j+s_j|<\si_j} \\
&\OO\tfrac1d\,\sum\int_\R|x|^{2p} \ii{|x|<\si_j}\,\dd x 
\OO\tfrac1d\,d\,\si^{2p+1}\to0, \\
	E_{21}&:=\sum\E\xi_{dj}^2=\sum\frac{\la_{p,2}(s_j)}{d\,\la_{p,2}(0)}=1+O(d^{-1/2})\to1,
\end{align*}
using \eqref{eq:la_2p(s)-la_2p} to get $|\la_{p,2}(s)-\la_{p,2}(0)|\OO s^2\ii{|s|\le1}+\ii{|s|>1}=s^2\wedge1\asymp\vpi(0)-\vpi(s)$ and then recalling condition \eqref{eq:assum,p in(-frac12,0)}. 
It follows that 
\begin{equation*}
	E_2\to1. 
\end{equation*}

Next, similarly to the case $p\in(-1,-\frac12)$, one shows that $E_3\to0$. 

The rest of the proof of Proposition~\ref{prop:AS shift iff} for this case, $p\in(-\frac12,0)$, is similar to that for $p\in(-1,-\frac12)$. 

\subsection{Case: $p\in(0,2)$}\label{p in(0,2)}
This case is somewhat similar to the case $p\in(-\frac12,0)$. However, the proof here is a bit simpler, as it relies on a Berry-Esseen bound on convergence to normality, rather than on the general conditions of convergence to a infinitely divisible distribution, stated in Lemma~\ref{lem:nu}. 

Take any varying vector $\s\in\R^d$, $d\to\infty$. Let 
\begin{equation}\label{eq:xi,p in(0,2)}
	\xi_{dj}:=\frac{|Z_j+s_j|^p-\mu_j}{\sqrt{\sum\la_{p,2}(s_j)}},\quad\text{where}\quad\mu_j:=\la_p(s_j).  
\end{equation} 
To proceed, assume that 
\begin{equation}\label{eq:assum,p in(0,2)}
	\sum g_p(s_j) \OO d^{1/2},
\end{equation}
where $g_p$ is as in in Definition~\ref{def:orl}. 
Note that $s^2\ii{|s|\le1}+(1+|s|^{2p-2})\ii{|s|>1}\le2g_p(s)$ for $p\in(0,2)$ and $|s|>1$, whence, 
by \eqref{eq:la_2p(s)-la_2p} and \eqref{eq:assum,p in(0,2)}, 
\begin{align}
	\Big|\sum\la_{p,2}(s_j)-d\,\la_{p,2}(0)\Big|&\OO\sum g_p(s_j) \OO d^{1/2}<<d, \label{eq:la_2p(s)-la_2p,p in(0,2)}
\end{align}
so that 
\begin{equation}\label{eq:si,p in(0,2)}
	\sum\la_{p,2}(s_j)\sim d\,\la_{p,2}(0). 
\end{equation}
Next, by \eqref{eq:la_3p(s)} and \eqref{eq:assum,p in(0,2)}, 
\begin{align}
	\sum\la_{p,3}(s_j)&\OO d+\sum|s_j|^{3p-1}\ii{|s_j|>1}
	\le d+\Big(\sum|s_j|^p\ii{|s_j|>1}\Big)^{1\vee\frac{3p-1}p} \notag\\ 
	&\OO d+\Big(\sum g_p(s_j)\Big)^{1\vee\frac{3p-1}p}
	\OO d+d^{\frac{3p-1}{2p}}<<d^{3/2}. \label{eq:sum la_3p(s),p in(0,2)}
\end{align}
Therefore and by \eqref{eq:si,p in(0,2)}, the Berry-Esseen bound on the convergence of the distribution of $\sum\xi_{dj}$ to $\No(0,1)$ is 
\begin{equation}\label{eq:BE OO}
	\OO\frac{\sum\la_{p,3}(s_j)}{\big(\sum\la_{p,2}(s_j)\big)^{3/2}}\longrightarrow0. 
\end{equation}
So, still in view of \eqref{eq:si,p in(0,2)}, 
\begin{equation}\label{eq:converg}	\frac{\sum|Z_j+s_j|^p-\sum\la_p(s_j)}{\sqrt{d\,\la_{p,2}(0)}}\overset{\mathrm{D}}{\longrightarrow}Z\sim\No(0,1).
\end{equation} 

The rest of the proof of Proposition~\ref{prop:AS shift iff} for this case, $p\in(0,2)$, is similar to that for $p\in(-1,-\frac12)$. 
Instead of \eqref{eq:assum,p in(-frac12,0)}, here one has to deal with condition \eqref{eq:assum,p in(0,2)} -- which is implied, in view of Lemma~\ref{lem:h}, by \eqref{eq:as iff} \big(now for $p\in(0,2)$\big). 
In this case, the sum $\sum f_p(ts_j)$ continuously increases in $t\in[0,1]$ from $0$ to 
to $\sum f_p(s_j)$, which, 
by Lemma~\ref{lem:h}, is w.l.o.g.\ 
$>>d^{1/2}$ if condition \eqref{eq:assum,p in(0,2)} fails to hold.  

\subsection{Case: $p=0$}\label{p=0}
This case is similar to the case $p\in(0,2)$. However, here we have to deal with the moments of $\ln|Z+s|$ rather than $|Z+s|$ and, correspondingly, with Lemma~\ref{lem:ttla} rather than Lemma~\ref{lem:lla}. 
Let us describe the main differences between the two cases. Here, let 
\begin{equation*}
	\xi_{dj}:=\frac{\ln|Z_j+s_j|-\mu_j}{\sqrt{\sum\tla_2(s_j)}},\quad\text{where}\quad\mu_j:=\tla(s_j).  
\end{equation*} 
An analogue of \eqref{eq:la_2p(s)-la_2p,p in(0,2)} here follows immediately from Lemma~\ref{lem:ttla}(iv). 
An analogue of \eqref{eq:sum la_3p(s),p in(0,2)} follows by Lemma~\ref{lem:ttla}(vi): 
\begin{align*}\label{eq:sum la_3p(s)}
	\sum\tla_3(s_j)&\OO d+\sum\ln^2|s_j|\ii{|s_j|>e}
	\le d+\Big(\sum\ln|s_j|\ii{|s_j|>e}\Big)^2\\ 
	&\le d+\Big(\sum g_0(s_j)\Big)^2\OO d, 
\end{align*}
under the assumption \eqref{eq:assum,p in(0,2)}, here with $p=0$. 

\subsection{Case: $p\in[2,\infty)$}\label{p in[2,infty)}
This case is somewhat similar to the case $p\in(0,2)$. 
Take any varying vector $\s$. Let the $\xi_{dj}$'s be defined here as in \eqref{eq:xi,p in(0,2)}. 
Assume \eqref{eq:assum,p in(0,2)}, but now with  
\begin{equation*}
	g_p(s):=s^2+|s|^p. 
\end{equation*}
Then, introducing the notation 
\begin{equation}\label{eq:S}
	V_p:=V_p(\s):=\sum|s_j|^p, 
\end{equation}
one has $V_2+V_p\OO d^{1/2}<<d$, whence 
$V_{2p-2}\le V_p^{(2p-2)/p}\OO d^{(p-1)/p}<<d$ 
and $V_{3p-1}\le V_p^{(3p-1)/p}\OO d^{(3p-1)/(2p)}<<d^{3/2}$. 
So,  
by parts (v) and (vi) of Lemma~\ref{lem:lla},  
\begin{align*}
\sum\la_{p,2}(s_j)&=d\la_{p,2}(0)+O(V_2+V_{2p-2})
\sim d\la_{p,2}(0), \\ 
\sum\la_{p,3}(s_j)&\OO d+V_{3p-1}<<d^{3/2}. 	
\end{align*}
Hence, the convergence to $0$ in \eqref{eq:BE OO} for the Berry-Esseen bound holds in the case $p\in[2,\infty)$ as well. 
The rest of the proof of Proposition~\ref{prop:AS shift iff} for this case is similar to that for $p\in(0,2)$.   

\subsection{Case: $p=\infty$}\label{p=infty}
In this case, a varying vector $\s$ is a $p$-AS shift if and only if 
\begin{align}
	1-\al&\approx\PP\Big(\max_1^d|Z_j|\le c\Big), \label{eq:al,infty}\\
	1-\be&\approx\PP\Big(\max_1^d|Z_j+s_j|\le c\Big) \label{eq:be,infty}
\end{align}
for some varying $c$. 
Rewrite \eqref{eq:al,infty} as $d\cdot\ln\PP(|Z|\le c)\to\ln(1-\al)$, which implies that $\ln\PP(|Z|\le c)\to0$ and hence $c\to\infty$ and $\ln\PP(|Z|\le c)\sim\break
\PP(|Z|\le c)-1=-\PP(|Z|>c)\sim-\frac2c\vpi(c)$, so that \eqref{eq:al,infty} can be rewritten as $\frac2c\,\vpi(c)\sim-\frac1d\,\ln(1-\al)$, which implies $c\sim\sqrt{2\ln d}$, so that \eqref{eq:al,infty} can be further rewritten as $\vpi(c)\sim-\frac{\sqrt{\ln d}}{d\sqrt2}\,\ln(1-\al)$ or, equivalently, as 
\begin{equation*}
	2\,\ln\Big(-\frac d{\sqrt{\pi\ln d}}\,\frac1{\ln(1-\al)}\Big)=c^2+o(1).  
\end{equation*}
Also, $c\sim\sqrt{2\ln d}$ implies $c^2+o(1)=c^2(1+o(1/c^2))=[c(1+o(1/c^2))]^2=[c+o(1/c)]^2=[c+o(1/\sqrt{\ln d})]^2$. 
So, recalling \eqref{eq:c(d,al)}, one finally rewrites \eqref{eq:al,infty} as  
\begin{equation}\label{eq:c}
	c=c_{d,\al}+o(1/\sqrt{\ln d}). 
\end{equation}

Take now indeed any $c$ as in \eqref{eq:al,infty} or, equivalently, as in \eqref{eq:c}. 
Rewrite \eqref{eq:be,infty} as 
\begin{equation}\label{eq:be,infty,rewr}
	\sum\ln\PP(|Z+s_j|\le c)\sim\ln(1-\be).
\end{equation}

To complete the consideration of the case $p=\infty$ and thus the entire proof of Proposition~\ref{prop:AS shift iff}, it remains to show that relation \eqref{eq:be,infty,rewr} is equivalent to the same but with $c_{d,\al}$ in place of $c$. 
To that end, it is enough to show that 
\begin{equation}\label{eq:c equiv c(d,al)}
	\ln\PP(|Z+s|\le c)\sim\ln\PP(|Z+s|\le c_{d,\al})
\end{equation}
uniformly over all $s\in\R$. 

Consider first the case when $s$ varies so that $\PP(|Z+s|\le c)\to0$, which is equivalent to $|s|-c\to\infty$, whence $\PP(|Z+s|\le c)\sim\PP(Z>|s|-c)$ and 
\begin{equation*}
	\ln\PP(|Z+s|\le c)\sim-\tfrac{(|s|-c)^2}2\sim-\tfrac{(|s|-c_{d,\al})^2}2\sim\ln\PP(|Z+s|\le c_{d,\al});
\end{equation*}
the second $\sim$ here takes place because $|s|-c\to\infty$ and $c_{d,\al}=c+o(1)$. 
So, \eqref{eq:c equiv c(d,al)} holds when $\PP(|Z+s|\le c)\to0$.  

It remains to consider the case when $\PP(|Z+s|\le c)\OOG1$, which implies $|s|\OO c$ and hence 
$|(c\pm s)(c-c_{d,\al})|<<\sqrt{\ln d}\,\frac1{\sqrt{\ln d}}=1$, which in turn yields  
\begin{align*}
	\PP(|Z+s|>c)&=\PP(Z>c-s)+\PP(Z>c+s)\\
	&\sim\PP(Z>c_{d,\al}-s)+\PP(Z>c_{d,\al}+s)=\PP(|Z+s|>c_{d,\al})  
\end{align*}
and hence \eqref{eq:c equiv c(d,al)}; the relation $\sim$ here follows because for $x>1$ and $h>0$ such that $hx<<1$ one has $0<\PP(Z>x)-\PP(Z>x+h)<h\vpi(x)\OO hx\PP(Z>x)<<\PP(Z>x)$, while for $x\le1$ and $h>0$ such that $h<<1$ one has $0<\PP(Z>x)-\PP(Z>x+h)\OO h<<1\OO\PP(Z>x)$.   

The proof of Proposition~\ref{prop:AS shift iff} is now complete. 

\section{Remaining proofs of the propositions of Section~\ref{proofs}
}\label{proofs of props p}

\begin{proof}[Proof of Proposition~\ref{prop:below}]\ 

\framebox{Case: $p\in[-\infty,-1)$} 
Take any $p$-AS shift $\s$ in the direction of the ``equalized'' $\no{\cdot}_2$-unit vector $\one$, so that  $\s=s\one=(s,\dots,s)$ for some $s\in\R$. 
Then, by Proposition~\ref{prop:AS shift iff} \big(for $p\in[-\infty,-1)$\big), one has $e^{-s^2/2}\asymp1$ and hence $s^2\asymp1$, whence   
\begin{equation}\label{eq:-infty,OOG}
	\|\s\|^2\OOG d. 
\end{equation}
Consider the convex function $\psi$ on $(0,\infty)$ given by the formula $\psi(v):=2\ln\frac1v$. Then, by Lemma~\ref{lem:hardy}, $\sum\psi(v_j)$ is Schur-convex in $\vv$. Hence, the minimum of 
$\|\s\|^2=\sum s_j^2=\sum\psi\big(e^{-s_j^2/2}\big)$ given a fixed value of $\sum f_{-\infty}(s_j)=\sum e^{-s_j^2/2}$ is attained when $\s$ is of the form $s\one$ for some $s\in\R$. 
So, \eqref{eq:-infty,OOG} holds for any $p$-AS shift $\s$, in any direction, which completes the verification of line~1 in \eqref{eq:si_p(d)}. 

\framebox{Case: $p=-1$} 
The proof of Proposition~\ref{prop:sARE} for $p=-1$ is somewhat similar to that in the previously considered case  $p\in[-\infty,-1)$. 
For any $(-1)$-AS shift $\s$ of the form $s\one$ one has 
$f_{-1}(s)\asymp1$, whence, by Lemma~\ref{lem:h}, $\vpi(0)-\vpi(s)\asymp\frac1{\ln d}$,
$s\to0$, $s^2\asymp\frac1{\ln d}$, 
and 
\begin{equation*}
	\|\s\|^2=d\,s^2\asymp\tfrac d{\ln d}. 
\end{equation*}
Concerning $(-1)$-AS shifts $\s$ in directions other than in that of the vector $\one$, \eqref{eq:as iff} and Lemma~\ref{lem:h} will imply $\sum\big(1-e^{-s_j^2/2}\big)\asymp\frac d{\ln d}$. Now the proof in this case is finished as was done for $p\in[-\infty,-1)$, except that here the convex function $\psi$ is given by the formula $\psi(v):=2\ln\frac1{1-v}$, so that $\|\s\|^2=\sum s_j^2=\sum\psi\big(1-e^{-s_j^2/2}\big)$. 

\framebox{Case: $p\in(-1,-\frac12)$}
By Lemma~\ref{lem:lla}(iii), $f_p(\sqrt u)=\la_p(0)-\la_p(\sqrt u)$ is concave in $u\ge0$. 
So, by Lemma~\ref{lem:hardy}, $\sum f_p(s_j)$ is Schur${}^2$-concave in $\s$. So, for any $p$-AS shift $\s$, 
\begin{equation*}
	d\,f_p(s)\ge\sum f_p(s_j)\asymp d^{|p|} 
\end{equation*}
by \eqref{eq:as iff}, where $s:=\no\s_2$. It follows by Lemma~\ref{lem:h} that $\vpi(0)-\vpi(s)\OOG d^{|p|-1}$, whence $s^2\OOG d^{|p|-1}$ and $\|\s\|^2=ds^2\OOG d^{|p|}$. 

\framebox{Case: $p\in(-\frac12,2)\setminus\{0\}$}
This case is quite similar to the case $p\in(-1,-\frac12)$. 

\framebox{Case: $p=0$}
This case too is similar to the case $p\in(-\frac12,0)$. 
Here, instead of Lemma~\ref{lem:lla}(iii), use Lemma~\ref{lem:ttla}(iii).  
 
\framebox{Case: $p=2$}
This case follows immediately by \eqref{eq:as iff}, since $f_2(s)=s^2$. 
 
\framebox{Case: $p\in(2,\infty)$} 
This case as well is much similar to the case $p\in(-\frac12,0)$. However, here $\sum f_p(s_j)$ is Schur${}^2$-\emph{convex} in $\s$, which reduces the consideration to 
$p$-AS shifts of the form $s\ee_1=(s,0,\dots,0)$ with $s>0$. Then \eqref{eq:as iff} implies $s\to\infty$. Hence, by Lemma~\ref{lem:lla}(ii), $\la_p(s)\sim s^p$, so that 
\begin{equation*}
	\|\s\|^p=s^p\sim\la_p(s)\sim f_p(s)=\sum f_p(s_j)\asymp d^{1/2}, 
\end{equation*}
which proves Proposition~\ref{prop:below} in this case.
 
\framebox{Case: $p=\infty$}
Here too, it is enough to consider $p$-AS shifts of the form $s\ee_1$, with $s>0$. 
If $s<<\sqrt{\ln d}$, then (cf.\ \eqref{eq:la infty sim}) $\la_{\infty;d}(s)$ and $\la_{\infty;d}(0)$ are each of the form $d^{-(1+o(1))}$, whence $f_{\infty;d}(s)=\la_{\infty;d}(s)-\la_{\infty;d}(0)\to0$, which contradicts \eqref{eq:as iff}. So, $s\OOG\sqrt{\ln d}$, which 
proves Proposition~\ref{prop:below} in this last case as well.  
\end{proof}

\begin{proof}[Proof of Proposition~\ref{prop:t->1}]\ 

\framebox{Case: $p\in[-\infty,-1)$}
Take any $p$-AS shift $\s$ and any varying $t$ such that $t\to1$, whence $v:=|t-1|\to0$. 
Then $(ts)^2-s^2\to0$ for any varying $s$ such that $|s|<(5/v)^{1/4}$. 
Write 
\begin{equation*}
	\sum e^{-(ts_j)^2/2}=S_1(t)+S_2(t),
\end{equation*}
where
\begin{align*}
	S_2(t)&:=\sum e^{-(ts_j)^2/2}\ii{|s_j|\ge(5/v)^{1/4}} 
	\OO\sum e^{-(ts_j)^2/2}\ii{|ts_j|\ge(4/v)^{1/4}} \\
	&\OO d\,e^{-1/\sqrt v}<<d,\\
	S_1(t)&:=\sum e^{-(ts_j)^2/2}\ii{|s_j|<(5/v)^{1/4}} 
	\sim\sum e^{-s_j^2/2}\ii{|s_j|<(5/v)^{1/4}} \\
		&=\sum e^{-s_j^2/2}-S_2(1)
	=\sum e^{-s_j^2/2}+o(d). 
\end{align*}
Therefore and because $\sum e^{-s_j^2/2}\ii{|s_j|<(5/v)^{1/4}}\le d$, one has 
$\sum e^{-(ts_j)^2/2}=\sum e^{-s_j^2/2}+o(d)$. In view of the condition \eqref{eq:as iff} (for $p\in[-\infty,-1)$), it follows that $t\s$ is a $p$-AS shift. 

\framebox{Case: $p\in(-1,-\frac12)$}
Take any $p$-AS shift $\s$ and any varying $t$ such that $t\to1$. 
By Lemma~\ref{lem:lla}(iii,ii) and l'Hospital's rule for limits, $f_p'(s)/s$ decreases in $s\in(0,\infty)$ from $f_p''(0)=|p|\,\la_p(0)$ and remains nonnegative for such $s$. So, $|sf_p'(s)|\OO s^2$ and hence, in view of \eqref{eq:la_p(s)-la_p, s to0}, 
\begin{equation}\label{eq:sf'(s),f(s)}
	|sf_p'(s)|\OO f_p(s)
\end{equation}
for all $s$ in a neighborhood of $0$; the same holds for $|s|\OO1$, since the even function $f_p$ is strictly increasing and positive on $(0,\infty)$. It follows by the mean value theorem that for any varying $|s|\OO1$ there exists some varying $\xi$ between $t$ and $1$ such that 
\begin{equation}\label{eq:f(ts)/f(s)}
	\frac{f_p(ts)}{f_p(s)}=\exp[\ln f_p(ts)-\ln f_p(s)]=\exp\big(\tfrac{t-1}\xi\,\tfrac{\xi sf_p'(\xi s)}{f_p(\xi s)}\big)\to1.
\end{equation}
If now $|s|\to\infty$ then, by Lemma~\ref{lem:lla}(i), both $f_p(ts)$ and $f_p(s)$ converge to $\la_p(0)-\la_p(\infty-)\in(0,\infty)$. 
In view of the condition \eqref{eq:as iff} (for $p\in(-1,-\frac12)$), it follows that $t\s$ is a $p$-AS shift. 

\framebox{Case: $p=-1$}
This case is similar to the case $p\in(-1,-\frac12)$; here, the result follows by \eqref{eq:f(ts)/f(s)} and Lemma~\ref{lem:ttmu}(iv). 

\framebox{Case: $p=-\frac12$}
This case is quite similar to the case $p\in(-1,-\frac12)$.

\framebox{Case: $p\in(-\frac12,0)$}
This case too is quite similar to the case $p\in(-1,-\frac12)$.

\framebox{Case: $p\in(0,2)$}
This case is somewhat similar to the case $p\in(-1,-\frac12)$. Indeed, in view of \eqref{eq:f(ts)/f(s)}, it suffices to show that \eqref{eq:sf'(s),f(s)} holds 
over all $s\in\R$. For $|s|\OO1$, this follows as in the case for $p\in(-1,-\frac12)$. 
As for $|s|\to\infty$, \eqref{eq:sf'(s),f(s)} follows by Lemma~\ref{lem:lla}(ii), which yields $f_p(s)\sim\la_p(s)\sim|s|^p$ and $|f_p'(s)|=|\la_p'(s)|=p|\E|Z+s|^{p-1}\sign(Z+s)|\le p\la_{p-1}(s)\sim p|s|^{p-1}$.  

\framebox{Case: $p=0$}
This case is somewhat similar to the cases $p\in(-1,-\frac12)$ and $p\in(0,2)$. Here, instead of Lemma~\ref{lem:lla}(iii,ii,i), use Lemma~\ref{lem:ttla}(iii,ii) for $|s|\OO1$ and Lemma~\ref{lem:ttla}(ii,vii) for $|s|\to\infty$. 

\framebox{Case: $p\in[2,\infty)$}
This case is quite similar to the case $p\in(0,2)$. 

\framebox{Case: $p=\infty$}
Here we need to require that $|t-1|<<\frac1{\ln d}$. In view of \eqref{eq:f(ts)/f(s)} and the symmetry in $s$, it suffices to show that $|r|\OO\ln d$ over all $s\ge0$, where
\begin{equation*}
	r:=\frac{sf_\infty'(s)}{f_\infty(s)}
	=\frac{s\De_\vpi(s)}{\De_\Phi(s)\ln\De_\Phi(s)},
\end{equation*}
where 
$c:=c_{d,\al}$, $\De_\vpi(s):=\vpi(s+c)-\vpi(s-c)$, and $\De_\Phi(s):=\Phi(s+c)-\Phi(s-c)$. Consider the three subcases, depending on whether $s$ is to the left of $c$, to the right of $c+1$, or between $c$ and $c+1$. 

If $0\le s\le c$, then $\De_\Phi(s)\ge\De_\Phi(c)=\Phi(2c)-\Phi(0)\OOG1$, $-\ln\De_\Phi(s)\asymp1-\De_\Phi(s)\ge\PP(Z>c-s)\asymp\frac{\vpi(c-s)}{c-s+1}$, while $|\De_\vpi(s)|\le\vpi(c-s)$; hence and by \eqref{eq:c(d,al)}, in this subcase $|r|\OO(c-s+1)s\OO c^2\asymp\ln d$. 

If $s\ge c+1$, then $\De_\Phi(s)=\PP(Z>s-c)-\PP(Z>s+c)\asymp\PP(Z>s-c)\asymp\frac{\vpi(s-c)}{s-c}$,  $-\ln\De_\Phi(s)\asymp(s-c)^2$, while $|\De_\vpi(s)|\le\vpi(c-s)$; hence, in this subcase $|r|\OO\frac s{s-c}\le c+1\asymp\sqrt{\ln d}\OO\ln d$. 

Finally, if $c<s<c+1$, then $\De_\Phi(s)=\PP(Z>s-c)-\PP(Z>s+c)\break
\asymp\PP(Z>s-c)\in[\PP(Z>0),\PP(Z>1)]\subset(0,1)$,  $-\ln\De_\Phi(s)\asymp1$, while $|\De_\vpi(s)|\le\vpi(c-s)\OO1$; hence, in this subcase $|r|\OO s<c+1\OO\ln d$. 
\end{proof}

\begin{proof}[Proof of Proposition~\ref{prop:s->n,th}]
Let $\s_p$ be any $p$-AS shift in the direction of $\th_1$. 
By Proposition~\ref{prop:below}, for each $p\in[-\infty,\infty]$ there is a function $\N\ni d\mapsto\si_p(d)\in(0,\infty)$ such that eventually $\|\s_p\|\ge\si_p(d)$. 
Let $n_p(d)$ be as in Proposition~\ref{prop:NA}. 
W.l.o.g., the function $d\mapsto\tth_p(d)$ in Proposition~\ref{prop:NA} is such that $\tth_p(d)\le\si_p(d)/\sqrt{n_p(d)}$ for all $d$, whence \eqref{eq:n} will imply $n_p\ge n_p(d)$ provided that
$\|\th_1\|\le\tth_p(d)$, and so, by Proposition~\ref{prop:NA}, the pair $(n_p,\th_1)$ will be $p$-NA. 
Moreover, for such $\th_1$ and $t:=\sqrt{n_p}\,\frac{\|\th_1\|}{\|\s_p\|}$, \eqref{eq:n} implies  
\begin{equation*}
	0\le t^2-1\le\frac{\|\th_1\|^2}{\|\s_p\|^2}\le\frac{\tth_p(d)^2}{\si_p(d)^2},
\end{equation*}
so that $t$ goes to $1$ fast enough if $\tth_p(d)$ is sufficiently small. 
Therefore, by Proposition~\ref{prop:t->1}, $\sqrt{n_p}\,\th_1=t\s_p$ is a $p$-AS shift in the direction of $\th_1$. 
It follows now by Propositions~\ref{prop:n,th->s} and \ref{prop:AS shift} that 
the $p$-NA varying pair $(n_p,\th_1)$ is 
$p$-AS. 
Finally, the relations $\sqrt{n_p}\,\th_1=t\s_p$ and $t\to1$ yield $\|\s_p\|\sim\sqrt{n_p}\,\|\th_1\|$.   
\end{proof}

\begin{proof}[Proof of Proposition~\ref{prop:AS shift exist}]\ 

\framebox{(I)} Part (I) of the proposition follows because, as it is easy to see, the d.f.\ of the r.v.\ $\no{\Zd}_p$ is continuously increasing on $[0,\infty)$ from $0$ to $1$. 

\framebox{(II)} In view of Lemma~\ref{lem:mono}, part (II) of the proposition follows because 
for each $p\in[0,\infty]$, each $d\in\N$, each $\no{\cdot}_2$-unit vector $\uu\in\R^d$, and each $\zz\in(\R\setminus\{0\})^d$, one has $\no{\zz+t\uu}_p\to\infty$ and hence 
$\PP(\no{\Zd+t\uu}_p>c)\to1$ as $t\to\infty$. 

\framebox{(III)} Implication 
(III)(a)$\implies$(III)(c) follows immediately by Definition~\ref{def:AS shift} and Proposition~\ref{prop:AS shift}. Implication (III)(c)$\implies$(III)(b) follows by part (I) of Proposition~\ref{prop:AS shift exist}. 

To prove (III)(b)$\implies$(III)(a), assume that (III)(b) holds, that is, for some varying $c$ such that $\PP(\no{\Zd}_p>c)\to\al$ and some varying vector $\s$ in the direction of $\uu$ one has $\tbe:=\lim\PP(\no{\Zd+\s}_p>c)\ge\be$. 
If $\tbe=\be$, then, by definition, $\s$ is a $p$-AS shift. Otherwise, if $\tbe>\be$, then eventually  $\PP(\no{\Zd+\s}_p>c)>\be$. So, by Lemma~\ref{lem:mono}, for some varying $t\in(0,1)$ the vector $t\s$ is a $p$-AS shift. That is, in any case (III)(a) holds. 

Thus, (III)(a)$\iff$(III)(b)$\iff$(III)(c).  

To establish the equivalences (III)(a)$\iff$(III)(d)$\iff$(III)(e), we shall use Proposition~\ref{prop:AS shift iff}. In the case $p\in[-\infty,-1)$, observe that $\sum f_p(tu_j)=\sum e^{-t^2u_j^2/2}$ continuously decreases in $t\in[0,\infty)$ from $d$ to $d_0(\uu)$; so, one has (III)(a)$\iff$(III)(d)$\iff$(III)(e)  
because in this case $K_p\in(0,1)$. 
In the case $p=-1$, by Lemma~\ref{lem:ttmu}(i,ii), $\sum f_p(tu_j)$ continuously increases in $t\in[0,\infty)$ from $0$ to $[d-d_0(\uu)]\tmu_d(0)\sim[d-d_0(\uu)]\sqrt{\tfrac2\pi}\ln d$, whence in this case  the equivalences (III)(a)$\iff$(III)(d)$\iff$(III)(e) follow. The remaining two cases are similar to the case $p=-1$.   
\end{proof}

\begin{proof}[Proof of Proposition~\ref{prop:are=a_p2}]
\ 

\textbf{Case 1:} \emph{there is no $(p,2)$-NAAS triple $(n_p,n_2,\th_1)$ with $\th_1$ in the direction of $\uu$}.\quad 
By virtue of parts (II) of Propositions~\ref{prop:AS exist} and \ref{prop:AS shift exist} and the equivalences (III)(a)$\iff$(III)(d) in these propositions, the condition defining Case~1 can be restated as ``there is no $p$-AS shift $\s$ in the direction of $\uu$''. Therefore and by Definitions~\ref{def:are} and \ref{def:a_p2}, in Case~1 both $\are_{p,2,\uu}$ and $a_{p,2,\uu}$ exist and equal $0$. 

\textbf{Case 2:} \emph{there is a $(p,2)$-NAAS triple $(n_p,n_2,\th_1)$ with $\th_1$ in the direction of $\uu$.}\quad  
This is equivalent to ``there is a $p$-AS shift $\s$ in the direction of $\uu$'' as well as to ``there are $p$-AS and $2$-AS shifts $\s_p$  and $\s_2$ in the direction of $\uu$''. 

Suppose now that $\are_{p,2,\uu}$ exists. By Definition~\ref{def:are}, in Case~2 this means that $\lim\frac{n_2}{n_p}$ exists in $[0,\infty]$ for any $(p,2)$-NAAS triple $(n_p,n_2,\th_1)$ with $\th_1$ in the direction of $\uu$. 
Take now any $p$-AS and $2$-AS shifts $\s_p$  and $\s_2$ in the direction of $\uu$. 
Then, by Proposition~\ref{prop:s->n,th}, there is a $(p,2)$-NAAS triple $(n_p,n_2,\th_1)$ with $\th_1$ in the direction of $\uu$ such that $\|\s_p\|\sim\sqrt{n_p}\,\|\th_1\|$ and $\|\s_2\|\sim\sqrt{n_2}\,\|\th_1\|$, whence $\frac{\|\s_2\|^2}{\|\s_p\|^2}\sim\frac{n_2}{n_p}$. 
It follows, by Definitions~\ref{def:a_p2} and \ref{def:are}, that $a_{p,2,\uu}$ exists and equals $\are_{p,2,\uu}$. 

Vice versa, suppose that 
$a_{p,2,\uu}$ exists. Then, by Definition~\ref{def:a_p2} and because there are $p$-AS and $2$-AS shifts $\s_p$  and $\s_2$ in the direction of $\uu$, one concludes that $\lim\frac{\|\s_2\|^2}{\|\s_p\|^2}$ exists in $[0,\infty]$, and is the same, for any such shifts $\s_p$  and $\s_2$. 
Take now any $(p,2)$-NAAS triple $(n_p,n_2,\th_1)$ with $\th_1$ in the direction of $\uu$. 
Then, by Proposition~\ref{prop:n,th->s}, the vectors $\s_p:=\sqrt{n_p}\,\th_1$ and $\s_2:=\sqrt{n_2}\,\th_1$ will, respectively, be $p$-AS and $2$-AS shifts in the direction of $\uu$; at that, obviously, one will have  $\frac{n_2}{n_p}=\frac{\|\s_2\|^2}{\|\s_p\|^2}$. 
So, by Definitions~\ref{def:are} and \ref{def:a_p2}, $\are_{p,2,\uu}$ exists and equals $a_{p,2,\uu}$. 
\end{proof}

\begin{proof}[Proof of Proposition~\ref{prop:a_p2 schur2}] \ 

Consider the case $p\in[-\infty,2]$. Assume, to the contrary, that $a_{p,2,\uu}<a_{p,2,\vv}$. In particular, this implies that $a_{p,2,\vv}>0$. So, by the last sentence of Definition~\ref{def:a_p2}, there is a $p$-AS shift in the direction of $\vv$. 

Consider first the subcase when there is a $p$-AS shift in the direction of $\uu$ as well. Then for each $r\in\{2,p\}$ and each $\ww\in\{\uu,\vv\}$ there is an $r$-AS shift $\s_{r,\ww}=t_{r,\ww}\ww$, with some varying $t_{r,\ww}\in(0,\infty)$; that is (recall Proposition~\ref{prop:AS shift}), for some (or, equivalently, any) varying $c_r$ such that $\PP(\no{\Zd}_r>c_r)\to\al$ one has 
$\PP(\no{\Zd+t_{r,\ww}\ww}_r>c_r)\to\be$. 
By \eqref{eq:sim K p=2} -- which, quite independently, will be proved later, one has $\|\s_{2,\uu}\|\sim\|\s_{2,\vv}\|$ or, equivalently, $t_{2,\uu}\sim t_{2,\vv}$. So, the assumption $a_{p,2,\uu}<a_{p,2,\vv}$ implies that eventually 
$\|\s_{p,\vv}\|<\|\s_{p,\uu}\|$ or, equivalently, $t_{p,\vv}<t_{p,\uu}$. 
Therefore, by Lemmas~\ref{lem:schur2} and \ref{lem:mono}, eventually 
\begin{equation}\label{eq:schur2,be}
\begin{aligned}
\be\approx\PP(\no{\Zd+t_{p,\uu}\uu}_p>c_p)
\ge\PP(\no{\Zd+t_{p,\uu}\vv}_p>c_p)
\ge\PP(\no{\Zd+t_{p,\vv}\vv}_p>c_p)\approx\be, 	
\end{aligned}	
\end{equation}
whence $\PP(\no{\Zd+t_{p,\uu}\vv}_p>c_p)\approx\be$, so that $\tilde\s_{p,\vv}:=t_{p,\uu}\vv$ is a $p$-AS shift. 
It follows that 
\begin{equation*}
	\lim\frac{t_{2,\vv}^2}{t_{p,\uu}^2}
	=\lim\frac{\|s_{2,\vv}\|^2}{\|\tilde\s_{p,\vv}\|^2}
	=a_{p,2,\vv}>a_{p,2,\uu}
	=\lim\frac{\|s_{2,\uu}\|^2}{\|\s_{p,\uu}\|^2}
	=\lim\frac{t_{2,\uu}^2}{t_{p,\uu}^2},
\end{equation*}
which contradicts the relation $t_{2,\uu}\sim t_{2,\vv}$. 

Consider now the other subcase of the case $p\in[-\infty,2]$, when there is no $p$-AS shift in the direction of $\uu$. Then, by the equivalence (III)(a)$\iff$(III)(b) in 
Proposition~\ref{prop:AS shift exist}, 
$\limsup\PP(\no{\Zd+t_{p,\vv}\uu}_p>c_p)<\be,$ 
so that eventually $\PP(\no{\Zd+t_{p,\vv}\uu}_p>c_p)<\tbe$ for some constant $\tbe\in(0,\be)$. Hence (cf.\ \eqref{eq:schur2,be})
\begin{equation*}
\begin{aligned}
\tbe>\PP(\no{\Zd+t_{p,\vv}\uu}_p>c_p)
\ge\PP(\no{\Zd+t_{p,\vv}\vv}_p>c_p)\approx\be, 	
\end{aligned}	
\end{equation*}
which is a contradiction. 

The case $p\in[2,\infty]$ is similar to the case $p\in[-\infty,2]$, and even simpler, because for any $p\in[2,\infty]$ and any varying direction, there is a $p$-AS shift in that direction, by Proposition~\ref{prop:AS shift exist}(II). 
\end{proof}

\begin{proof}[Proof of Proposition~\ref{prop:sARE}]\ 

\framebox{Case: $p\in[-\infty,-\frac12]$} 
This case follows by relation \eqref{eq:sim K p=2}, to be proved later, and Proposition~\ref{prop:below}. 

\framebox{Case: $p\in(-\frac12,0)$}
Let $\s$ be any $p$-AS shift in the direction of some $\no{\cdot}_2$-unit vector $\uu\in\R^d$. 

Consider first the subcase $\no{\uu}_{p,2}<<d^{1/4}$ \big(of the current case $p\in(-\frac12,0)$\big). 

In the subsubcase $0\ne\no{\uu}_{p,2}<<d^{1/4}$, introduce the vector $\vv:=\uu/\no\uu_{p,2}$. Then $\sum f_p(v_j)=K_p d^{1/2}$, by Definition~\ref{def:orl} and because $f_p(s)$ is continuously increasing in $|s|$ (by Lemma~\ref{lem:lla}(i)). 
On the other hand, $\sum f_p(s_j)\sim K_p d^{1/2}$, by Proposition~\ref{prop:AS shift iff}. 
Also, $\sum v_j^2=\|\vv\|^2=\|\uu\|^2/\no\uu_{p,2}^2=d/\no\uu_{p,2}^2>>d^{1/2}$, by the subcase assumption $\no{\uu}_{p,2}<<d^{1/4}$. Therefore, by Lemma~\ref{lem:s,v} (with $S:=K_p d^{1/2}$), one has $\|\s\|^2=\sum s_j^2>>d^{1/2}$, whence, by \eqref{eq:sim K p=2}, $a_{p,2,\uu}=0$. 

Otherwise, one has the subsubcase $\no{\uu}_{p,2}=0$; that is, $\sum f_p(u_j/u)\le K_p d^{1/2}$ for all $u\in(0,\infty)$. At that, one still has $\sum f_p(s_j)\sim K_p d^{1/2}$, while $s_j=tu_j$ for some varying $t\in(0,\infty)$ and all $j$. One can find a varying $u\in(0,\infty)$ so small as $tu\le1$ and $\sum v_j^2>>d^{1/2}$, with $v_j:=u_j/u$. 
Then $|s_j|=t|u_j|=tu|v_j|\le|v_j|$ for all $j$. Therefore and because $f_p(s)$ is increasing in $|s|$, one has 
$K_p d^{1/2}\sim\sum f_p(s_j)\le\sum f_p(v_j)=\sum f_p(u_j/u)\le K_p d^{1/2}$, which yields $\sum f_p(v_j)\sim K_p d^{1/2}$. Now, using Lemma~\ref{lem:s,v} again, one concludes that $a_{p,2,\uu}=0$ -- in this subsubcase of the subcase $\no{\uu}_{p,2}<<d^{1/4}$ as well. 

It remains to consider the subcase $\no{\uu}_{p,2}\OOG d^{1/4}$. Of course, this assumption excludes the possibility $\no{\uu}_{p,2}=0$. So, reasoning quite similarly to the subsubcase $0\ne\no{\uu}_{p,2}<<d^{1/4}$, one concludes that, in the current subcase $\no{\uu}_{p,2}\OOG d^{1/4}$, the value of $a_{p,2,\uu}$ is strictly positive -- whenever it exists. 

Thus and in view of Proposition~\ref{prop:a_p2 schur2}, to complete the the proof of Proposition~\ref{prop:sARE} for $p\in(-\frac12,0)$, it suffices to show that the limit $a_{p,2}$ exists and equals $a_p$ for any $p$- and $2$-sufficient shifts $\s_p$ and $\s_2$ in the direction of vector $\one_d=(1,\dots,1)$. 
Let indeed $\s:=\s_p$ be any $p$-sufficient shift of the form $(s,\dots,s)$, with $s\ge0$. Then, by Proposition~\ref{prop:AS shift iff}, 
$\la_p(0)-\la_p(s)\sim K_p\, d^{-1/2}$, whence, by Lemma~\ref{lem:lla}(i,ii), $s\to0$ and 
\begin{equation}\label{eq:OOG,p in(-frac12,0)}
	\|\s\|^2\sim\tfrac2{|p|\la_p(0)}\,K_p\, d^{1/2}. 
\end{equation}
Proposition~\ref{prop:sARE} \big(for $p\in(-\frac12,0)$\big) now follows by \eqref{eq:sim K p=2} and \eqref{eq:la_p}, since 
\begin{equation}\label{eq:la_p(0)}
	\la_p(0)=\frac{2^{p/2}\Ga(\frac{p+1}2)}{\sqrt\pi}
\end{equation}
for all $p\in(-1,\infty)$, and so, by \eqref{eq:a(p)}, 
\begin{equation}\label{eq:=a_p}
	\tfrac{K_2}{K_p}\,\tfrac{|p|\la_p(0)}2=a_p.
\end{equation}
 
\framebox{Case: $p\in(0,2)$}
Here we first note that, by Lemma~\ref{lem:h}, $g_p(s)\asymp f_p(s)$ over all $s\in\R$.  
Therefore, the current case $p\in(0,2)$ is quite similar to the case $p\in(-\frac12,0)$, and even a bit simpler -- since $\no{\uu}_{p,2}=0$ is impossible for $p\in(0,2)$, because then $g_p(s)\to\infty$ as $|s|\to\infty$. 
Here use the relations $\sum g_p(v_j)=K_p d^{1/2}$ and $\sum g_p(s_j)\asymp K_p d^{1/2}$  instead of $\sum f_p(v_j)=K_p d^{1/2}$ and $\sum f_p(s_j)\sim K_p d^{1/2}$, and then use Lemma~\ref{lem:s,v,ge0} instead of Lemma~\ref{lem:s,v}.  

\framebox{Case: $p=0$}
This case is quite similar to the case $p\in(0,2)$. 

\framebox{Case: $p=2$} 
This follows immediately 
from \eqref{eq:as iff} 
-- because $\la_2(s)-\la_2(0)=s^2$, and so,  
\begin{equation}\label{eq:sim K p=2}
	\|\s\|^2\sim K_2\,d^{1/2} \text{ for any $2$-AS shift $\s$.}
\end{equation}
 
\framebox{Case: $p\in(2,\infty)$} By Lemma~\ref{lem:h}, the condition \eqref{eq:as iff} implies 
\begin{equation}\label{eq:S2+Sp asymp}
	V_2+V_p\asymp\sqrt d. 
\end{equation}
Consider first the subcase  
$\no\uu_p>>d^{(p-2)/(4p)}$.   
Let us show that then \eqref{eq:as iff} implies 
$V_2<<V_p$, where $V_p$ is defined in \eqref{eq:S}. Indeed, otherwise w.l.o.g.\ for some varying vector $\s$ in the direction of $\uu$ with $\no\uu_p^p>>d^{(p-2)/4}$ 
one would have $V_p\OO V_2$, and then \eqref{eq:S2+Sp asymp} would imply $V_2\asymp\sqrt d$, that is, $d\,\no\s_2^2\asymp\sqrt d$, or $\no\s_2\asymp d^{-1/4}$, and then $V_p/V_2=\no\s_2^{p-2}\no\uu_p^p>>1$. This contradiction shows that $V_2<<V_p$. Now \eqref{eq:S2+Sp asymp} yields 
\begin{equation*}
	\|\s\|^2=V_2<<V_p\le V_2+V_p\asymp\sqrt d. 
\end{equation*}
Thus, in view of \eqref{eq:sim K p=2}, $a_{p,2,\uu}=\infty$. 
 
Next, consider 
the subcase $\no\uu_p<<d^{(p-2)/(4p)}$.  
Note that \eqref{eq:S2+Sp asymp} implies $V_2\OO d^{1/2}$, that is, $\no\s_2\OO d^{-1/4}$. 
So, $V_p/V_2=\no\s_2^{p-2}\no\uu_p^p<<1$, whence $V_p<<V_2\le V_2+V_p\asymp\sqrt d$, by \eqref{eq:S2+Sp asymp}. 
Next, in view of \eqref{eq:la_p(s)-la_p, s to infty}, 
\begin{equation*}
	\sum\big(\la_p(s_j)-\la_p(0)\big) \ii{|s_j|>\vp} 
	\OO\sum|s_j|^p=V_p<<\sqrt d
\end{equation*}
for any fixed $\vp>0$, so that \eqref{eq:as iff} implies 
\begin{equation}\label{eq:sim C}
	\sum\big(\la_p(s_j)-\la_p(0)\big) \ii{|s_j|\le\vp} 
	\sim K_p\,\sqrt d. 
\end{equation} 
On the other hand, by \eqref{eq:la_p(s)-la_p, s to0}, \eqref{eq:la_p(s)-la_p>}, and \eqref{eq:as iff}, 
\begin{multline*}
	\sum\big(\la_p(s_j)-\la_p(0)\big) \ii{|s_j|\le\vp} 
	=C_1(\vp)\sum s_j^2 \ii{|s_j|\le\vp} \\
	\le C_1(\vp)V_2
	\le\tfrac{C_1(\vp)}{C_1(0)}\,\sum\big(\la_p(s_j)-\la_p(0)\big)
	\sim K_p\,\sqrt d,  
\end{multline*}
for some varying $C_1(\vp)$ such that $C_1(\vp)\to C_1(0):=\tfrac p2\,\la_p(0)$ as $\vp\downarrow0$. 
Comparing this with \eqref{eq:sim C}, one concludes that 
$C_1(0)V_2\sim K_p\,\sqrt d$, that is, \eqref{eq:OOG,p in(-frac12,0)} holds. 
So, by \eqref{eq:sim K p=2}, 
$a_{p,2,\uu}=a_p$ 
for all $p\in(2,\infty)$ -- under the condition $\no\uu_p^p<<d^{(p-2)/4}$. 	

\framebox{Case: $p=\infty$} 
Let $\s$ be an $\infty$-AS shift in the direction of a $\no\cdot$-unit vector $\uu$.   
First here, consider the subcase when $\no\uu_\infty<<d^{1/4}\,\sqrt{\ln d}$. 
Then 
\begin{align}
	d\,f(0)&=-\ln(1-\al)+o(1), \notag
	\\
	\sum f(s_j)&=-\ln(1-\be)+o(1), \label{eq:be}
\end{align}
where 
\begin{equation*}
	f(s):=-\ln\PP(|Z+s|\le c)  
\end{equation*}
for some varying $c$. 
As shown in the proof of Proposition~\ref{prop:AS shift iff} for $p=\infty$, one must have 
$
	c\sim\sqrt{2\ln d}. 
$ 
Introduce now $v_j:=s_j^2$, so that $\sum v_j=\|\s\|^2$ and $f(s_j)=g(v_j)$, where $g(v):=f(\sqrt v)$. At that, the function $g$ is 
convex on $[0,\infty)$, by Lemma~\ref{lem:lla}(iii).  

Now, to obtain a contradiction, assume that the conclusion that $a_{\infty,2,\uu}=0$ does not hold in this subcase. Then, in view of \eqref{eq:sim K p=2}, w.l.o.g.\ $\|\s\|^2\OO d^{1/2}$, which can be rewritten in each of the following two forms: $\no\s^2\OO d^{-1/2}$ and $\sum v_j\OO d^{1/2}$. 
So,  
the condition $\no\uu_\infty<<d^{1/4}\,\sqrt{\ln d}$ implies 
$\|\s\|_\infty=\no\s_\infty=\no\s_2\no\uu_\infty<<\sqrt{\ln d}$. 
Thus, $\|\s\|_\infty\le B$ for some variable $B$ such that $1\le B<<\sqrt{\ln d}$. 
Recalling now \eqref{eq:be} and the fact that the function $g$ is convex and nonnegative, one has 

\begin{multline*}
	-\ln(1-\be)\approx\sum g(v_j) 
	\le\sum \big[g(0)+\tfrac{v_j}{B^2}\,\big(g(B^2)-g(0)\big)\big] \\
	\le d\,g(0)+O\big(d^{1/2}\,g(B^2)\big)\approx-\ln(1-\al)+O\big(d^{1/2}\,g(B^2)\big), 
\end{multline*}
whence $g(B^2)\OOG d^{-1/2}$. 
On the other hand, recalling that $c\sim\sqrt{2\ln d}$ and $B<<\sqrt{\ln d}$, one obtains the sought contradiction: 
\begin{multline*}
	g(B^2)=-\ln\PP(|Z+B|\le c)
	\sim\PP(|Z+B|>c)\\
	=e^{-c^2/2(1+o(1))}=d^{-(1+o(1))}<<d^{-1/2}. 
\end{multline*}

Finally, consider the subcase when $\no\uu_\infty>>d^{1/4}\,\sqrt{\ln d}$. 
To obtain a contradiction, assume that the conclusion that $a_{\infty,2,\uu}=\infty$ does not hold in this subcase. Then, similarly to the previous subcase, in view of \eqref{eq:sim K p=2}, w.l.o.g.\ $\|\s\|^2\OOG\sqrt d$, whence $\|\s\|_\infty>>\sqrt{\ln d}$, so that $\|\s\|_\infty\ge2\sqrt{\ln d}$ eventually. By \eqref{eq:be} and the condition $c\sim\sqrt{2\ln d}$, 
\begin{multline*}
	-\ln(1-\be)\approx\sum f(s_j)\ge f(\|\s\|_\infty)\ge f(2\sqrt{\ln d}) \\
	=-\ln\PP\big(\big|Z-2\sqrt{\ln d}\,\big|\le\sqrt{(2+o(1))\ln d}\,\big)\to\infty,
\end{multline*}
a contradiction. 
\end{proof}

\begin{proof}[Proof of Proposition~\ref{prop:a_p2 range}]\ 

\framebox{(I)} The equality $a_{p,2,\one}=a_p$ for all $p\in[-\infty,\infty]\setminus\{2\}$ follows by Proposition~\ref{prop:sARE}, taking also into account the extended definition of $a_p$ for $p\in[-\infty,\infty]\setminus(-\frac12,\infty)$ as given in Proposition~\ref{prop:a_p}. 
The equality $a_{p,2,\sqrt d\ee_1}=\infty$ for all $p\in(2,\infty]$ also follows by Proposition~\ref{prop:sARE}. 
The equality $a_{p,2,\sqrt d\ee_1}=0$ for all $p\in[-\infty,2)$ follows by Proposition~\ref{prop:sARE} and \eqref{eq:min}, since $\frac{p-1}{2p}<\frac14$ for $p\in(0,2)$. 
The equality $a_{p,2,\uu}=1$ for $p=2$ and all $\uu$ follows by Proposition~\ref{prop:sARE} as well. 
As for the inequalities in part (I) of Proposition~\ref{prop:a_p2 range}, they follow by Proposition~\ref{prop:a_p2 schur2}.  

\framebox{(II)} The proof of part (II) of Proposition~\ref{prop:a_p2 range} will be done depending on a set of values of $p$. 

\framebox{Case: $p\in[-\infty,-\frac12]$} This case follows by Proposition~\ref{prop:sARE}.  

\framebox{Case: $p\in(-\frac12,0)$}
Consider vectors $\s$ of the form 
\begin{equation}\label{eq:sk}
(\underbrace{s,\dots,s}_k,\underbrace{0,\dots,0}_{d-k}) 	
\end{equation}
with $s\in(0,\infty)$ and $k\in\{1,\dots,d\}$. 
To begin, let $s$ be fixed and then let $k$ be varying so that 
\begin{equation}\label{eq:k}
k\sim K_p\,d^{1/2}/f_p(s),	
\end{equation}
whence  
\begin{equation*}
	\|\s\|^2=ks^2\sim K_p\,d^{1/2}\,\frac{s^2}{f_p(s)} 
\end{equation*}
and, by Proposition~\ref{prop:AS shift iff}, the vector $\s$ is a $p$-AS shift. 
Observe that the ratio $\frac{f_p(s)}{s^2}$ is continuous in $s\in(0,\infty)$; also, in view of Lemma~\ref{lem:lla}(ii), this ratio tends to $0$ as $s\to\infty$ and to $\frac{|p|\la_p(0)}2$ as $s\to0$. 
On the other hand, if $t\s$ is a $2$-AS shift in the direction of the vector $\s$, then $kt^2s^2\sim K_2d^{1/2}$, whence $\frac{\|t\s\|^2}{\|\s\|^2}\sim\frac{K_2}{K_p}\frac{f_p(s)}{s^2}$. 
By \eqref{eq:=a_p}, 
it follows that for each $p\in(-\frac12,0)$ the values of $a_{p,2,\cdot}$ fill the interval $(0,a_p)=\big(a_{p,2,\sqrt{d}\ee_1},a_{p,2,\one}\big)$. 

Now -- instead taking a fixed value of $s\in(0,\infty)$ -- let $s$ tend to $0$ slowly enough so that 
there still be a varying integer $k\in[1,d]$ such that \eqref{eq:k} holds. 
Then one obtains the limit $a_{p,2}$ equal to $a_p$. Finally, if one lets $s$ tend to $\infty$, then one has $a_{p,2}=0$.  
Thus, the values of $a_{p,2,\cdot}$ fill the entire interval $[0,a_p]$. 

\framebox{Case: $p\in(0,2)$}
This case is quite similar to the case $p\in(-\frac12,0)$.

\framebox{Case: $p=0$}
This case too is quite similar to the case $p\in(-\frac12,0)$. 
Here, instead of Lemma~\ref{lem:lla}(ii), use Lemma~\ref{lem:ttla}(ii). 

\framebox{Case: $p=2$} This case follows by part (I) of Proposition~\ref{prop:a_p2 range}. 

\framebox{Case: $p\in(2,\infty)$}
This case as well is similar to the case $p\in(-\frac12,0)$, but with the different range of values of $a_{p,2}$; also, in this case $\frac{f_p(s)}{s^2}$ tends to $\infty$, rather than to $0$, as $|s|\to\infty$.

\framebox{Case: $p=\infty$} 
Consider again vectors of the form \eqref{eq:sk}, but now with 
\begin{equation}\label{eq:s,k,infty}
	s=\la\sqrt{\ln d}\quad\text{and}\quad k\sim Ld^{1/2}/s^2,
\end{equation}
where $L$ is an arbitrary positive real constant, while $\la$ takes values in $(0,\sqrt2)$, is possibly varying, but bounded away from $0$ and $\sqrt2$ \big(cf.\ \eqref{eq:c(d,al)}, which implies $c_{d,\al}\sim\sqrt{2\ln d}$\,\big).  
Then $k\asymp d^{1/2}/\ln d$, and so, $k\in[1,d]$ eventually, for large enough $d$. 
Recalling also \eqref{eq:la infty} and letting for brevity $c:=c_{d,\al}$, one has  
\begin{equation}\label{eq:la infty sim}
\begin{aligned}
	\la_{\infty;d}(s)
	&=-\ln\PP(|Z+s|\le c)\sim\PP(|Z+s|>c)\sim\PP(Z>c-s)\sim\frac{\vpi(c-s)}{c-s} \\
	&\sim\frac1{\sqrt{2\pi}(\sqrt2-\la)\sqrt{\ln d}}\,d^{-(\sqrt2-\la)^2/(2+o(1))}=d^{-(\sqrt2-\la)^2/(2+o(1))}.  
\end{aligned}	
\end{equation}
Similarly, $\la_{\infty;d}(0)=d^{-(1+o(1))}$, so that
\begin{align*}
	f_{\infty;d}(s):=\la_{\infty;d}(s)-\la_{\infty;d}(0)=d^{-(\sqrt2-\la)^2/(2+o(1))}, 
\end{align*}
whence, for a fixed value of $\la$, 
\begin{align*}
	\frac{f_{\infty;d}(s)}{s^2}\,d^{1/2}\longrightarrow
	\left\{
	\begin{aligned} 
0 &\text{ if } \la\in(0,\sqrt2-1), \\
\infty &\text{ if } \la\in(\sqrt2-1,\sqrt2). 
	\end{aligned} 
	\right.
\end{align*}
Hence, by the continuity of $f_{\infty;d}(s)$ in $s$, there exists some varying $\la$ (necessarily converging to $\sqrt2-1$) such that for $s$ as in \eqref{eq:s,k,infty} one has 
\begin{equation*}
	\frac{f_{\infty;d}(s)}{s^2}\,d^{1/2}=\frac{K_\infty}L, 
\end{equation*}
so that \eqref{eq:s,k,infty} yields 
\begin{equation*}
	\sum f_{\infty;d}(s_j)=kf_{\infty;d}(s)\sim K_\infty. 
\end{equation*}
So, by Proposition~\ref{prop:AS shift iff}, $\s$ is an $\infty$-AS shift; at that, 
\begin{equation*}
	\|\s\|^2=ks^2\sim Ld^{1/2}. 
\end{equation*}
Comparing this with relation \eqref{eq:sim K p=2} for $2$-AS shifts and recalling that $L$ is an arbitrary positive real constant, one concludes that the range of the values of $a_{\infty,2,\cdot}$ contains the interval $(0,\infty)$. To complete the proof of Proposition~\ref{prop:a_p2 range}, it remains to refer to the last two lines of formula \eqref{eq:are,d->infty} \big(with $a_{\infty,2,\uu}$ in place of $\are_{\infty,2,\uu}$, as in the already proved Propositions~\ref{prop:sARE} and \ref{prop:u range}\big).    
\end{proof}

\section{Proofs of the lemmas}\label{proofs of lemmas}

\begin{proof}[Proof of Lemma~\ref{lem:vpi}]
Suppose that indeed $v>\max_{j=1}^d|v_j|$ and $v\to0$.  
Write
\begin{equation*}
	\sum\vpi(s_j+v_j)=S_1(\vv)+S_2(\vv),
\end{equation*}
where
\begin{align*}
	S_2(\vv)&:=\sum\vpi(s_j+v_j)\ii{|s_j|\ge\sqrt{3/v}} \\
	&\OO\sum\vpi(s_j+v_j)\ii{|s_j+v_j|\ge\sqrt{2/v}}
	\OO d\,e^{-1/v},\\
	S_1(\vv)&:=\sum\vpi(s_j+v_j)\ii{|s_j|<\sqrt{3/v}} \\
	&=\sum\vpi(s_j)e^{-s_jv_j}e^{-v_j^2/2}\ii{|s_j|<\sqrt{3/v}} \\
	&\sim\sum\vpi(s_j)\ii{|s_j|<\sqrt{3/v}} \\
	&=\sum\vpi(s_j)-S_2(\0)
	=\sum\vpi(s_j)+O(d\,e^{-1/v}). 
\end{align*}
Now the lemma follows.  
\end{proof}

\begin{proof}[Proof of Lemma~\ref{lem:lla}] By the symmetry, w.l.o.g.\ $s>0$.  \\
\framebox{(i)} 
Part (i) of the lemma follows because 
$\la_p(s)=\int_0^\infty\PP(|Z-s|^p>c)\,\dd c$ and, as it is easy to see, $\PP(|Z-s|\le c)$ is strictly decreasing in $s$ for each $c\in(0,\infty)$. 

\framebox{(ii)}
Note that 
$s^{-p}\,\la_p(s)=\La_1+\La_2+\La_3$, 
where, over all $s>1$,  
\begin{align*}
	\La_1&:=\E|1+\tfrac Zs|^p\ii{|Z|>2s}
	\OO s^{-p}\E|Z|^p\ii{|Z|>2s}\OO s^{-2}, \\
	\La_2&:=\E|1+\tfrac Zs|^p\ii{2s\ge|Z|>\tfrac s2}
	\le\vpi(\tfrac s2)\int|1+\tfrac zs|^p\ii{2s\ge|z|>\tfrac s2}\dd z\OO s^{-2}, \\
%
	\La_3&:=\E(1+\tfrac Zs)^p\ii{|Z|\le \tfrac s2}
	=\E\Big(1+p\tfrac Zs+O\big(\tfrac{Z^2}{s^2}\big)\Big)\ii{|Z|\le \tfrac s2} \\
	&=\E\big(1+p\tfrac Zs)-\E\big(1+p\tfrac Zs)\ii{|Z|>\tfrac s2}+O\big(\E \tfrac{Z^2}{s^2}\big)
	=1+O(s^{-2}). 
\end{align*}
So, \eqref{eq:la_p(s)-la_p, s to infty} follows. 

Next, 
\begin{equation}\label{eq:la_p(s)}
\la_p(s)=\int_0^\infty x^p\,[\vpi(s+x)+\vpi(s-x)]\,\dd x. 	
\end{equation}
Differentiating here in $s$
, one easily finds that $\la'_p(0)=0$ and $\la''_p(0)=\la_{p+2}(0)-\la_p(0)=p\la_p(0)$, by \eqref{eq:la_p(0)}. 
This implies \eqref{eq:la_p(s)-la_p, s to0}. 


\framebox{(iii)} 
Consider first the case $p\in(-1,\infty)\setminus\{0,2\}$. 
Differentiating in $s$ under the integral in \eqref{eq:la_p(s)} and then integrating by parts, one has
\begin{equation}\label{eq:la'}
\begin{aligned}
	\la'_p(s)&=\int_0^\infty x^p\,[\vpi'(s+x)+\vpi'(s-x)]\,\dd x \\
	&=-p\,\int_0^\infty x^{p-1}\,[\vpi(s+x)-\vpi(s-x)]\,\dd x, 
\end{aligned}
\end{equation}
whence
\begin{align*}
	\la''_p(s)&=-p\,\int_0^\infty x^{p-1}\,[\vpi'(s+x)-\vpi'(s-x)]\,\dd x, \\
	s^2\,\frac{\dd}{\dd s}\frac{\la'_p(s)}s
	&=s\la''_p(s)-\la'_p(s) \\
	&=p\,\int_0^\infty x^{p-1}\,[\vpi(s+x)-\vpi(s-x)-s\vpi'(s+x)+s\vpi'(s-x)]\,\dd x. 
\end{align*}	
Using now the identities $\vpi'(u)=-u\vpi(u)$ and $\vpi(s\pm x)=e^{-s^2/2}\vpi(x)e^{\mp sx}$, and then expanding $e^{\mp sx}$ into powers of $s$, one obtains 
\begin{align}
e^{s^2/2}\frac{s^2}2\frac{\dd}{\dd s}\frac{\la'_p(s)}s
&=
p\,\int_0^\infty x^{p-1}\,\sum _{n=1}^{\infty } \frac{2n s^{2n+1}}{(2n+1)!}\,[x^{2n+1}-(2n+1)x^{2n-1}]\,
\vpi(x)\,\dd x \notag\\ 
&=
p\,\sum _{n=1}^{\infty } \frac{n s^{2n+1}}{(2n+1)!}\,[\la_{p+2n}(0)-(2n+1)\la_{p+2n-2}(0)] \notag\\ 
&=
p(p-2)\,\sum _{n=1}^{\infty } \frac{n s^{2n+1}}{(2n+1)!}\,\la_{p+2n-2}(0), \label{eq:Dla'(s)/s}
\end{align}
since $\la_{p+2n}(0)=(p + 2 n - 1)\la_{p+2n-2}(0)$. 
So, for any $s\in(0,\infty)$, the sign of $\frac{\dd}{\dd s}\frac{\la'_p(s)}s$ is the same as that of $p(p-2)$. 
This implies the stated monotonicity patterns for $\frac{\la'_p(s)}s$, depending on whether $p\in(0,2)$ or not. 
The same monotonicity patterns hold for $\frac{\la_p(s)-\la_p(0)}{s^2}$, 
by the special l'Hospital-type rule for monotonicity -- see e.g.\ \cite[Proposition~4.1]{pin06}. 
To complete the consideration of this case, it remains to note that 
$\frac\partial{\partial u}\la_p(\sqrt u)=\frac{\la_p'(\sqrt u)}{2\sqrt u}$ for $u>0$. 

In the remaining case $p=\infty$, one can deduce part (iii) of Lemma~\ref{lem:lla} from Lemma~\ref{lem:schur2}. 
Alternatively, one can do this directly by calculus, as follows. Recalling \eqref{eq:la infty}, one has $\frac{\la_\infty'(s)}s=\frac{F(s)}{G(s)}$ for $s>0$, where $F(s):=\frac1s\,[\vpi(s-c)-\vpi(s+c)]$, $G(s):=\PP(|Z+s|\le c)$, and $c:=c_{d,\al}$. Next, for $\rho(s):=\frac{F'(s)}{G'(s)}$, one has $\rho'(s)s^3(e^{2cs}-1)^2=H(cs)$, where $H(u):=e^{4 u} (u-2)+4 e^{2 u} \left(u^2+1\right)-u-2$. 
Note also that $H(0)=H'(0)=0$, $H''(u)=8e^{4 u}H_2(u)$, $H_2(u):=e^{-2 u} \left(2 u^2+4 u+3\right)+2 u-3$, $H_2(0)=H_2'(0)=0$, and $H_2''(u)=8e^{-2 u} u^2$. It follows that $H>0$ and hence $\rho'>0$ on $[0,\infty)$. Now it remains to again refer to the special l'Hospital-type rule for monotonicity.

\framebox{(iv)} Note that inequality \eqref{eq:la_p(s)-la_p>} turns into an equality if $p=2$ or $s=0$. 
Next \big(cf.\ \eqref{eq:la'}\big) 
\begin{equation*}
\frac{\la'''_p(s)}{p(p-1)(p-2)}
=\int_0^\infty x^{p-3}\,[\vpi(x-s)-\vpi(x+s)]\,\dd x>0	
\end{equation*}
for all $p>2$ and $s>0$, since $\vpi(x-s)>\vpi(x+s)$ for all $x>0$ and $s>0$. 
Now part (iv) of the lemma follows, because $\la_p'(0)=0$ and $\la_p''(0)=p\la_p(0)$. 

\framebox{(v)} This part follows by part (ii) of the lemma, since $\la_{p,2}(s)=\la_{2p}(s)-\la_p(s)^2$ and $\la_{p,2}(s)$ is continuous and hence bounded in $|s|\le1$. 

\framebox{(vi)} By part (ii) of the lemma, $\la_p(s)
=s^p\big(1+O(s^{-2})\big)$ over $s\in(1,\infty)$. So,  
\begin{equation*}
	\la_{p,4}(s)=\la_{4p}(s)-4\la_p(s)\la_{3p}(s)+6\la_p(s)^2\la_{2p}(s)-3\la_p(s)^4\OO s^{4p-2}
\end{equation*}
over $s>1$. 
By part (v) of the lemma, 
$\la_{p,2}(s)\OO 1+s^{2p-2}$ over all $s>1$. So, 
$\la_{p,3}(s)\le\sqrt{\la_{p,2}(s)\,\la_{p,4}(s)}\OO s^{3p-1}$ over $s>1$. 
It is also clear that $\la_{p,3}(s)\OO1$ for $|s|\le1$. 
\end{proof}

\begin{proof}[Proof of Lemma~\ref{lem:ttmu}]\ 

\framebox{(i)} Part (i) of the lemma follows because  
\begin{equation*}
	\tmu_d(s)=\int_0^\infty\PP(|Z+s|^{-1}\wedge d>x)\,\dd x
	=\int_0^d\PP(|Z+s|<1/x)\,\dd x. 
\end{equation*}

\framebox{(ii)} Note that 
\begin{equation}\label{eq:tmu=}
	\tmu_d(s)=\int_0^\infty\frac{\dd u}{u\vee\si}\,[\vpi(s-u)+\vpi(s+u)],
\end{equation}
where
\begin{equation*}
	\si:=\tfrac1d.
\end{equation*}
So, for $\eta:=\frac1{\ln d}$ and large enough $d$, 
\begin{equation*}
	\tfrac12\tmu_d(0)=\int_0^\infty\frac{\dd u}{u\vee\si}\,\vpi(u)
	=\int_0^\si+\int_\si^\eta+\int_\eta^1+\int_1^\infty, 
\end{equation*}
\begin{align*}
	\int_0^\si&=\frac1\si \int_0^\si\dd u\,\vpi(u)\sim\vpi(0)<<\ln d,\\
	\int_\si^\eta&=\int_\si^\eta\frac{\dd u}u\,\vpi(u)
	\sim(\ln\eta-\ln\si)\vpi(0)\sim\vpi(0)\ln d,\\
	\int_\eta^1&\OO\int_\eta^1\frac{\dd u}u<<\ln d,\\
	\int_1^\infty&\OO 1<<\ln d.
\end{align*}
So,
\begin{equation}\label{eq:tmu(0)}
	\tmu_d(0)\sim2\vpi(0)\ln d. 
\end{equation}
In view of part (i) of the lemma, this completes the proof of part (ii). 

\framebox{(iii)} 
Here, consider first the case $s\in(0,\frac12]$. Then $\De(s,u):=2\vpi(u)-\vpi(u-s)-\vpi(u+s)\asymp -s^2\vpi''(\tu)\asymp s^2(1-\tu^2)\vpi(\tu)$ for all $u>0$ and some $\tu\in[u-s,u+s]$. 
In particular, $\De(s,u)\asymp s^2$ for all $u\in[0,\frac16]$, $|\De(s,u)|\OO s^2$ for all $u>0$, and $|\De(s,u)|\OO s^2(u^2+1)\vpi(u-s)\OO s^2\vpi(u/3)$ for all $u>1$. 
So,  
\begin{equation*}
	\tmu_d(0)-\tmu_d(s)=\int_0^\infty\frac{\dd u}{u\vee\si}\,\De(s,u)
	=\int_0^\si+\int_\si^\eta+\int_\eta^1+\int_1^\infty, 
\end{equation*}
\begin{align*}
	\Big|\int_0^\si\Big|&=\Big|\frac1\si \int_0^\si\dd u\,\De(s,u)\Big|\OO s^2<<s^2\ln d,\\
	\int_\si^\eta&=\int_\si^\eta\frac{\dd u}u\,\De(s,u)
	\asymp s^2(\ln\eta-\ln\si)\sim s^2\ln d,\\
	\Big|\int_\eta^1\Big|&\OO\int_\eta^1\frac{\dd u}u\, s^2<<s^2\ln d,\\
	\Big|\int_1^\infty\Big|&\OO \int_1^\infty s^2\vpi(u/3)\,\dd u\OO s^2<<s^2\ln d.
\end{align*}
Thus,
\begin{equation}\label{eq:s<1/2}
	\tmu_d(0)-\tmu_d(s)\asymp s^2\ln d\quad\text{over}\, s\in(0,\tfrac12]. 
\end{equation}

Similarly but simpler, for all $s\in(3,\infty)$ and large $d$ 
\begin{equation*}
	\tmu_d(s)
	=\Big(\int_0^\si+\int_\si^1+\int_1^\infty\Big)\frac{\dd u}{u\vee\si}\,[\vpi(s-u)+\vpi(s+u)], 
\end{equation*}
\begin{align*}
	\int_0^\si&\OO\frac1\si \int_0^\si\dd u=1<<\ln d\asymp\tmu_d(0),\\
	\int_\si^1&\le\int_\si^1\frac{\dd u}u\,2\vpi(s-1)
	\le2\vpi(2)\ln d<\tfrac13\tmu_d(0), \\
	\int_1^\infty&\le\int_1^\infty[\vpi(s-u)+\vpi(s+u)]\,\dd u\le2<<\tmu_d(0);
\end{align*}
here, to estimate $\int_\si^1$ we used \eqref{eq:tmu(0)}. 
So, for all $s\in(3,\infty)$ and large enough $d$, one has $\tmu_d(s)<\frac12\tmu_d(0)$ and hence $\tmu_d(0)-\tmu_d(s)\asymp\tmu_d(0)\asymp\ln d$. 
It remains to recall \eqref{eq:s<1/2} and the fact that $\tmu_d(s)$ is decreasing in $|s|$.  

\framebox{(iv)} 
W.l.o.g.\ $s\ge0$. 
By \eqref{eq:tmu=},  
 \begin{equation*}
	-s\,f'(s)=s\,\tmu'_d(s)=\int_0^\infty\frac{s\,\dd u}{u\vee\si}\,[\vpi'(s-u)+\vpi'(s+u)]
	=\int_0^\si+\int_\si^1+\int_1^\infty.  
\end{equation*}
Observe that $\vpi'(s-u)+\vpi'(s+u)=\vpi'(u+s)-\vpi'(u-s)\OO s|\vpi''(\tu)|\OO s\,\vpi(\frac{u-s}2)\OO s$ for all $u>0$ and $s\ge0$, and for some $\tu\in[u-s,u+s]$. So,
\begin{align*}
	\Big|\int_0^\si\Big|&\OO\frac1\si\int_0^\si s\,\dd u\, s=s^2<<s^2\ln d, \\
	\Big|\int_\si^1\Big|&\OO\int_\si^1\frac{s\,\dd u}u s\OO s^2\ln d, \\
	\Big|\int_1^\infty\Big|&\OO\int_1^\infty s^2\,\dd u\,\vpi(\tfrac{u-s}2)\OO s^2<<s^2\ln d, \\
\end{align*}
whence 
\begin{equation}\label{eq:f'=}
	|sf'(s)|\OO s^2\ln d\quad\text{over all $s\ge0$.} 
\end{equation}
Also, for $s>1$ write 
\begin{align*}
	-s\,f'(s)&=\int_0^\infty\frac{s\,\dd u}{u\vee\si}\,[\vpi'(s-u)+\vpi'(s+u)]
	=\int_0^\si+\int_\si^{s/2}+\int_{s/2}^\infty, \\
	\int_0^\si&\OO\frac1\si\int_0^\si \dd u\, s^2\vpi(s/4)\OO1<<\ln d, \\
	\int_\si^{s/2}&\OO\int_\si^{s/2}\frac{s\,\dd u}u s\,\vpi(s/4)\OO\ln d, \\
	\int_{s/2}^\infty&=\int_{s/2}^\infty \frac{s\,\dd u}u\,[\vpi'(u+s)-\vpi'(u-s)] \\
	 &=\frac su\,[\vpi(u+s)-\vpi(u-s)]\Big|_{s/2}^\infty 
	+\int_{s/2}^\infty \frac{s\,\dd u}{u^2}\,[\vpi(u+s)-\vpi(u-s)] \\
	&\OO1<<\ln d, 
\end{align*}
so that $|sf'(s)|\OO\ln d=(s^2\wedge1)\ln d$ for $s>1$. Recalling now \eqref{eq:f'=} and \eqref{eq:tmu(0)-tmu(s)}, one has $|sf'(s)|\OO(s^2\wedge1)\ln d\asymp f(s)$ over all $s\in\R$. 
\end{proof}

\begin{proof}[Proof of Lemma~\ref{lem:ttla}] \ 

\framebox{(i)} Part (i) of the lemma follows because 
$\tla(s)=\int_0^\infty\PP(\ln|Z+s|>c)\,\dd c=\int_0^\infty\PP(|Z+s|>e^c)\,\dd c$.  

\framebox{(ii)} Note that 
\begin{equation}\label{eq:tla(s)}
\tla(s)=\int_0^\infty (\ln x)\,[\vpi(s-x)+\vpi(s+x)]\,\dd x.	
\end{equation}
Differentiating here in $s$, one easily finds that $\tla'(0)=0$. Also, using integration by parts, one has  $\tla''(0)=2\int_0^\infty\ln x\,\vpi''(x)\,\dd x=-2\int_0^\infty\frac1x\,\vpi'(x)\,\dd x=1$. 
Now \eqref{eq:tla(s)-tla, s to0} follows. 

Fix now any $m\in\N$ and let $s$ vary arbitrarily in $[e,\infty)$. 
Since the density of $|Z+s|$ is bounded uniformly in $s\in\R$, 
\begin{equation*}
	\E\big|\ln|Z+s|\big|^m\ii{|Z+s|\le1}\OO\int_0^1|\ln x|^m\,\dd x\OO1. 
\end{equation*}
Also,
\begin{gather*}
	\E\big|\ln|Z+s|\big|^m\ii{|Z+s|>1,|Z|>\tfrac s2}
	\OO\E(|Z|+|s|)\ii{|Z|>s/2}\OO1; \\
\frac1{\ln^m s}\,\E\ln^m|Z+s|\ii{|Z|\le\tfrac s2}
=\E\Big(1+\frac{\ln(1+Z/s)}{\ln s}\Big)^m\ii{|Z|\le\tfrac s2} \\
=\E\Big[1+m\,\frac{Z/s+O(Z^2/s^2)}{\ln s}+O\Big(\frac{Z^2}{s^2\ln^2s}\Big)\Big]\ii{|Z|\le\tfrac s2} \\
=\E\Big[\dots\Big]-\E\Big[\dots\Big]\ii{|Z|>\tfrac s2} 
=1+O\Big(\frac1{s^2\ln s}\Big)
=1+O\Big(\frac1{\ln^m s}\Big).
\end{gather*}
Now \eqref{eq:tla(s), s to infty} follows, which completes the proof of part (ii). 

\framebox{(iii)} The derivatives $\tla^{(i)}(s)$ of $\tla(s)$ are related to those of $\la_p(s)$; namely, 
$\tla^{(i)}(s)=\frac{\partial}{\partial p}\la_p^{(i)}(s)\big|_{p=0}$ for $i=0,1,\dots$; this follows because, in view of \eqref{eq:la_p(s)} and \eqref{eq:tla(s)}, 
$\la_p^{(i)}(s)=\int_0^\infty x^p\,[\vpi^{(i)}(s-x)+\vpi^{(i)}(s+x)]\,\dd x$ 
and $\tla^{(i)}(s)=\int_0^\infty \ln x\,[\vpi^{(i)}(s-x)+\vpi^{(i)}(s+x)]\,\dd x$. 
Therefore, by \eqref{eq:Dla'(s)/s}, 
\begin{equation*}
\frac{\dd}{\dd s}\frac{\tla'(s)}s
=\frac{\partial}{\partial p}\Big[\frac{\dd}{\dd s}\frac{\la'_p(s)}s\Big]\Big|_{p=0}
=
	-\frac4{s^2 e^{s^2/2}}\, \sum _{n=1}^{\infty } \frac{n s^{2 n+1} \la_{2n-2}(0)}{(2n+1)!}<0
\end{equation*}
for $s>0$, which implies the stated monotonicity pattern for $\frac{\tla'(s)}s$. 
The rest of the proof of part (iii) of the lemma is quite similar to that of part (iii) of Lemma~\ref{lem:lla}. 

\framebox{(iv)}  
By \eqref{eq:tla(s)-tla, s to0}, w.l.o.g.\ $s\in[e,\infty)$. Then, by \eqref{eq:tla(s), s to infty},  
\begin{align*}
	\tla_2(s)&=\E\ln^2|Z+s|-\E^2\ln|Z+s| \\
	&=\ln^2 s\,[1+O(\ln^{-2}s)]-\ln^2 s\,[1+O(\ln^{-1}s)]\OO\ln s=g_0(s). 
\end{align*}

\framebox{(v)} 
Using the change of variables $u:=x^2/2$ and well-known identities for the Gamma function (see e.g.\ \cite[(1.1.22), (1.2.14), and (1.2.9)]{andrews}, one has 
\begin{equation}\label{eq:Ga,pi}
\begin{aligned}
	\tla_2(0)=\Var\ln|Z|&=\Var\ln|\tfrac Z{\sqrt2}| \\
	&=2\int_0^\infty\ln^2\tfrac x{\sqrt2}\,\vpi(x)\,\dd x
	-4\Big(\int_0^\infty\ln\tfrac x{\sqrt2}\,\vpi(x)\,\dd x\Big)^2 \\
	&=\tfrac1{4\pi}\,\big(\Ga''(\tfrac12)\Ga(\tfrac12)-\Ga'(\tfrac12)^2\big) \\
	&=\tfrac14\,\tfrac{\dd^2}{\dd z^2}\ln\Ga(z)|_{z=\frac12} \\
	&=\frac14\,\sum_{k=0}^\infty\frac1{(\frac12+k)^2}=\frac{\pi^2}8, 
\end{aligned}	
\end{equation}
so that part (v) of the lemma is verified as well. 

\framebox{(vi)} 
The proof of this part of the lemma is similar to that of part (vi) of Lemma~\ref{lem:lla}, using now parts (ii) and (iv) of Lemma~\ref{lem:ttla} instead of parts (ii) and (v) of Lemma~\ref{lem:lla}. 

\framebox{(vii)} 
By the symmetry, w.l.o.g.\ $s>0$. 
By \eqref{eq:tla(s)} and in view of the estimate $|\vpi'(u)|\OO\vpi(u/2)$ over all $u\in\R$, for $s\to\infty$ one has 
\begin{align*}
\tla'(s)&=\int_0^\infty (\ln x)\,[\vpi'(s-x)+\vpi'(s+x)]\,\dd x	
=\int_0^{s/2}+\int_{s/2}^{3s/2}+\int_{3s/2}^\infty, \\
\int_0^{s/2}&\OO\int_0^{s/2}\ln x\,\dd x\,\vpi(s/4)<<s^{-1},\\
\int_{3s/2}^\infty&\OO\int_{3s/2}^\infty (\ln x)\,\vpi(\tfrac{x-s}2)\,\dd x\,
\le\int_{3s/2}^\infty (\ln x)\,\vpi(\tfrac x6)\,\dd x\,<<s^{-1},\\
\int_{s/2}^{3s/2}&=
(\ln x)[\vpi(s+x)-\vpi(s-x)]\Big|_{s/2}^{3s/2}
-\int_{s/2}^{3s/2}x^{-1}[\vpi(s+x)-\vpi(s-x)] \,\dd x\\
&=o(s^{-1})+O\Big(s^{-1}\int_{s/2}^{3s/2}[\vpi(s+x)-\vpi(s-x)] \,\dd x\Big)
=O(s^{-1}). 
\end{align*}
This completes the proof of part (vii) and thus the entire proof of Lemma~\ref{lem:ttla}. 
\end{proof}

\begin{proof}[Proof of Lemma~\ref{lem:h}] 
The cases $p=-1$, $p\in(-1,\infty)\setminus\{0\}$, and $p=0$ 
follow immediately by Lemma~\ref{lem:ttmu}(iii), Lemma~\ref{lem:lla}(ii,i), and Lemma~\ref{lem:ttla}(ii,i), respectively.  
\end{proof}

\begin{proof}[Proof of Lemma~\ref{lem:s,v}]
Suppose that the lemma is false. That is, suppose that $\s$, $\vv$, $S$ satisfy the conditions of the lemma, while the conclusion \eqref{eq:s,v,S} is false. Then w.l.o.g.\ $\sum v_j^2>>S\OOG\sum s_j^2$, while $\vv=t\s$ for some varying $t\in(0,\infty)$. At that, necessarily $t=\|\vv\|/\|\s\|\to\infty$. So, w.l.o.g.\ $t\ge2$. 

Take any $\si\in(0,\infty)$. Then there is some $\g(\si)\in(0,1)$ such that $\frac{f_p(s)}{f_p(ts)}\le \frac{f_p(s)}{f_p(2s)}\le1-\g(\si)$ for all $s\in(0,\si]$ \big(since $f_p(s)$ is continuous and strictly increasing in $s\in(0,\infty)$ and hence the ratio $\frac{f_p(s)}{f_p(2s)}$ is continuous and strictly less than $1$ for all $s\in(0,\infty)$, while, by Lemma~\ref{lem:lla}(ii),  
this ratio tends to $\frac1{2^2}<1$ as $s\downarrow0$.\big) 
So,
\begin{align*}
	\tfrac1{1-\g(\si)}\sum f_p(s_j)\ii{|s_j|\le\si}&\le A:=\sum f_p(v_j)\ii{|s_j|\le\si}, \\
	\sum f_p(s_j)\ii{|s_j|>\si}&\le B:=\sum f_p(v_j)\ii{|s_j|>\si}, 
\end{align*}
since $v_j=ts_j$, $t\ge2\ge1$, and $f_p(s)$ is increasing in $|s|$. 
Also, $A+B\sim S$. 
So, if $A\ge S/2$, then 
\begin{equation*}
	S\sim\sum f_p(s_j)\le(1-\g(\si))A+B
	\le A+B-\tfrac{\g(\si)}2\,S\sim\big(1-\tfrac{\g(\si)}2\big)\,S,
\end{equation*}
a contradiction. 
Otherwise, $A<S/2$, whence 
\begin{equation*}
\frac{f_p(\infty-)}{\si^2}\,\sum s_j^2\ge f_p(\infty-)\sum\ii{|s_j|>\si}\ge B\ge\big(\tfrac12+o(1)\big)\,S,
\end{equation*}
which contradicts the assumption $S\OOG\sum s_j^2$, because $f_p(\infty-)<\infty$ while  
$\si\in(0,\infty)$ was chosen arbitrarily. 
\end{proof}

\begin{proof}[Proof of Lemma~\ref{lem:s,v,ge0}]
Suppose that the lemma is false, so that $\s$, $\vv$, $S$ satisfy the conditions of the lemma, while the conclusion \eqref{eq:s,v,S,ge0} is false. Then w.l.o.g.\ $\sum v_j^2>>S\OOG\sum s_j^2$, while $\vv=t\s$ for some varying $t\in(0,\infty)$. At that, necessarily $t=\|\vv\|/\|\s\|\to\infty$. 

Note that 
\begin{equation*}
	g_p(s)=\left\{
	\begin{aligned}
		\tfrac{s^2}{e^2}\ii{|s|\le e}+\ln|s|\ii{|s|>e}&\text{ if }p=0, \\
		s^2\ii{|s|\le1}+|s|^p\ii{|s|>1}&\text{ if }p\in(0,2),  
	\end{aligned}
	\right.
\end{equation*}
$g_p(s)\asymp s^2$ for $|s|\OO1$, 
$g_p(s)\to\infty$ and $\frac{g_p(s)}{s^2}\to0$ as $|s|\to\infty$, and $\frac{g_p(s)}{s^2}$ is non-increasing in $|s|$. 
It follows that 
\begin{multline*}
	\sum g_p(s_j)\ii{|s_j|>\si}\le\tfrac{g_p(\si)}{\si^2}\,\sum s_j^2\ii{|s_j|>\si} \\
	\le\tfrac{g_p(\si)}{\si^2}\,\sum s_j^2 
\OO\tfrac{g_p(\si)}{\si^2}\,S 
\asymp\tfrac{g_p(\si)}{\si^2}\,\sum g_p(s_j), 
\end{multline*}
whence 
$\sum g_p(s_j)\ii{|s_j|>\si}\le\tfrac12\sum g_p(s_j)$ 
for large enough $\si\in(0,\infty)$. 
Fixing any such $\si$ and recalling that $g_p(s)\asymp s^2$ for $|s|\OO1$, one has 
\begin{equation*}
	\sum s_j^2\ii{|s_j|\le\si}\asymp\sum g_p(s_j)\ii{|s_j|\le\si}\ge\tfrac12\sum g_p(s_j)\asymp S
\end{equation*}
and 
\begin{align*}
	S&\asymp\sum g_p(ts_j)\ge\sum g_p(ts_j)\ii{|s_j|\le\si}
	\ge\tfrac{g_p(t\si)}{\si^2}\,\sum s_j^2\ii{|s_j|\le\si}
	\OOG\tfrac{g_p(t\si)}{\si^2}\,S \\
	&>>S,
\end{align*}
since $t\to\infty$ and $\si>0$ is fixed. 
This contradiction completes the proof.    
\end{proof}



\bibliographystyle{acm}
\bibliography{C:/Users/Iosif/Documents/mtu_home01-30-10/bib_files/citations}

\end{document}